%% file: main.tex
\documentclass{uscthesis}

\usepackage{rorabaugh}

\usepackage[style=uscauthoryear]{biblatex}
\bibliography{refs}

\title{Toward the Combinatorial Limit Theory of Free Words} 

\author{Danny}{Rorabaugh}    

\date{2015}                      

\otherdegrees{
Bachelor of Science\\
Seattle Pacific University, 2010\\ 
[\baselineskip] 
Bachelor of Arts\\
Seattle Pacific University, 2010\\ 
}                             

\degree{Doctor of Philosophy}     
\field{Mathematics}              
\college{College of Arts and Sciences}  

\advisor {Dr.}{Joshua N. Cooper}{Major Professor}  
\readerb{Dr.}{\'Eva Czabarka}{Committee Member}          
\readera{Dr.}{George F. McNulty}{Committee Member}     
\readerc{Dr.}{Michael Filaseta}{Committee Member}
\readerd{Dr.}{Duncan Buell}{Committee Member}

\dean{Lacy Ford, Jr.}   

\copyrightpage       
\abstract{abstract}  

\input thanks 

\makeLoT               

\makeLoF 

\allowdisplaybreaks

\begin{document}

\include{Intro}

\include{ZiminAvoidance}

\include{Densities}

\include{DensityDichotomy}

\include{Asymp}

\include{Further}
\printbibliography

\appendix

\include{AppendixZ}

\include{AppendixZ2Z3}

\include{AppendixF}

\include{AppendixT}

	\include{AppendixT21}
	\include{AppendixT22}
	\include{AppendixT23}
	\include{AppendixT31}

\include{AppendixN}

\end{document}

%% file: thanks.tex
I am most grateful for the love and support of my family, friends, and professors throughout the process of becoming a Doctor of Philosophy.

This research was supported in part by High Performance Computing
resources and expertise provided by the Research Cyberinfrastructure (RCI)
group at the University of South Carolina.

%% file: Intro.tex
\chapter{Background and Introduction} \label{INTRO}




\section{Discrete Structures and Combinatorics}

Any mathematical structure that is enumerable or noncontinuous can be referred to as discrete.
Discrete mathematicians, therefore, usually study such things as sets, integers, groups, graphs, logical statements, or geometric objects.
However, even uncountable or continuous objects such as topological spaces, contours, differential equations, or dynamical systems can be discretized or otherwise studied by their discrete properties.

Perhaps the structure most commonly identified with discrete mathematics is a graph. 
A graph $G$ consists of a set $V(G)$ of points, called vertices or nodes, and a set $E(G)$ of unordered pairs of points, called edges.
It is often represented visually, with points or circles as vertices, and line segments that connect the points as edges.



Though the term ``discrete mathematics'' can technically encompass any study of discrete objects, including much of algebra, number theory, logic, and theoretical computer science, it is more commonly used as a synonym for combinatorics. 

Combinatorialists are, generally speaking, interested in counting.
Of the nature of combinatorics, \textcite{Cam-94} says: ``Its tentacles stretch into virtually all corners of mathematics.'' 
Though some mathematical structures are inherently more discrete, and thus more susceptible to combinatorial analysis, any structure can be the subject of combinatorial investigation.
Two particular combinatorial perspectives, Ramsey theory and extremal theory, are especially important for the present work. 

\subsection{Ramsey Theory}
\textcite{R-29} proved that, for any fixed $r,n,\mu \in \Z^+$, every sufficiently large set $\Gamma$ with its $r$-subsets partitioned into $\mu$ classes is guaranteed to have an $n$-element subset $\Delta_n \subseteq \Gamma$ such that all the $r$-subsets of $\Delta_n$ are in the same class.
This was the advent of a major branch of combinatorics known as Ramsey theory.
If a given property holds for every sufficiently ``large'' structure within a class of structures, then a combinatorialist might investigate how large a structure must be to guarantee the property.

\subsection{Extremal Theory}
In combinatorial optimization, we look at structures subject to given constraints and ask: ``What are the optimal values obtained by such-and-such function within these constraints?'' or ``Which structures satisfy the constraints and optimize the function?''
That is, we might try to find extremal values and a characterization of the structures which obtain the extremal values. 
A foundational example of this school of thought comes from \textcite{T-41}, who classified graphs on $n$ vertices with the highest possible number of edges but with no set of $k+1$ vertices for which all possible edges are present. 

\section{Words}

Our present interest is in words--not the linguistic units with lexical value, but rather strings of symbols or letters.
We are interested in words as abstract discrete structures.
There are many different ways discrete mathematicians view words: as sequences, permutations, elements of a monoid, etc. 
Within each perspective there is a distinct set of axioms for how words are built and how they interact.
Consequently, the theory and applications that arise for each perspective are drastically different.
One ubiquitous approach for studying discrete structures is to consider the substructures.
In the case of sequences or permutations, the ``subword'' generally consists of a subsequence of not-necessarily consecutive terms.

Some number theorists and combinatorialists study sequences of numbers $\big[$for example: $1, 1, 2, 3, 5, 8, 13, 21, 34, \ldots\big]$. 
A numeric list might be generated by a recursive formula $\big[f(1) = f(2) = 1, f(n+2) = f(n+1) + f(n)\big]$, an explicit formula $\big[f(n) = \frac{1}{2^n\sqrt{5}}\left(\left(1+\sqrt{5}\right)^n - \left(1 - \sqrt{5}\right)^n\right)\big]$, or enumeration of a particular class of structures $\big[f(n)$ is the number of way to tile a $2 \times (n-1)$ rectangle with $2 \times 1$ dominoes$\big]$.
See the Online Encyclopedia of Integer Sequences \parencite{OEIS} for many such sequences $\big[$including oeis.org/A000045, the Fibonacci sequence$\big]$.
There are natural questions one might ask about such a sequence: Is it periodic? Is it bounded? Does it converge? What is the asymptotic rate of growth?

The elements of a sequence need not be numbers to be of mathematical interest. 
In a sequence of colors, for example, one can identify the frequency with which yellow appears, or the probability that red is followed by blue, or whether there exists a subsequence of $k$ black entries that are equally spaced in the original sequence. 
One seminal result on nonnumeric sequences was by \textcite{vdW-27}, who showed that, for any positive integers $k$ and $r$, every sufficiently long sequence containing at most $r$ distinct colors contains a monochromatic $k$-term arithmetic progression (i.e., a length-$k$ subsequence of a single color and equally spaced terms). 

A large body of work exists for permutations, which are sequences of elements of a linearly ordered set (generally with no element occurring twice).
The substructures for permutations are subsequences, which are usually only identified in terms of their permutation pattern $\sigma$.
For example, the permutation 1342 encounters the pattern $\sigma = 1$ (via subsequences 1, 3, 4, and 2),  $\sigma = 12$ (via 13, 14, 12, and 34),  $\sigma = 21$ (32 and 42),  $\sigma = 123$ (134),  $\sigma = 132$ (132 and 142),  $\sigma = 231$ (342), and  $\sigma = 1342$ (1342).
Perhaps the first work on permutation patterns was that of \textcite{M-15}, who showed that 132-avoiding permutations are enumerated by the Catalan numbers (see oeis.org/A000108).
For more on permutation patterns, see \textcite{K-11}.

For our present study of words, we consider only ``subwords'' that consist of consecutive letters. 
This is the perspective that holds for elements of a free monoid. 
A monoid is an algebraic structure consisting of a set, an associative binary operation on the set, and an identity element. 
A free monoid is defined over some generating set of elements, which we view as an alphabet of letters. 
Its binary operation is simply concatenation, its elements--called free words--are all finite strings of letters, and its identity element is the empty word (generally denoted with $\varepsilon$ or $\lambda$). 
Often, the operation of a monoid is called multiplication, so it is fitting that a ``subword'' of a free word is called a ``factor.'' 
For example, in the free monoid over alphabet $\{a,b,c,d,r\}$, the word $cadabra$ is a factor of $abracadabra$ because $abracadabra$ is the product of $abra$ and $cadabra$. 

If there is an inverse element $s^{-1}$ for every element $s$ in the generating set, we are dealing with a free group. 
Then any word with $ss^{-1}$ or $s^{-1}s$ as a factor is equivalent to the word obtained by removal of said factor. 
For example, $tee^{-1}hee^{-1}e$ is equivalent to reduced word $the$. 
Within what came to be known as combinatorial group theory, \textcite{D-11} first proposed the Word Problem for Groups: Given two words formed from the set of generators of a group, determine whether the words represent the same group element?


\section{Combinatorial Limit Theory}

In an era of massive technological and computational advances, we have large systems for transportation, communication, education, and commerce (to name a few examples). 
We also possess massive quantities of information in every part of life. 
Therefore, in many applications of discrete mathematics, the useful theory is that which is relevant to arbitrarily large discrete structures. 
For example, graphs can be used to model a computer network, with each vertex representing a device and each edge a data connection between devices. 
The most well-known computer network, the Internet, consists of billions of devices with constantly changing connections; one cannot simply create a database of all billion-vertex graphs and their properties. 

We use the term ``combinatorial limit theory'' in general reference to combinatorial methods which help answer the following question: What happens to discrete structures as they grow large? 
Many classical questions from combinatorics fall naturally into this field of study. 
One incredibly productive approach to handling large discrete structures is the probabilistic method, the origin of which is generally credited to Paul Erd\H{o}s. 
See \textcite{AS-08} for standard probabilistic tools used in combinatorics. 
Many asymptotic results from such methods, which may be wildly inaccurate for small values, become increasingly more accurate as the relevant structures grow. 


In the combinatorial limit theory of graphs, major recent developments include the flag algebras of \textcite{R-07} and the graph limits of Borgs, Chayes, Freedman, Lov\'asz, Schrijver, S\'os, Szegedy, Vesztergombi, etc. (see \cite{L-12}).
Given the fundamental reliance of these methods on graph homomorphisms and graph densities, we strive to apply the same ideas to words. 
We discuss graph limits in more detail when describing future research directions in Section~\ref{future:WordLimits}. 

\section{Combinatorics of Free Words} \label{words}

We are henceforth focused on free words, which we will simply call words.
For a summary of notation used throughout this text, see Appendix~\ref{NOTATION}.

\begin{defn} \label{defn:word}
	For a fixed set $\Sigma$, called an \emph{alphabet},  denote with $\Sigma^*$ the set of all finite words formed by concatenation of elements of $\Sigma$, called \emph{letters}.
	Words in $\Sigma^*$ are called \emph{$\Sigma$-words}.
	The set of length-$n$ $\Sigma$-words is denoted with $\Sigma^n$.
	The \emph{empty word}, $\varepsilon$, consisting of zero letters, is a $\Sigma$-word for any alphabet $\Sigma$.
\end{defn}

The set $\Sigma^*$, together with the associative binary operation of concatenation and the identity element $\varepsilon$, forms a free monoid.
We denote concatenation with juxtaposition.
Generally we use natural numbers or minuscule Roman letters as letters and majuscule Roman letters (especially $T,U,V,W,X,Y,$ and $Z$) to name words.
Majuscule Greek letters (especially $\Gamma$ and $\Sigma$) name alphabets, though for a standard $q$-letter alphabet, we frequently use the set $[q] = \{1, 2, \ldots, q \}$.

\begin{ex}
	Alphabet $[3]$ consists of letters 1, 2, and 3. 
	The set of $[3]$-words is \[\{1,2,3\}^* = \{\varepsilon, 1, 2, 3, 11, 12, 13, 21, 22, 23, 31, 32, 33, 111, 112, 113, 121, \ldots \}.\]
\end{ex}

\begin{defn} \label{defn:letter}
	A word $W$ is formed from the concatenation of finitely many letters.
	If letter $x$ is one of the letters concatenated to form $W$, we say $x$ \emph{occurs in} $W$, or $x \in W$.
	For natural number $n \in \N$, an $n$-fold concatenation of word $W$ is denoted $W^n$.
	The \emph{length} of word $W$, denoted $|W|$, is the number of letters in $W$, counting multiplicity.
	$\L(W)$, the \emph{alphabet generated by} $W$, is the set of all letters that occur in $W$.
	For $q \in \N$, word $W$ is \emph{$q$-ary} provided $|L(W)| \leq q$.
	We use $||W||$ to denote the number of letter recurrences in $W$, so $||W|| = |W| - |L(W)|$.
\end{defn}

\begin{ex}
	Let $W = bananas$.
	Then $a,b \in W$, but $c \not\in W$. 
	Also $|W| = 7$, $L(W) = \{a,b,n,s\}$, and $||W|| = 3$.

	For the empty word, we have $|\varepsilon| = 0$, $L(\varepsilon) = \emptyset$, and $||\varepsilon|| = 0$.
\end{ex} 

\begin{defn} \label{defn:factor}
	Word $W$ has $\binom{|W|+1}{2}$ (nonempty) \emph{substrings}, each defined by an integer pair $(i,j)$ with $0\leq i < j \leq |W|$.
	Denote with $W[i,j]$ the word in the $(i,j)$-substring, consisting of $j-i$ consecutive letters of $W$, beginning with the $(i+1)$-th.

	$V$ is a \emph{factor} of $W$, denoted $V \leq W$, provided $V = W[i,j]$ for some integers $i$ and $j$ with $0\leq i < j \leq |W|$; equivalently, $W = SVT$ for some (possibly empty) words $S$ and $T$.
\end{defn}

\begin{ex}
	$nana \leq nana \leq bananas$, with $nana = nana[0,4] = bananas[2,6]$.
\end{ex}

\section{Word Avoidability} \label{avoid}

\begin{defn}
	For alphabets $\Gamma$ and $\Sigma$, every (monoid) homomorphism $\phi: \Gamma^* \rightarrow \Sigma^*$ is uniquely defined by a function $\phi:\Gamma \rightarrow \Sigma^*$. 
	We call a homomorphism \emph{nonerasing} provided it is defined by $\phi:\Gamma \rightarrow \Sigma^* \setminus \{\varepsilon\}$; that is, no letter maps to $\varepsilon$.
\end{defn}

\begin{ex}
	Consider the homomorphism $\phi: \{b,n,s,u\}^* \rightarrow \{m,n,o,p,r,v\}^*$ defined by Table~\ref{table:exFunction}. 
	Then $\phi(sun) = moon$ and $\phi(bus) = vroom$.

	\begin{table}[ht]
	\centering
	\begin{threeparttable}
		\caption{Example nonerasing function.} \label{table:exFunction}
		
		\begin{tabular}{c | c | c | c | c} 
			$x$ & $b$ & $n$ & $s$ & $u$\\ \hline
			$\phi(x)$ & $vr$ & $n$ & $m$ & $oo$
		 \end{tabular}
	\end{threeparttable}
	\end{table}

\end{ex}

\begin{defn} \label{defn:instance}
	$U$ is an \emph{instance of $V$}, or a \emph{$V$-instance}, provided $U = \phi(V)$ for some nonerasing homomorphism $\phi$; equivalently,
	\begin{itemize}
		\item $V = x_0x_1 \cdots x_{m-1}$ where each $x_i$ is a letter;
		\item $U = A_0A_1 \cdots A_{m-1}$ with each word  $A_i \neq \varepsilon$ and $A_i = A_j$ whenever $x_i=x_j$.
	\end{itemize}
	$W$ \emph{encounters} $V$, denoted $V \preceq W$, provided $U \leq W$ for some $V$-instance $U$.
	If $W$ fails to encounter $V$, we say $W$ \emph{avoids} $V$.
\end{defn}

To help distinguish the encountered word and the encountering word, ``pattern'' is elsewhere used to refer to $V$ in the encounter relation $V \preceq W$. 
Also, an instance of a word is sometimes called a ``substitution instance'' and ``witness'' is sometimes used in place of encounter.

%

%

%
%
%
%
%

\subsection{$r$-th Power-Free Words} \label{rthPow}

The earliest results in avoidability involved avoiding words of the form $x^r$.
When specifically discussing $x^r$-avoidance, the term \emph{$r$-th power-free} is generally used (or \emph{square-free} for $r=2$ and \emph{cube-free} for $r=3$).
We see in Figure~\ref{xx} that only finitely many square-free words exist over a given two-letter alphabet.
However, \textcite{T-06} demonstrated the existence of arbitrarily long (even infinite), ternary, square-free words.

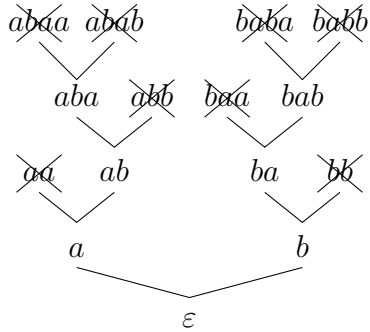
\begin{figure}[ht] 
	\centering
	\begin{threeparttable}

	\begin{tabular}{c}
	\begin{tikzpicture}
		\draw (0,0) node{$\varepsilon$};
		\draw (-1.5,.7) node[above]{$a$}--(0,.3)--(1.5,.7) node[above]{$b$};
		\draw (-2,1.7) node[above]{$aa$}--(-1.5,1.3)--(-1,1.7) node[above]{$ab$};
		\draw (1,1.7) node[above]{$ba$}--(1.5,1.3)--(2,1.7) node[above]{$bb$};
		\begin{scope}[xshift=-2cm, yshift=1.7cm]
			\draw (-.3,0)--(.3,.5);
			\draw (-.3,.5)--(.3,0);
		\end{scope}
		\begin{scope}[xshift=2cm, yshift=1.7cm]
			\draw (-.3,0)--(.3,.5);
			\draw (-.3,.5)--(.3,0);
		\end{scope}
		\draw (-1.5,2.7) node[above]{$aba$}--(-1,2.3)--(-.5,2.7) node[above]{$abb$};
		\draw (1.5,2.7) node[above]{$bab$}--(1,2.3)--(.5,2.7) node[above]{$baa$};
		\begin{scope}[xshift=-.5cm, yshift=2.7cm]
			\draw (-.3,0)--(.3,.5);
			\draw (-.3,.5)--(.3,0);
		\end{scope}
		\begin{scope}[xshift=.5cm, yshift=2.7cm]
			\draw (-.3,0)--(.3,.5);
			\draw (-.3,.5)--(.3,0);
		\end{scope}
		\draw(-2,3.7) node[above]{$abaa$}--(-1.5,3.2)--(-1,3.7) node[above]{$abab$};
		\draw(2,3.7) node[above]{$babb$}--(1.5,3.2)--(1,3.7) node[above]{$baba$};
		\begin{scope}[xshift=-2cm, yshift=3.7cm]
			\draw (-.3,0)--(.3,.5);
			\draw (-.3,.5)--(.3,0);
		\end{scope}
		\begin{scope}[xshift=-1cm, yshift=3.7cm]
			\draw (-.3,0)--(.3,.5);
			\draw (-.3,.5)--(.3,0);
		\end{scope}
		\begin{scope}[xshift=1cm, yshift=3.7cm]
			\draw (-.3,0)--(.3,.5);
			\draw (-.3,.5)--(.3,0);
		\end{scope}
		\begin{scope}[xshift=2cm, yshift=3.7cm]
			\draw (-.3,0)--(.3,.5);
			\draw (-.3,.5)--(.3,0);
		\end{scope}
	\end{tikzpicture}
	\end{tabular}

	\caption{Binary words that avoid $xx$.} \label{xx}

	\end{threeparttable}
\end{figure}

In the 1970s, a number of important results were proved regarding square-free words.
For example: \textcite{J-76} showed that there exists an infinite set of ternary square-free words $\mathcal{F}$ such that, for each $W \in \mathcal{F}$,  every word in $\mathcal{F} \setminus \{W\}$ avoids $W$; \textcite{L-76} characterized all maximal square-free words.
Within their seminal work on avoidability--the central result which we discuss later--\textcite{BEM-79} defined two interesting homomorphisms that preserved the property of being $r$-th power-free. In particular, $h: \N \rightarrow [3]$ that preserves it for $r \geq 2$ and $g: \N \rightarrow [2]$ for $r \geq 3$.

\subsection{k-Avoidability}

\begin{defn}
	A word $V$ is \emph{$k$-avoidable} provided, over a fixed alphabet of size $k$, there are infinitely many words that avoid $V$. 
	Inversely, $V$ is \emph{$k$-unavoidable} provided every sufficiently long word with at most $k$ distinct letters encounters $V$.
\end{defn}

We saw in Section~\ref{rthPow} that the word $xx$ is 3-avoidable but 2-unavoidable.
A word is \emph{doubled} provided every letter in the word occurs at least twice. 
Every doubled word is $k$-avoidable for some $k>1$ (see \cite{L-02}).

\begin{thm}[\cite{BW-12}, Theorem 2]
	``Let $p$ be a [word] of $m$ distinct [letters].
	\begin{enumerate}[1.]
		\item If $|p| \geq 3(2^{m-1})$, then $p$ is 2-avoidable. 
		\item If $|p| \geq 2^m$, then $p$ is 3-avoidable.''
	\end{enumerate}
\end{thm}

There remain a number of open problems regarding which words are $k$-avoidable for particular $k$.
See \textcite{L-02} and \textcite{C-05} for surveys on avoidability results.

\subsection{General Avoidability}

\begin{defn} \label{def:unavoidable}
	A word $V$ is \emph{unavoidable} provided, for any finite alphabet, there are only finitely many words that avoid $V$; equivalently, $V$ is $k$-unavoidable for all $k \geq 2$.
\end{defn}

The first classification of unavoidable words (Theorem~\ref{thm:BEM}) was by \textcite{BEM-79}, using the following definitions.

\begin{defn}
	``Let $W$ be a word. The letter $x$ is \emph{free for $W$} provided $x$ occurs in $W$ and for no $n \in \omega$ is it possible to find letters $e_0, \cdots, e_n$ and $f_0, \cdots, f_n$ such that all of the following are [factors] of $W$:
	\[ xe_0 \quad f_0e_0 \quad f_0e_1 \quad f_1e_1 \quad \cdots \quad f_ne_n \quad f_nx . '' \]

	``If $x$ is free for $W$, then $W^x$ is the word obtained from $W$ by deleting all occurrences of $x$.''

	``$U$ is \emph{obtained from $W$ by identification of letters} whenever'' for some letters ``$x$ and $y$ [...] occurring in $W$, $U$ is the word obtained from $W$ by substituting $x$ for $y$.''

	``$W$ \emph{reduces to $U$} provided there are words $V_0, V_1, \cdots, V_{n-1}$ with $W = V_0$, $U = V_{n-1}$ and [either] $V_{i+1} = V_i^x$ for some letter $x$ free in $V_i$ or $V_{i+1}$ is obtained from $V_i$ by identification of letters, for all $i$ with $0 \leq i < n-1$.''
\end{defn}

\begin{thm}[\cite{BEM-79}, Theorem 3.22] \label{thm:BEM}
	``The word $W$ is unavoidable if and only if $W$ reduces to a word of length one.''
\end{thm}

Three years later, Zimin published a fundamentally different classification of unavoidable words (\cite{Z-82} in Russian, \cite{Z-84} in English).

\begin{defn} \label{defn:Zimin}
	Define the \emph{$n$-th Zimin word} recursively by $Z_0 := \varepsilon$ and, for $n \in \N$, $Z_{n+1} = Z_nx_nZ_n$. Using the English alphabet rather than indexed letters:
	\[Z_1 = \textbf{a}, \quad Z_2 = a\textbf{b}a, \quad Z_3 = aba\textbf{c}aba, \quad Z_4 = abacaba\textbf{d}abacaba, \quad \ldots . \]
\end{defn}
Equivalently, $Z_n$ can be defined over the natural numbers as the word of length $2^n-1$ such that the $i$-th letter, $1 \leq i < 2^n$, is the 2-adic order of $i$.
	
\begin{thm}[\cite{Z-84}] \label{thm:Zimin}
	A word $V$ with $n$ distinct letters is unavoidable if and only if $Z_n$ encounters $V$.
\end{thm}

$Z_n$-instances are precisely \emph{sesquipowers of order $n$}. 
From \textcite{BLRS-08}, ``any nonempty word is a sesquipower of order 1; 
a word $w$ over an alphabet $A$ is a sesquipower of order $n > 1$ if $w = w_0vw_0$ for some words $w_0, v \in A^*$ with $v \neq \varepsilon$ and $w_0$ a sesquipower of order $n - 1$.''

\subsection{A Ramsey-Type Question}

With Zimin's concise characterization of unavoidable words, a natural combinatorial question follows: How long must a $q$-ary word be to guarantee that it encounters a given unavoidable word? 
By Definition~\ref{defn:f}, $\f(n,q)$ is the smallest integer $M$ such that every $q$-ary word of length $M$ encounters $Z_n$.

In 2014, three papers by different authors appeared, each independently proving bounds for $\f(n,q)$. 
\textcite{CR-14} showed that (Theorems~\ref{upper},~\ref{fnqlower})
\[ q^{2^{(n-1)}(1+o(1))} \leq \f(n,q) \leq {}^{n-1}(2q+1),\] 
where ${}^ba$ denotes an exponential tower with $b$ copies of $a$. 
These results were presented at the 45th Southeast International Conference on Combinatorics, Graph Theory, and Computing in March 2014. 

In June, \textcite{T-14} introduced a more general function $L(q,V)$ for what he calls the ``Ramsey number'' of any unavoidable word $V$.
He also attained similar lower and upper bounds for $L(q,Z_n) = \f(n,q)$.
Tao's lower bound, which we restate as Theorem~\ref{TaoLower}, is even more general, applying to any unavoidable word.

In September, \textcite{RS-14} also introduced the function $\f(n,q)$, together with the concept of ``minimal words of Zimin type $n$''; that is, instances of $Z_n$ which contain no $Z_n$-instance as a proper factor. 
We call such words \emph{minimal $Z_n$-instances}. 
Using minimal instances, and some computation, Rytter and Shur establish the best known upper bounds for $\f(3,q)$ and $\f(4,2)$. 
We restate their results in Section~\ref{MinZimin} for further use. 

A factor-avoidance variant of this function has been considered at least as early as the German work of \textcite{E-83}, some results of which were made more readily available in English by \textcite{BK-06}. 
For some fixed alphabet $\mathcal{A}$, a set of words $S$ is called unavoidable provided there are only finitely many words in $\mathcal{A}^*$ that do not contain any word in $S$ as a factor. 
Note that if the alphabet has at least 2 letters, every nonempty word by itself is avoidable. 
In Kitaev's work, $L_w(n)$ is the maximum length of a word in $\mathcal{A}^*$ that avoids some unavoidable set $S \subseteq \mathcal{A}^n$. 
\begin{thm}[\cite{E-83}, Theorem~1; \cite{BK-06}, Theorem~2.3]
	\[ L_w(n) = | \mathcal{A}|^{n-1}+n-2. \]
\end{thm}

%
%
%

\section{Word Densities} \label{dense}

Given nonempty words $V$ and $W$, the \emph{(instance) density of $V$ in $W$}, denoted $\delta(V,W)$, is the proportion of substrings of $W$ that contain instances of $V$.
For example, two of the $\binom{6+1}{2}$ substrings of $banana$ contain $xx$-instances: $anan$ and $nana$. Therefore, $\delta(xx, banana) = 2/\binom{7}{2}$. 


Recall that a word $V$ is doubled provided every letter in $V$ occurs at least twice.
For a doubled word $V$ with $k \geq 2$ distinct letters and an alphabet $\Sigma$ with $|\Sigma| = q \geq 4$, $(k,q) \neq (2,4)$, \textcite{BG-07} showed that there are at least $\lambda(k,q)^n$ words in $\Sigma^n$ that avoid $V$, where they defined the function $\gamma$ to be
\[\lambda(k,q) = m\left(1 + \frac{1}{(m-2)^k}\right)^{-1}.\]
This exponential lower bound on the number of words avoiding a doubled word hints at the moral of Chapter~\ref{DICHOT}: instances of doubled words are rare.
For doubled word $V$ and an alphabet $\Sigma$ with $q \geq 2$ letters, the probability that a random word $W_n \in \Sigma^n$ encounters $V$ is asymptotically 1. 
Indeed, the event that $W_n[b|V|,(b+1)|V|]$ is an instance of $V$ has nonzero probability and is independent for distinct $b \in \N$.
Nevertheless, the expected density $\delta_n(V,q) = \EE(\delta(V,W_n))$ (Definition~\ref{defn:density2}) is asymptotically negligible.
Specifically, the central result of Chapter~\ref{DICHOT} is the following dichotomy.
\begin{thm*}[\ref{dichotomy}]
	Let $V$ be a word on any alphabet.
	Fix integer $q \geq 2$.
	$V$ is doubled if and only if $\delta(V,q) = \lim_{n\rightarrow \infty} \delta_n(V,q) = 0$.
\end{thm*}

For doubled $V$, not only does $\delta(V,q) = 0$, but we establish tight concentration of $\delta(V,W_n)$ for random word $W_n \in [q]^n$.
\begin{thm*}[\ref{expectation}, \ref{variance}]
	Let $V$ be a doubled word, $q \geq 2$, and $W_n \in [q]^n$ chosen uniformly at random. 
	\[\frac{1}{n} \ll \EE(\delta(V,W_n)) \ll \frac{\log n}{n};\]
	\[ \Var(\delta(V,W_n)) \ll \frac{(\log n)^3}{n^3} \ll \EE(\delta(V,W_n))^2 \frac{(\log n)^3}{n}. \]
\end{thm*}

For nondoubled $V$, we know from the dichotomy that, if $\delta_n(V,q)$ converges, its limit is not 0. 
To get a handle on the nondoubled case, we consider instances of specified length, a perspective used in the proof of Theorem~\ref{fnqlower}.  
From Definition~\ref{defn:I}:
	Let $\Inst_n(W,\Sigma)$ be the set of $W$-instances in $\Sigma^n$, and $\II_n(W,q)$ the probability that a random length-$n$ $q$-ary word is a $W$-instance; that is, 
	\[ \II_n(W,|\Sigma|) =\frac{|\Inst_n(W,\Sigma)|}{|\Sigma|^n}. \]

\begin{ex}
	$\Inst_4(wow,[2]) = \{1111,1121, 1211,1221,2112, 2122,2212,2222\}$ and $\II_4(wow,2) = \frac{8}{2^4} = \frac{1}{2}$.
\end{ex}

\begin{thm*}[\ref{nondoubledProb}, \ref{cor:EdI}]
	Fix word $V$ and positive integer $q$. The limits $\delta(V,q)$ and $\II(V,q) = \lim_{n\rightarrow \infty} \II_n(V,q)$ both exist, and $\delta(V,q) = \II(V,q)$.
\end{thm*}
We also establish bounds for $\II(V,q)$ under various conditions.

\section{Looking Forward}

There are still many unexplored avenues within the combinatorial limit theory of free words.
The final part of this work, Chapter~\ref{FUTURE}, summarizes a few directions for further development.
There we also pose a number of open questions that arise from the present research.

%% file: ZiminAvoidance.tex
\chapter{Bounds on Zimin Word Avoidance} \label{AVOID}

Recall that $V$ is unavoidable provided, for any finite alphsabet, there are only finitely many words that avoid (i.e., do not encounter) $V$. 
Moreover, we stated Zimin's classification (Theorem~\ref{thm:Zimin}) that the unavoidable words are precisely the words encountered by what are now known as Zimin words (Definition~\ref{defn:Zimin}):
	\[Z_1 = a, \quad Z_2 = aba, \quad Z_3 = abacaba, \quad Z_4 = abacabadabacaba, \quad \ldots\]
\textcite{CR-14}, \textcite{T-14}, and \textcite{RS-14}, independently began investigating bounds on the length of words that avoid unavoidable words. 

\section{Avoiding the Unavoidable}


From Zimin's explicit classification of unavoidable words, a natural question arises in the Ramsey-theoretic paradigm: for a fixed unavoidable word $V$, how long can a word be that avoids $V$?
Our approach to this question is to start with avoiding the Zimin words, which gives upper bounds for all unavoidable words.
\begin{defn} \label{defn:f}
	$\f(n,q)$ is the least integer $M$ such that every $q$-ary word of length $M$ encounters $Z_n$. 
\end{defn}

Let ${}^ba$ denote the towering exponential $a^{a^{\cdot^{\cdot^a}}}$ with $b$ occurrences of $a$. 
This tetration is elsewhere denoted with Knuth's up-arrow notation by $a \uparrow \uparrow b$. 
${}^0a$ is defined to be 1. 

\begin{thm}[\cite{CR-14}, Theorem 1.1] \label{upper}
	For $n,q \in \Z^+$, 
	\[ f(n,q) \leq {}^{n-1}(2q+1) . \]
\end{thm}

\begin{proof}
	We proceed via induction on $n$. For the base case, set $n=1$. Every nonempty word is an instance of $Z_1$, so $\f(1,q) = 1$.

	For the inductive hypothesis, assume the claim is true for some positive $n$ and set $T=\f(n,q)$. That is, every $q$-ary word of length $T$ encounters $Z_n$.  Concatenate any $q^T+1$ strings $W_0, W_1, \ldots, W_{q^T}$ of length $T$ with an arbitrary letter $a_i$ between $W_{i-1}$ and $W_{i}$ for each positive $i \leq q^T$:
	\[U = W_0 \; a_1\; W_1 \; a_2 \; W_2 \; a_3\; \cdots \; W_{q^T-1} \; a_{q^T} \; W_{q^T}.\]
	
	By the pigeonhole principle, $W_i = W_j$ for some $i < j$.
	That string, being length $T$, encounters $Z_n$.  
	Therefore, we have some word $W \leq W_i$ that is an instance of $Z_n$ and shows up twice, disjointly, in $U$.
	The extra letter $a_{i+1}$ guarantee that the two occurrences of $W$ are not consecutive.
	This proves that an arbitrary word of length $(T+1)(q^T + 1)-1$ witnesses $Z_{n+1}$, so
	\[\f(n+1,q) \leq (T+1)(q^T + 1)-1  \leq (2q+1)^T = Q^T.\]
\end{proof}

There is clearly a function $Q(n,q)$ such that $\f(n+1,q) \leq {Q(n,q)}^{\f(n,q)}$ and $Q(n,q)$ tends to $q$ as $n \rightarrow \infty$. 
No effort has been made to optimize the choice of function, as such does not decrease the tetration in the bound.

The technique used to prove Theorem~\ref{upper} is first found in Lothaire's proof of unavoidability of $Z_n$ (\cite{L-02}, 3.1.3). 
\textcite{T-14} uses the same technique with different approximation to establish a similar upper bound.
\begin{thm}[\cite{T-14}, Theorem 6] \label{TaoUpper} 
	For integer $n \geq 2$ and $q \geq 2$,
	\[ \f(n,q) < {}^{(2n-1)}q. \]
\end{thm}

The technique used in the original proof by \cite{Z-84} implicitly gives, for $n\geq 2$, 
\[\f(n+1,q+1) \leq (\f(n+1,q)+2|Z_{n+1}|)\f(n,|Z_{n+1}|^2q^{\f(n+1,q)}).\]
This is an Ackermann-type function for an upper bound, which is much larger than the primitive recursive bound from Theorems~\ref{upper} and ~\ref{TaoUpper}.

Table~\ref{fn2} shows known values of $\f(n,2)$. 
Supporting word-lists and Sage code are found in Appendix~\ref{AppendZ}.

\begin{table}[ht]
\centering
\begin{threeparttable}
	\caption{Values of $\f(n,2)$.} \label{fn2}
	\begin{tabular}{c|c|c}
		$n$ & $Z_n$ & $\f(n,2)$\\ \hline
		0 & $\varepsilon$ & 0\\
		1 & a & 1\\
		2 & aba & 5\\
		3 & abacaba & 29\\
		4 & abacabadabacaba & $\geq 10483$
	\end{tabular}
\end{threeparttable}
\end{table}

\section{Finding a Lower Bound with the First Moment Method}

Throughout this section, $\Sigma$ is a fixed alphabet with $|\Sigma| = q \geq 2$ letters.

\begin{defn} \label{defn:I}
	Let $\Inst_n(W,\Sigma)$ be the set of $W$-instances in $\Sigma^n$, and $\II_n(W,q)$ the probability that a random length-$n$ $q$-ary word is a $W$-instance; that is, 
	\[ \II_n(W,|\Sigma|) =\frac{|\Inst_n(W,\Sigma)|}{|\Sigma|^n}. \]
\end{defn}

%

\begin{lem} [\cite{CR-14}, Lemma 2.1]
	For all $n,M \in \Z^+$,
	\[|\Inst_{(M+1)}(Z_n,\Sigma)| \geq q\cdot|\Inst_M(Z_n,\Sigma)|.\]
\end{lem}

\begin{proof}
	Take arbitrary $W \in \Inst_M(Z_n,\Sigma)$. 
	By the recursive construction of $Z_n$, we can write $W = W_1W_0W_1$ with $W_1 \in \Inst_N\left(Z_{(n-1)},\Sigma\right)$, where $2N < M$. 
	Choose the decomposition of $W$ to minimize $|W_1|$. 
	Then $W_1W_0x_iW_1 \in \Inst_{(M+1)}(Z_n,\Sigma)$ for each $i < q$. 
	
	The lemma follows, unless a $Z_n$-instance of length $M+1$ can be generated in two ways -- that is, if $W_1W_0aW_1 = V_1V_0bV_1$ for some $V_1V_0V_1 = V$, where $|V_1|$ is also minimized. 
	If $|V_1|<|W_1|$, then $V_1$ is a prefix and suffix of $W_1$, so $|W_1|$ was not minimized.
	But if $|V_1|>|W_1|$, then $W_1$ is a prefix and suffix of $V_1$, so $|V_1|$ was not minimized.
	Therefore, $|V_1|=|W_1|$, so $V_1=W_1$, which implies $a=b$ and $V=W$.
\end{proof}

\begin{cor}[\cite{CR-14}, Corollary 2.2] \label{ZnMonot}
	For all $n,M \in \Z^+$,
	\[\II_{(M+1)}(Z_n,q) \geq \II_M(Z_n,q).\]
\end{cor}

\begin{lem} [\cite{CR-14}, Lemma 2.3]
	For all $n,M \in \Z^+$,
	\[|\Inst_M(Z_n,\Sigma)| \leq \left(\frac{q}{q-1}\right)^{n-1}q^{(M-2^n+n+1)}.\]
\end{lem}

\begin{proof}
	The proof proceeds by induction on $n$. For the base case, set $n=1$. Every nonempty word is an instance of $Z_1$, so $|\Inst_M(Z_1,\Sigma)| = q^M$.
	
	For the inductive hypothesis, assume the inequality is true for some $n \in \Z^+$. The first inequality below comes from the following overcount of $Z_{n+1}$-instances of length $M$. Every such word can be written as $UVU$ where $U$ is a $Z_n$-instance of length $j<\frac{M}{2}$. Since an instance of $Z_n$ can be no shorter than $Z_n$, $2^n-1 \leq j <\frac{M}{2}$. For each possible $j$, there are $|\Inst_j(Z_n,\Sigma)|$ ways to choose $U$ and $q^{M-2j}$ ways to choose $V$. This is an overcount, since a Zimin-instance may have multiple decompositions.
	\begin{eqnarray*}
		\left|\Inst_M\left(Z_{(n+1)},\Sigma\right)\right| & \leq & \sum_{j = 2^n-1}^{\floor{(M-1)/2}} |\Inst_j(Z_n,\Sigma)|q^{M-2j} \\
		& \leq & \sum_{j = 2^n-1}^{\floor{(M-1)/2}}\left(\frac{q}{q-1}\right)^{n-1}q^{(j-2^n+n+1)}q^{M-2j}\\
		& = & \left(\frac{q}{q-1}\right)^{n-1}q^{(M-2^n+n+1)} \sum_{j = 2^n-1}^{\floor{(M-1)/2}}q^{-j}\\
		& < & \left(\frac{q}{q-1}\right)^{n-1}q^{(M-2^n+n+1)} \sum_{j = 2^n-1}^{\infty}q^{-j}\\
		& = & \left(\frac{q}{q-1}\right)^{n-1}q^{(M-2^n+n+1)}\left(\frac{q^{-(2^n-1)+1}}{q-1}\right)\\
		& = & \left(\frac{q}{q-1}\right)^{(n-1)+1}q^{(M -2^{n+1}+ (n+1)+1)}.
	\end{eqnarray*}
\end{proof}

\begin{cor} [\cite{CR-14}, Corollary 2.4]
	For all $n,M \in \Z^+$,
	\[ \II_M(Z_n,q) \leq \left(\frac{q}{q-1}\right)^{n-1}q^{(-2^n+n+1)}.\]
\end{cor}

\begin{thm}[\cite{CR-14}, Theorem 2.5] \label{fnqlower}
	As $q \rightarrow \infty$ or $n \rightarrow \infty$,
	\[\f(n,q) \geq \sqrt{\frac{2q^{2^n}}{q^{(n+1)}e^{\left(\frac{n-1}{q-1}\right)}}} - 1 = q^{2^{(n-1)}(1+o(1))} . \] 
\end{thm}

\begin{proof}
	Let word $W$ consist of $M$ uniform, independent random selections from $\Sigma$. 
	Define the random variable $X$ to count the number of subwords of $W$ that are instances of $Z_n$ (including repetition if a single subword occurs multiple times in $W$):
	\[X = \big|\big\{(i,j) \mid 0\leq i < j \leq M, W[i,j] \in \Inst_{(j-i)}(Z_n,\Sigma)\big\}\big|.\]
	By monotonicity with respect to word length (Corollary~\ref{ZnMonot}):
	\begin{eqnarray*}
		\EE(X) & = & \sum_{0 \leq i < j \leq M} \II_{(j-i)}(Z_n,q)\\	
		& \leq & \big|\big\{(i,j) \mid 0 \leq i < j \leq M\big\}\big| \cdot \II_M(Z_n,q)\\	
		& \leq & \binom{M+1}{2} \left(\frac{q}{q-1}\right)^{n-1}q^{(-2^n+n+1)}\\	
		& < & \frac{1}{2}(M+1)^2 e^{\left(\frac{n-1}{q-1}\right)}q^{(-2^n+n+1)}.
	\end{eqnarray*}
	
	There exists a word of length $M$ that avoids $Z_n$ when $E(X) < 1$.
	It suffices to show that:
	\begin{equation} \label{Msuffice}
		(M+1)^2 \left(\frac{1}{2}e^{\left(\frac{n-1}{q-1}\right)}q^{(-2^n+n+1)}\right) \leq 1.
	\end{equation}
	
	Solving \eqref{Msuffice} for $M$:
	\begin{eqnarray*}
		M &\leq&  \left(\frac{1}{2}e^{\left(\frac{n-1}{q-1}\right)}q^{(-2^n+n+1)}\right)^{-1/2} - 1\\
		&=& q^{2^{(n-1)}}\left(\frac{1}{2}e^{\left(\frac{n-1}{q-1}\right)}q^{(n+1)}\right)^{-1/2} - 1\\
		&=& q^{2^{(n-1)}(1+o(1))}.
	\end{eqnarray*}
\end{proof}

\textcite{T-14} uses the probabilistic method and generating functions and to prove a more general result.

\begin{thm}[\cite{T-14}, Corollary 1] \label{TaoLower}
	Suppose word $V$ has $r$ distinct letters with multiplicities $1 = k_1 = \cdots = k_s < k_{s+1} \leq \cdots \leq k_r$. If 
	\[ n < (1 + o(1)) \left[ (s+1)! \prod_{j = s+1}^r (q^{k_j - 1} - 1) \right]^{\frac{1}{s+1}}, \]
there is a length-$n$ $q$-ary word that avoids $V$.
\end{thm}

Applying Theorem~\ref{TaoLower} to Zimin words, Tao obtains
\[ \f(n,q) \geq (1 + o(1)) \sqrt{2\prod_{j=1}^{n-1} (q^{2^j-1} - 1)}. \]
As $q \rightarrow \infty$, 
\begin{eqnarray*}
	\sqrt{2\prod_{j=1}^{n-1} (q^{2^j-1} - 1)} & \sim & \sqrt{2\prod_{j=1}^{n-1} (q^{2^j-1})},
\end{eqnarray*}
and as $n \rightarrow \infty$,
 \begin{eqnarray*}
	 \sqrt{2\prod_{j=1}^{n-1} (q^{2^j-1})} & = & \sqrt{2} \left(q^{\left(\sum_{j = 1}^{n-1} (2^j - 1) \right)}\right)^{\frac{1}{2}} \\
	& \sim & \sqrt{2} \left(q^{\left(2^n - (n-1)\right)}\right)^{\frac{1}{2}} \\
	& = & q^{2^{n-1}(1 + o(1))}.
\end{eqnarray*}

\section{Using Minimal Zimin-Instances} \label{MinZimin}

\begin{defn} \label{defn:min}
	For fixed $n \in \Z^+$, a $Z_n$-instance is \emph{minimal} provided it has no $Z_n$-instance as a proper factor.

	Let $\m(n,q)$ be the number of minimal $Z_n$-instances over a fixed $q$-letter alphabet.
\end{defn}

The function $\m(n,q)$ was first introduces by \textcite{RS-14}. 
They used this concept of minimal Zimin-instances to improve the upper bounds of $\f(3,q)$ and $\f(4,2)$. 

\begin{lem}[\cite{RS-14}, Lemma 4.6] \label{RSlem1}
	The following holds for any integers $n,q > 2$:
	\[ \f(n+1,q) \leq (\f(n,q) + 1) \cdot \m(n,q) + \f(n,q). \]
\end{lem}

\begin{lem}[\cite{RS-14}, Lemma 4.7] \label{RSlem2}
	\[ \m(2,q) = q! \cdot \sum_{i = 1}^{q-1} \frac{2^{q-1-i}}{i!}. \]
\end{lem}

\begin{thm}[\cite{RS-14}, Theorem 4.4] \text{}
	\begin{itemize}
		\item $\f(1,q) = 1$;
		\item $\f(2,q) = 2q+1$;
		\item $\f(3,2) = 29$, $\f(3,q) = \sqrt{e} \cdot 2^q (q+1)! + 2q + 1$;
		\item $\f(4,2) \leq 236489$.
	\end{itemize}
\end{thm}

Lemma~\ref{RSlem1} follows from the same method used in Theorem~\ref{upper}. The bound on $\f(4,2)$ was established using a computer search to find $\m(3,2) = 7882$.

%% file: Densities.tex
\chapter{Word Densities} \label{DENSE}

\begin{defn} \label{defn:density}
	The \emph{factor density} of $V$ in $W$, denoted $\d(V,W)$, is the proportion of length-$|V|$ substrings of $W$ that are copies of $V$; that is
	\[ \d(V,W) = \frac{\big| \big\{(i,j) : 0\leq i < j \leq |W|, W[i,j]=V \big\}\big|}{|W| +1 - |V|}. \]

	The \emph{(instance) density} of $V$ in $W$, denoted $\delta(V,W)$, is the proportion of substrings of $W$ that are instances of $V$; that is
	\[ \delta(V,W) = \frac{\big| \big\{(i,j) : 0\leq i < j \leq |W|, W[i,j] \text{ is a } V\text{-instance} \big\}\big|}{\binom{|W|+1}{2}}. \]
	The \emph{($q$-)liminf density} of $V$ is,
	\[  \underline{\delta}(V,q) = \liminf_{\substack{W \in [q]^* \\ |W|\rightarrow \infty}} \delta(V,W). \]
\end{defn}

The liminf density is defined in terms of alphabet $[q]$ for convenience, but any fixed $q$-letter alphabet would suffice.
We need not define a limsup density or liminf factor density, as these would always be trivially 1 or 0.
A $\Sigma$-limsup factor density of $V$ might be of interest for alphabet $\Sigma \supseteq \L(V)$, but we do not investigate this here.
Table~\ref{table:numer} below gives a numeric summary of the best know bounds for $\underline{\delta}(Z_n,q)$.

The value of $\underline{\delta}(Z_2,q)$ for $q \geq 2$ is from Theorem~\ref{dZ2}. 
For $n=3$, the upper bound comes from Section~\ref{prob}, and the lower bounds are stated in Corollary~\ref{dZ3}. 
There we establish that $\underline{\delta}(Z_3,2) \geq \frac{1}{54}$, but Section~\ref{DeB} gives reason to believe that the truth is greater than $1/28$. 
Lower bounds for $\underline{\delta}(Z_4,q)$ are found in Theorem~\ref{dZn}, though the best lower bound for $q=2$ is in Corollary~\ref{dZ3}. 
Finally, the best upper bounds for $\underline{\delta}(Z_n,q)$ when $n\geq 4$ are from Section~\ref{IV}.

\begin{table}[ht]
\centering
\begin{threeparttable}

	\caption{Best known bounds for the $q$-liminf density of $Z_n$.} \label{table:numer}

	\begin{tabular}{c}
%
	$\begin{array}{c|c|c|c|c|c}
		\underline{\delta}(Z_n,q) &q=2&3&4&5&\cdots \\ \hline
		n=2 & 1/2 = .5 & 1/3 \approx .333 & 1/4 = .25 & 1/5 = .2 & \cdots \\ \hline
		3&\begin{matrix} .119 \\ 1.85\cdot 10^{-2}\end{matrix}  &\begin{matrix}1.84 \cdot 10^{-2} \\ 8.33\cdot 10^{-4} \end{matrix}& \begin{matrix}5.19 \cdot 10^{-3}\\ 5.31 \cdot 10^{-5} \end{matrix}& \begin{matrix}2.00 \cdot 10^{-3} \\ 3.22 \cdot 10^{-7} \end{matrix} & \cdots\\ \hline
	
		4&\begin{matrix}1.12\cdot10^{-3} \\ 2.40\cdot10^{-7} \end{matrix}  &\begin{matrix}8.80 \cdot 10^{-6} \\ 6.64\cdot 10^{-392943} \end{matrix} & \begin{matrix}3.23 \cdot 10^{-7} \\ 9.42\cdot 10^{-233250395} \end{matrix} & \begin{matrix}2.58 \cdot 10^{-8} \\ - \end{matrix}& \cdots \\ \hline
		5&\begin{matrix}3.43\cdot10^{-8} \\ - \end{matrix} &\begin{matrix}6.13 \cdot 10^{-13} \\ - \end{matrix} & \begin{matrix}3.01 \cdot 10^{-16} \\ - \end{matrix} & \begin{matrix}8.46 \cdot 10^{-19} \\ - \end{matrix} & \cdots \\ \hline
		\vdots & \vdots & \vdots & \vdots & \vdots & \ddots
	\end{array}$
	\end{tabular}
%

\end{threeparttable}
\end{table}

\section{Density Comparisons}


For graphs $F$ and $G$, $t(F,G)$ is the homomorphism density of $F$ in $G$: 
\[ t(F,G) = \frac{|\{\phi:V(F)\rightarrow V(G) \mid xy \in E(F) \Rightarrow \phi(x)\phi(y) \in E(G)\}|}{|V(G)|^{|V(F)|}} . \]
$K_n$ is the complete graph on $n$ vertices; that is, the graph $\ang{[n], \binom{[n]}{2}}$ with all $\binom{n}{2}$ possible edges.
In particular, $K_2$ is often simply called the edge graph, and $K_3$ the triangle graph.
For every graph $G$, we can plot an ordered pair $(x,y) = (t(K_2,G),t(K_3,G))$.
The closure of the set of all such points forms a connected region in $[0,1]^2$ (see Section 2.1 of \cite{L-12}), with which we can visualize the relationship between edge-densities and triangle-densities in graphs.
The tight upper bound for this region is $y \leq x^{\frac{3}{2}}$, which is a case of the Kruskal-Katona Theorem (\cite{K-63}, \cite{K-68}).
The lower bound of $y \geq x(2x-1)$ is a result of \textcite{G-59}, but was shown to be tight only for $x = 1 - \frac{1}{k}$ by \textcite{B-76}.

We perform a similar comparison for word densities of some fundamental words. 
In Section~\ref{Compareakal}, we calculate the limit set, as $|W| \rightarrow \infty$, of the closure of the set of points of the form $( \d(a^k,W), \d(a^\ell,W))$.
Then Section~\ref{Comparez2z3} shows all points $(\delta(Z_2,W),\delta(Z_3,W))$ for all $W$ of particular, small lengths, presenting them in the context of bounds to be proved later.

\subsection{Factor Density of \texorpdfstring{$a^k$}{aaa.....a}.} \label{Compareakal}

\begin{lem} \label{lem:uppertri}
	For word $W$ and integers $0 < k < \ell$, 
	\[ \d(a^\ell,W) \leq \d(a^k,W), \]
	with equality only when either $\d(a^\ell,W) = 1$ (that is, $W = a^m$ with $m \geq \ell$) or $\d(a^k,W) = 0$.
\end{lem}

\begin{proof}
	Within any $ba^rc$ in $W$ with $a \not\in \{b, c\}$ and $r \geq \ell$, there are $\ell - k$ more copies of $a^k$ than of $a^\ell$.
	Hence, unless $\d(a^\ell,W) = 0$, 
	\[\d(a^k,W) \geq \frac{(|W|+1-\ell)\d(a^\ell,W) + (\ell - k)}{|W|+1-k} 
\geq \d(a^\ell,W),\]
	with equality on the right only when $\d(a^\ell,W) = 1$.
\end{proof}

\begin{lem}
	For integers $0 < k < \ell$ and rational number $d_k \in \left[0,\frac{\ell-k}{\ell}\right] \cap \Q$, there exit arbitrarily large words $W$ with $\d(a^k,W) = d_k$ and $\d(a^\ell,W) = 0$.
\end{lem}

\begin{proof}
	Let $d = \frac{u}{v}$ for positive integers $1\leq u<v$.
	For $u,v \in \Z^+$, $\frac{u}{v} = d \leq \frac{\ell-k}{\ell}$ implies $v(\ell - k) - u\ell \geq 0$.
	Let $W_r = (a^{\ell-1}b)^{ru}b^{r(v(\ell - k) - u\ell) + k - 1}$ for $r \in \Z^+$.
	The number of length-$k$ substrings in $W_r$ is
	\[ |W_r| + 1 - k = (\ell - 1 + 1)(ru) + (r(v(\ell - k) - u\ell) + k - 1) + 1 - k = rv(\ell - k). \]
	Now $a^\ell \not \leq W$, and the number of occurrences of $a^k$ in $W_r$ is
	\[ ((\ell-1) + 1 - k)(ur) = ru(\ell - k). \]
	Therefore, $\d(a^\ell,W_r)=0$ and 
	\[ \d(a^k,W_r) = \frac{ru(\ell - k)}{rv(\ell - k)} = \frac{u}{v} = d. \]

\end{proof}

\begin{lem} \label{lem:lowertri}
	For integers $0 < k < \ell$, and as $|W| \rightarrow \infty$, 
	\[ \ell(\d(a^k,W) - 1) \lesssim k(\d(a^\ell,W) - 1) . \]
\end{lem}

\begin{proof}
	Let $|W| = M$. 
	For given $W$, set $d_k = \d(a^k,W)$. 
	Also, let $c_k$ count the number of maximal factors in $W$ of the form $a^x$ for $k\leq x \leq \ell -1$ and $A_k$ count the number of $a^k$-occurrences in the $c_k$ such strings, so $A_k \leq (\ell - k)c_k$.
	Similarly, set $d_\ell = \d(a^\ell, W)$ and let $c_\ell$ count the number of maximal factors in $W$ of the form $a^x$ for $\ell \leq x$ and $A_\ell$ count the number of $a^\ell$-occurrences in the $c_\ell$ such strings.
	Hence, as $M \rightarrow \infty$,
	\begin{eqnarray*}
		d_k & = & \frac{c_\ell(\ell-k) + A_\ell + A_k}{M + 1 - k} \\
			& \sim & \frac{c_\ell(\ell-k) + A_\ell + A_k}{M+1}; \\
		d_\ell & = & \frac{A_\ell}{M + 1 - \ell} \\
			& \sim & \frac{A_\ell}{M+1}; \\
		M & \geq & \ell c_\ell + A_\ell + k c_k + A_k - 1.
	\end{eqnarray*}

	The desired asymptotic inequality is $\ell(d_k - 1) \lesssim k(d_\ell - 1)$, which is equivalent to $\ell d_k - k d_\ell \lesssim \ell - k$. 
	Applying what we said about $d_k$, $d_\ell$, and $M$:
	\begin{eqnarray*}
		\ell d_k - k d_\ell & \sim & \frac{\ell[c_\ell(\ell-k) + A_\ell + A_k] - k[A_\ell]}{M+1} \\
		& \leq & \frac{\ell[c_\ell(\ell-k) + A_\ell + A_k] - k[A_\ell]}{\ell c_\ell + A_\ell + k c_k + A_k}.
	\end{eqnarray*}
	Therefore, it suffices to show one of the following equivalent statements, the last of which we already established.
	\begin{eqnarray*}
		\frac{\ell[c_\ell(\ell-k) + A_\ell + A_k] - k[A_\ell]}{\ell c_\ell + A_\ell + k c_k + A_k} & \leq & \ell - k; \\
		\ell[\ell c_\ell + A_\ell + A_k] - k[\ell c_\ell + A_\ell] & \leq & (\ell - k )[\ell c_\ell + A_\ell + k c_k + A_k]; \\
		k([\ell c_\ell + A_\ell + k c_k + A_k] - [\ell c_\ell + A_\ell]) & \leq & \ell([\ell c_\ell + A_\ell + k c_k + A_k] - [\ell c_\ell + A_\ell + A_k]); \\
		k(k c_k + A_k) & \leq & \ell(kc_k); \\
		k c_k + A_k & \leq & \ell c_k; \\	
		A_k & \leq & (\ell - k) c_k. 
	\end{eqnarray*}
\end{proof}
%
%
%

\begin{lem} \label{lem:densetri}
	Let $0 < k < \ell$ be integers and $(d_k, d_\ell) \in \Q^2$ be found on the triangle defined by the following inequalities:
	\begin{itemize}
		\item $0 \leq d_\ell \leq d_k$; 
		\item $k(d_\ell - 1) \geq \ell(d_k - 1)$.
	\end{itemize}
	Then for all $\epsilon > 0$, there exist arbitrarily long words $W$ such that
	\[ \big| \d(a^k,W) - d_k \big| < \epsilon \textnormal{ and }  \big| \d(a^\ell,W) - d_\ell \big| < \epsilon. \]
\end{lem}

\begin{proof}
	Since $k(d_\ell - 1) = \ell(d_k - 1)$ and $d_\ell = 0$ intersect when $d_k = \frac{\ell - k}{\ell}$, We can break the triangle into two cases:
	\begin{enumerate}[(I)]
		\item $0 \leq d_\ell \leq d_k \leq \frac{\ell - k}{\ell}$.
		\item $0 \leq d_\ell \leq d_k$, $\frac{\ell-k}{\ell} < d_k$, $k(d_\ell - 1) \geq \ell(d_k - 1)$.
	\end{enumerate}

	Without loss of generality, let $d_k = \frac{u_k}{v}$ and $d_\ell = \frac{u_\ell}{v}$ for some integers $u_\ell, u_k, v \in \Z$ satisfying $0 \leq u_\ell \leq u_k \leq v \neq 0$.
	For $r \in \Z^+$, define length $vr$-word $W_r$ to be
	\[ W_r = a^{ru_\ell}(ba^{\ell-1})^{\floor{\frac{ru_k - ru_\ell}{\ell - k}}}b^{r'},\]
	with $r' = rv - ru_\ell - \ell\floor{\frac{ru_k - ru_\ell}{\ell - k}}$ in order that $|W_r| = vr$.
	This word is constructed to give necessary densities for all sufficiently large $r$:
	\begin{eqnarray*}
		d(a^\ell,W_r) & = & \frac{ru_\ell + 1 - \ell}{rv + 1 - \ell} \sim \frac{ru_\ell}{rv} = d_\ell; \\
		d(a^k,W_r) & = & \frac{(ru_\ell + 1 - k) + (\ell - k)\floor{\frac{ru_k - ru_\ell}{\ell - k}}}{rv + 1 - k} \sim \frac{ru_\ell + (\ell-k)\frac{ru_k - ru_\ell}{\ell - k}}{rv} = d_k.
	\end{eqnarray*}
	But for $W_r$ to be well-defined, we need $r' \geq 0$. It suffices to show that 
	\[ rv - ru_\ell - \ell\left(\frac{ru_k - ru_\ell}{\ell - k}\right) \geq 0, \]
	which is equivalent to both of the following:
	\[ (v - u_\ell)(\ell - k) \geq \ell(u_k - u_\ell); \quad \frac{\ell - k}{\ell} \geq \frac{u_k - u_\ell}{v - u_\ell}. \]

	\textbf{Case (I):} Since $u_k \leq v$,
	\[ \frac{u_k - u_\ell}{v - u_\ell} \leq \frac{u_k}{v} = d_k \leq \frac{\ell - k}{\ell}.\]

	\textbf{Case (II):} Since $k(d_\ell - 1) \geq \ell(d_k - 1)$, 
	\[ \frac{k}{\ell} \leq \frac{1 - d_k}{1 - d_\ell}, \]
	which implies
	\[ \frac{\ell - k}{\ell} = 1 - \frac{k}{\ell} \geq 1 -  \frac{1 - d_k}{1 - d_\ell} =  \frac{d_k - d_\ell}{1 - d_\ell} = \frac{u_k - u_\ell}{v - u_\ell}.\]
\end{proof}

\begin{thm} \label{thm:dtri}
	For integers $0 < k < \ell$ and ordered pair $(x,y) \in [0,1]^2$, there exist arbitrarily long words $W$ with $\d(a^k,W)\sim x$ and $\d(a^\ell,W) \sim y$ if and only if $(x,y)$ falls in the triangular region shown in Figure~\ref{figure:dakl}, defined as follows:
	\begin{itemize}
		\item $0 \leq y \leq x$; and 
		\item $k(y - 1) \geq \ell(x - 1)$.
	\end{itemize}
\end{thm}

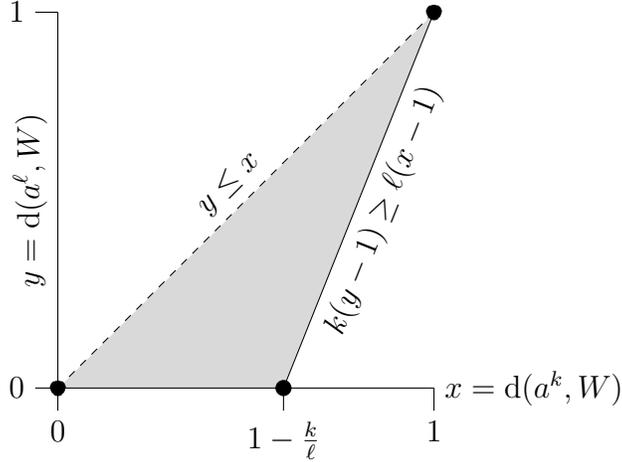
\begin{figure}[ht]
\centering
\begin{threeparttable}

	\begin{tabular}{c}
	\begin{tikzpicture}
		\filldraw[fill=black!15,draw=white] (0,0)--(5,5)--(3,0)--(0,0);
		\draw (0,0) --  node [sloped,above] {$y = \d(a^\ell,W)$} (0,5);
		\draw (0,0) -- (5,0) node [right] {$x = \d(a^k,W)$};
		\foreach \x in {0,1} {
			\draw (5*\x,0)--(5*\x,-.3) node [below] {\x};
			\draw (0,5*\x)--(-.3,5*\x) node [left] {\x};
		};
		\filldraw (5,5) -- node[sloped,below] {$k(y - 1) \geq \ell(x - 1)$} (3,0) circle (.1cm) -- (3,-.3) node [below] {$1 - \frac{k}{\ell}$};
		\filldraw[dashed] (0,0) circle (.1cm) -- node [sloped,above] {$y\leq x$} (5,5) circle (.1cm);
	\end{tikzpicture}
	\end{tabular}

	\caption{Relation between $\d(a^k,W)$, $\d(a^\ell,W)$ for $0 < k < \ell$ as $|W|\rightarrow \infty$.} \label{figure:dakl}

\end{threeparttable}
\end{figure}

\begin{proof}
	The upper and lower bounds are established in Lemmas~\ref{lem:uppertri} and \ref{lem:lowertri}, respectively.
	The density of points in this triangle is established in Lemma~\ref{lem:densetri}.
\end{proof}

\subsection{Instance Density of Zimin Words} \label{Comparez2z3}

The same sort of comparison as we see in Theorem~\ref{thm:dtri} can also be made for instance densities.
Figure~\ref{figure:deltaZ2Z3} shows the relationship between the instance densities of $Z_2$ and $Z_3$ in binary words of length 28.
See Appendix~\ref{Z2Z3} for plots corresponding to binary words of lengths 13, 16, 19, 22, 25, and 28 and the code used to generate the points.
The graphs also give a preview of some asymptotic results that we will establish later.

\begin{figure}[ht]
\centering
\begin{threeparttable}

	\begin{tabular}{c}
		\includegraphics[width=400px]{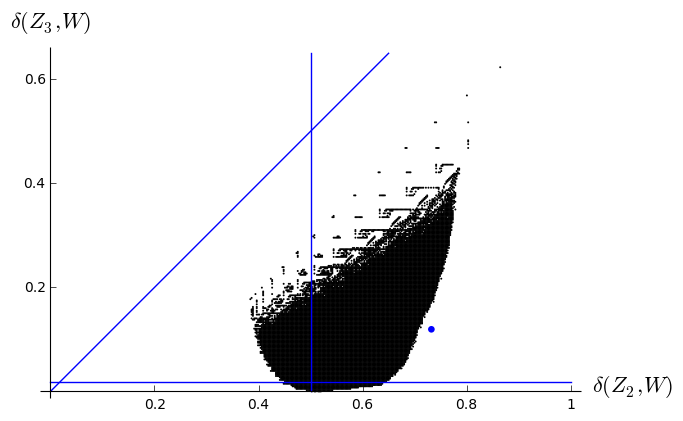}
	\end{tabular}

	\caption{All $(x,y)$ with $x=\delta(Z_2,W)$, $y=\delta(Z_3,W)$ for $W \in [2]^{28}$.} \label{figure:deltaZ2Z3}

	\begin{tablenotes}[flushleft]
		\item Assuming binary $W$:
		\item The line $y=x$ is an absolute upper bound.
		\item The vertical blue line is $\underline{\delta}(Z_2,2) = \frac{1}{2}$.
		\item The horizonal blue line is a lower bound on $\underline{\delta}(Z_3,2)$.
		\item The point at $\approx (0.7322, 0.1194)$ shows expected densities in large random $W$.
	\end{tablenotes}

\end{threeparttable}
\end{figure}

\section{Minimum Density of Zimin Words}

Recall that $\delta(Z,W)$, the (instance) density of word $Z$, is the proportion of substring of $W$ that are $Z$-instances. 
Thus, $\delta(Z,W)$ can always be written as a rational number with denominator $\binom{|W|+1}{2}$, the number of substrings of $W$.
Let us begin with the following trivial facts.
\begin{fact} 
	$\delta(Z_1,W) = 1$ for every nonempty word $W \neq \varepsilon$.
\end{fact}

\begin{fact} \label{fact:triv2}
	For any $q \in \Z^+$, if $V$ has no recurring letter, $\underline{\delta}(V,q) = 1$.
\end{fact}

\begin{proof}
	The density of $V$ is bounded above by 1.
	As $|W|$ grows, the proportion of substrings of length at least $|V|$ goes to 1:
		\[\sum_{\ell = |V|}^{|W|} (|W| + 1 - \ell) \sim \binom {|W|+1}{2}.\]
	Since no letter occurs twice in $V$, every word of length at least $|V|$ is a $V$-instance.
\end{proof}

The remainder of this chapter is primarily devoted to finding $\underline{\delta}(Z_n,q)$, the liminf density of Zimin words.

\begin{thm} \label{dZ2} 
	\[ \underline{\delta}(Z_2,q) = \frac{1}{q}.\]
\end{thm}

\begin{proof}
	Fix alphabet $\{x_0, \ldots, x_{q-1}\}$.
	Given word $W$, let $a_i$ be the number of occurrences of $x_i$ in $W$ for each $i<q$. The number of $Z_2$-instances of the form $x_iBx_i$ is at least
	\[\binom{a_i}{2} - (a_i - 1),\]
	where $(a_i-1)$ is subtracted to avoid counting consecutive occurrences of $x_i$.
	Therefore, using the Cauchy-Schwarz inequality,
	\begin{eqnarray*}
		\binom{|W|+1}{2} \delta(Z_2,W) & \geq & \sum_{i = 0}^{q-1} \left(\binom{a_i}{2} - (a_i - 1)\right) \\
		& = & \left( \sum_{i = 0}^{q-1} \frac{a_i(a_i-1)}{2}\right) - \left(\sum_{i = 0}^{q-1}  (a_i - 1)\right)\\
		& = & \frac{1}{2} \left(\sum_{i = 0}^{q-1} a_i^2\right) - \frac{3}{2}\left(\sum_{i = 0}^{q-1}  a_i\right) + q\\
		& \geq &  \frac{1}{2q} \left(\sum_{i = 0}^{q-1} a_i\right)^2 - \frac{3}{2}\left(\sum_{i = 0}^{q-1}  a_i\right) + q\\
		& = & \frac{|W|^2}{2q} - \frac{3|W|}{2} + q.
		\\
		\delta_W(Z_2) & \geq & \left(\frac{|W|^2}{2q} - \frac{3|W|}{2} + q\right) \frac{1}{\binom{|W|+1}{2}}\\
		& \sim & \frac{1}{q}.
	\end{eqnarray*}

	Consider words $W_k = x_0^kx_1^k\cdots x_{q-1}^k$, so $|W_k| = qk$.
	Every $Z_2$-instance in $W_k$ is with subword $x_i^\ell$ for $3\leq \ell \leq k$.
	Therefore
	\begin{eqnarray*}
		\delta(Z_2,W_k) & = & \frac{\sum_{i = 0}^{q-1} \left(\binom{k}{2} - (k - 1)\right)}{\binom{qk+1}{2}}\\
		& \sim & \frac{qk^2/2}{(qk)^2/2}\\
		& = & \frac{1}{q}.
	\end{eqnarray*}
\end{proof}

Recall that the function $\f(n,q)$ from Chapter~\ref{AVOID} gives the least $M$ such that every $q$-ary word of length $M$ encounters $Z_n$.

\begin{thm} \label{dZn}
	\[ \underline{\delta}(Z_{n+1},q) \geq \frac{1}{(\f(n,q) - 2^n +2)^2 q^{\f(n,q)+1}}.\]
\end{thm}

\begin{proof}
	On a fixed $q$-letter alphabet, there are fewer than $q^{\f(n,q)+1}$ words of length at most $\f(n,q)$. 
	In particular, there are fewer than $q^{\f(n,q)+1}$ $Z_n$-instances of length at most $\f(n,q)$. 
	If given word $W$ is spliced into substrings of length $\f(n,q)$, each substring is guaranteed to contain a $Z_n$-instance. 
	In fact, since the shortest images of $Z_n$ are length $2^n - 1$, we can allow the substrings to overlap by $2^n-2$ letters and still avoid counting the same encounter of $Z_n$ twice.
	Picking one $Z_n$-instance from each substring, we form a set of $\floor{|W|/\f(n,q)}$ nonoverlapping $Z_n$-occurrences in $W$. Enumerate the $Z_n$-instances of length at most $\f(n,q)$ by $V_0, V_1, \ldots, V_{k-1}$ for some $k < q^{\f(n,q)+1}$. Let $a_i$ be the number of occurrences of $V_i$ in the set for each $i < k$. Then
	\[\sum_{i=0}^{k-1} a_i = \floor{\frac{|W|}{\f(n,q) - (2^n -2)}}.\]
	Therefore, 
	\begin{eqnarray*}
		\binom{|W|+1}{2} \delta(Z_{n+1},W) & \geq & \sum_{i=0}^{k-1}  \left(\binom{a_i}{2} - (a_i - 1)\right) \\
		& \sim & \frac{1}{2} \left(\sum_{i = 0}^{k-1} a_i^2\right)\\
		& \geq &  \frac{1}{2k} \left(\sum_{i = 0}^{k-1} a_i\right)^2 \\
		& = & \frac{\floor{\frac{|W|}{\f(n,q) - (2^n -2)}}^2}{2k}\\
		\delta(Z_{n+1},W) & \gtrsim & \frac{\floor{\frac{|W|}{\f(n,q) - (2^n -2)}}^2}{2k}\frac{1}{\binom{|W|+1}{2}}\\
		& \sim & \frac{1}{(\f(n,q) - 2^n + 2)^2k}\\
		& > & \frac{1}{(\f(n,q) - 2^n + 2)^2 q^{\f(n,q)+1}}.
	\end{eqnarray*}
\end{proof}

%

%

We call a $Z_n$-instance \emph{minimal} provided it has no proper factor that is also a $Z_n$-instance (a concept  introduced by \cite{RS-14}).
Recall that $\m(n.q)$ is the number of minimal $Z_n$-instances over a fixed $q$-letter alphabet.
Any time a string encounters $Z_n$, it must contain a minimal $Z_n$-instance.
Therefore, we can replace $q^{\f(n,q)+1}$ in Theorem~\ref{dZn} with $\m(n,q)$.

\begin{cor} \label{dZn-Mn}
	\[ \underline{\delta}(Z_{n+1},q) \geq \frac{1}{(\f(n,q) - 2^n + 2)^2 \m(n,q)}.\]
\end{cor}

\begin{lem}[Corollary of Lemma~\ref{RSlem2}] \label{M2q}
	\[\m(2,q) < q!2^q.\]
\end{lem}


Recall $\f(2,q) = 2q+1$, $\m(2,2)=6$, $\f(3,2) = 29$ (Table~\ref{fn2}), and $\m(3,2) = 7882$ \parencite{RS-14}.

\begin{cor} \label{dZ3} \text{}
	\begin{itemize}
		\item $\displaystyle \underline{\delta}(Z_3,2) \geq \frac{1}{54}$;
		\item $\displaystyle \underline{\delta}(Z_3,q) \geq \frac{1}{(2q-1)^2q!2^q}$; 
		\item $\displaystyle \underline{\delta}(Z_4,2) \geq \frac{1}{4169578}$.
	\end{itemize} 
\end{cor}
We have strong evidence in Section~\ref{DeB} that $\underline{\delta}(Z_3,2) > \frac{1}{28}$.


\subsection{Limits of Probabilities} \label{prob}

We denote with $\II_M(V,q)$ the probability that a random $q$-ary word of length $M$ is a $V$-instance. 
We prove in Chapter~\ref{DICHOT} that the limit probability $\II(V,q) = \lim_{M \rightarrow \infty} \II_M(V,q)$ always exists.
Consequently,
\[ \underline{\delta}(V,q) \leq \II(V,q). \]

In Chapter~\ref{ASYMP}, we provide upper bounds for $\II(Z_n,q)$ and a method to explicitly calculate $\II(Z_2,q)$ and $\II(Z_3,q)$, thus establishing various upper bounds for $\underline{\delta}(Z_n,q)$.

\section{The de Bruijn Graph} \label{DeB}

\begin{defn}
	For a fixed alphabet $\Sigma$ and positive integer $k$, the \emph{$k$-dimensional de Bruijn graph} is a directed graph with vertex set $\Sigma^k$ and an edge from $U$ to $W$ whenever $U = aV$ and $W = Vb$ for some $V \in \Sigma^{k-1}$ and $a,b \in \Sigma$.
\end{defn}

\textcite{E-83} construed words as walks on a de Bruijn graphs to prove bounds for permutation pattern avoidance, and his work is delivered to us from German into English by \textcite{BK-06}.
We now demonstrate how this perspective can be utilized to find minimum word densities.

\begin{defn}
	A \emph{bifix} of $W$ is a word that is both a proper initial string and terminal string.
	$W$ is \emph{bifix-free} provided $W$ has no bifix.
	$W$ is \emph{V-bifix-free} provided $W$ has no bifix that is a $V$-instance.
	$W$ is a \emph{minimal $V$-instance} provided there is no proper factor of $W$ that is a $V$-instance.
\end{defn}

Every $Z_3$-instance can be described by its shortest $Z_2$-bifix (that is, its $Z_2$-bifix that is itself $Z_2$-bifix-free). 
While building long words you can undercount the number of $Z_3$-instances by keeping track of the number of each $Z_2$-bifix-free $Z_2$-instance of length at most $k$. 

\begin{lem}
	Fix integers $q,n \geq 2$.
	Let $\mathbf{V}$ be a finite set of $Z_{(n-1)}$-bifix-free $Z_{(n-1)}$-instances in $[q]^*$.
	For $V \in \mathbf{V}$, let $c_V$ be the count of $V$-occurrences in $W$.
	Then
	\[ \delta(Z_n,W) \geq \frac{1}{\binom{|W|+1}{2}} \sum_{V \in \mathbf{V}} \left( \binom{c_V}{2} - |V|c_V\right). \]
\end{lem}

\begin{proof}
	For any given $V$-occurrence, the next $|V|$ occurrences might overlap or be consecutive, not allowing for a $Z_n$-instance. 
	But that still leaves at least $\binom{c_V}{2} - |V|c_V$ words of the form $VUV$ where $|U|>0$.
\end{proof}

Since Zimin words are unavoidable, if $\mathbf{V}$ contains all the minimal Zimin words, then the subtracted $|V|c_V$ terms is asymptotically negligible, because
	\[ \lim_{|W|\rightarrow \infty} \sum_{V \in \mathbf{V}} c_V = \infty. \]

For demonstration, the set of minimal $Z_2$-instance in $\{0,1\}^*$, which are inherently $Z_2$-bifix-free, is $\mathbf{V} = \{ 000, 010, 101, 111, 0110, 1001\}$. 
Let us look at word construction as taking a walk on the 4-dimensional de Bruijn graph.
Each of the $2^4$ vertices is a nyble, which is a 4-bit string (half the length of a byte).
In Figure~\ref{deB}, the solid arrow indicates appending a 1 and a dashed line, a 0.

\begin{figure}[ht]
\centering
\begin{threeparttable}

	\begin{tabular}{c}
	\begin{tikzpicture}[rotate=270, xscale=.58, yscale=.92]
		\tikzstyle{node} = [rectangle, fill=white, opacity=1, draw=black, minimum width = 1cm]
		\draw (-8,0) node [node] {0000};
		\draw (-6,1) node [node] {0001};
		\draw (-4,1) node [node] {0010};
		\draw (0,3) node [node] {0011};
		\draw (-4,-1) node [node] {0100};
		\draw (0,1) node [node] {0101};
		\draw (2,0) node [node] {0110};
		\draw (6,1) node [node] {0111};
		\draw (-6,-1) node [node] {1000};
		\draw (-2,0) node [node] {1001};
		\draw (0,-1) node [node] {1010};
		\draw (4,1) node [node] {1011};
		\draw (0,-3) node [node] {1100};
		\draw (4,-1) node [node] {1101};
		\draw (6,-1) node [node] {1110};
		\draw (8,0) node [node] {1111};
	
		\draw[->,thick] (-7.5,.5) -- (-6.5,1);
		\draw[->,thick] (-5.5,1.5) -- (-.5,3);
		\draw[->,thick] (-3.5,1) -- (-.5,1);
		\draw[->,thick] (.5,3) -- (5.5,1.5);
		\draw[->,thick] (-3.5,-.4) -- (-2.5,-.4);
		\draw[->,thick] (.5,1) -- (3.5,1);
		\draw[->,thick] (2.5,-.4) -- (3.5,-.4);
		\draw[->,thick] (6.5,1) -- (7.5,.6);
		\draw[->,thick] (-6,-.4) -- (-6,.4);
		\draw[->,thick] (-2,.6) -- (-.5,2.5);
		\draw[->,thick] (-.5,-.4) -- (-.5,.4);
		\draw[->,thick] (4.5,1) -- (5.5,1);
		\draw[->,thick] (-.5,-2.5) -- (-2,-.6);
		\draw[->,thick] (4,-.4) -- (4,.4);
		\draw[->,thick] (5.5,-1) -- (4.5,-1);
		\draw[->,thick] (8.5,.4) -- (9,.4) -- (9,-.4)--(8.5,-.4);
	
		\draw[->,thick,dashed] (7.5,-.6) -- (6.5,-1);
		\draw[->,thick,dashed] (5.5,-1.5) -- (.5,-3);
		\draw[->,thick,dashed] (3.5,-1) -- (.5,-1);
		\draw[->,thick,dashed] (-.5,-3) -- (-5.5,-1.5);
		\draw[->,thick,dashed] (3.5,.4) -- (2.5,.4);
		\draw[->,thick,dashed] (-.5,-1) -- (-3.5,-1);
		\draw[->,thick,dashed] (-2.5,.4) -- (-3.5,.4);
		\draw[->,thick,dashed] (-6.5,-1) -- (-7.5,-.6);
		\draw[->,thick,dashed] (6,.4) -- (6,-.4);
		\draw[->,thick,dashed] (2,-.6) -- (.5,-2.5);
		\draw[->,thick,dashed] (.5,.4) -- (.5,-.4);
		\draw[->,thick,dashed] (-4.5,-1) -- (-5.5,-1);
		\draw[->,thick,dashed] (.5,2.5) -- (2,.6);
		\draw[->,thick,dashed] (-4,.4) -- (-4,-.4);
		\draw[->,thick,dashed] (-5.5,1) -- (-4.5,1);
		\draw[->,thick,dashed] (-8.5,-.4) -- (-9,-.4) -- (-9,.4)--(-8.5,.4);
	
		\begin{scope}[yshift=8.5cm]
		\draw (-8,0) node [node] {000};
		\draw (-6,1) node [node] {X};
		\draw (-4,1) node [node] {010};
		\draw (0,3) node [node] {X};
		\draw (-4,-1) node [node] {X};
		\draw (0,1) node [node] {101};
		\draw (2,0) node [node] {0110};
		\draw (6,1) node [node] {111};
		\draw (-6,-1) node [node] {000};
		\draw (-2,0) node [node] {1001};
		\draw (0,-1) node [node] {010};
		\draw (4,1) node [node] {X};
		\draw (0,-3) node [node] {X};
		\draw (4,-1) node [node] {101};
		\draw (6,-1) node [node] {X};
		\draw (8,0) node [node] {111};
	
		\draw[->,thick] (-7.5,.6) -- (-6.5,1);
		\draw[->,thick] (-5.5,1.5) -- (-.5,3);
		\draw[->,thick] (-3.5,1) -- (-.5,1);
		\draw[->,thick] (.5,3) -- (5.5,1.5);
		\draw[->,thick] (-3.5,-.45) -- (-2.5,-.45);
		\draw[->,thick] (.5,1) -- (3.5,1);
		\draw[->,thick] (2.5,-.45) -- (3.5,-.45);
		\draw[->,thick] (6.5,1) -- (7.5,.6);
		\draw[->,thick] (-6,-.45) -- (-6,.45);
		\draw[->,thick] (-2,.6) -- (-.5,2.5);
		\draw[->,thick] (-.5,-.45) -- (-.5,.45);
		\draw[->,thick] (4.5,1) -- (5.5,1);
		\draw[->,thick] (-.5,-2.5) -- (-2,-.6);
		\draw[->,thick] (4,-.45) -- (4,.45);
		\draw[->,thick] (5.5,-1) -- (4.5,-1);
		\draw[->,thick] (8.5,.45) -- (9,.45) -- (9,-.45)--(8.5,-.45);
	
		\draw[->,thick,dashed] (7.5,-.6) -- (6.5,-1);
		\draw[->,thick,dashed] (5.5,-1.5) -- (.5,-3);
		\draw[->,thick,dashed] (3.5,-1) -- (.5,-1);
		\draw[->,thick,dashed] (-.5,-3) -- (-5.5,-1.5);
		\draw[->,thick,dashed] (3.5,.45) -- (2.5,.45);
		\draw[->,thick,dashed] (-.5,-1) -- (-3.5,-1);
		\draw[->,thick,dashed] (-2.5,.45) -- (-3.5,.45);
		\draw[->,thick,dashed] (-6.5,-1) -- (-7.5,-.6);
		\draw[->,thick,dashed] (6,.45) -- (6,-.45);
		\draw[->,thick,dashed] (2,-.6) -- (.5,-2.5);
		\draw[->,thick,dashed] (.5,.45) -- (.5,-.45);
		\draw[->,thick,dashed] (-4.5,-1) -- (-5.5,-1);
		\draw[->,thick,dashed] (.5,2.5) -- (2,.6);
		\draw[->,thick,dashed] (-4,.45) -- (-4,-.45);
		\draw[->,thick,dashed] (-5.5,1) -- (-4.5,1);
		\draw[->,thick,dashed] (-8.5,-.45) -- (-9,-.45) -- (-9,.45)--(-8.5,.45);
	
		\end{scope}
	
	\end{tikzpicture}
	\end{tabular}
	
	\caption{$Z_2$-instances on the 4-dimensional de Bruijn graph.}
	\label{deB}
	
	\begin{tablenotes}[flushleft]
	\item Left is the 4-dimensional de Bruijn graph; right is a graph indicating the minimal $Z_2$-instances encountered walking on the de Bruijn graph.
	\end{tablenotes}

\end{threeparttable}
\end{figure}
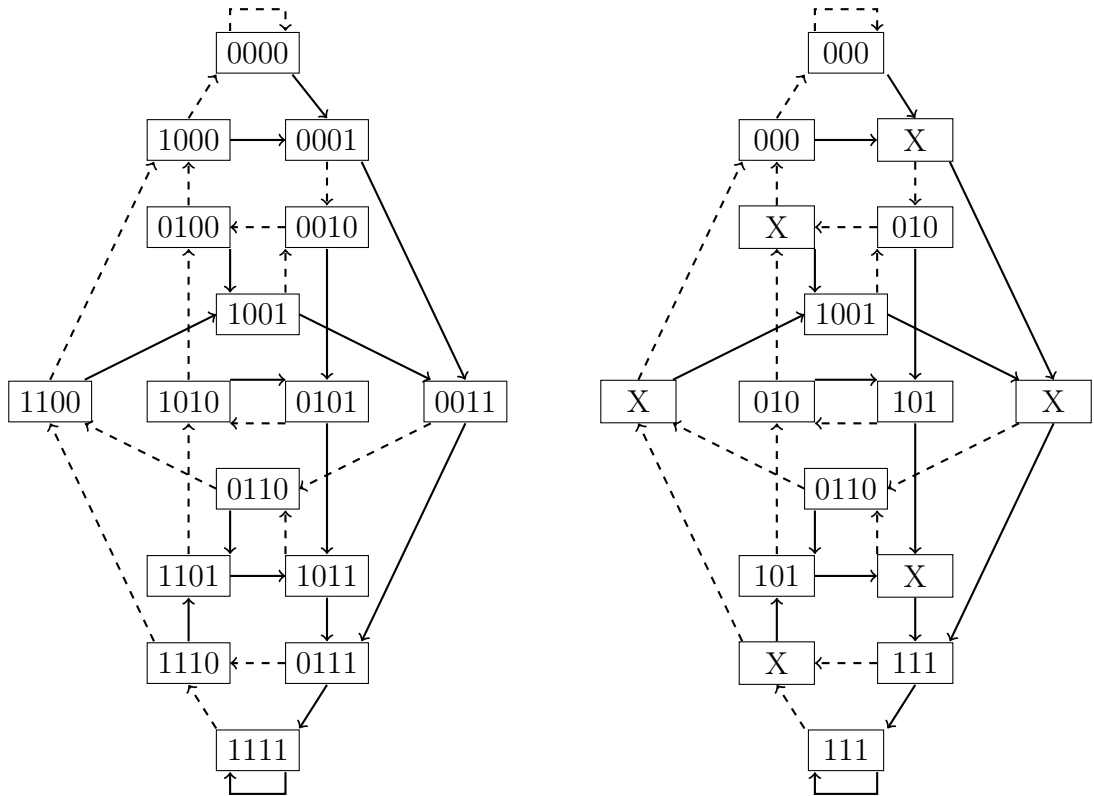

For a random walk of length $M$ on the de Bruin graph--so the corresponding word $W$ has length $(M+3)$--let $Q_n(M)$ be the number of times node $n$ showed up, which means $\sum_{n=0}^{15} Q_n(M) = M$. 
We can count the number of occurrences, $R_V(M)$, of each minimal $Z_2$-instances, $V$, in $W$ as follows. 
(To avoid any undercount, assume we do not start on a node beginning with a length-3 minimal $Z_2$-instance.)
\[\begin{array}{l l l}
	R_{000}(M) = Q_{0000}(M) + Q_{1000}(M); && R_{111}(M) = Q_{0111}(M) + Q_{1111}(M);\\
	R_{010}(M) = Q_{0010}(M) + Q_{1010}(M); && R_{101}(M) = Q_{0101}(M) + Q_{1101}(M);\\
	R_{0110}(M) = Q_{0110}(M); && R_{1001}(M) = Q_{1001}(M).
\end{array}\]
As $M \rightarrow \infty$, the density of $Z_3$-subwords is asymptotically at least
\[\frac{\sum_{V \in \mathbf{V}} \binom{R_V(M)}{2}}{\binom{M}{2}} \sim \frac{\sum_{V \in \mathbf{V}} R_V^2(M)}{M^2}.\]

One can assign probabilities to the outgoing edges of each nyble. 
Define probability tuple $p = \ang{p_n : n \in \{0,\ldots,15\}} \in [0,1]^{16}$ with $p_n$ being the probability that node $n$ is followed by a 1. 
Given an long random walk with fixed probabilities $p$, define $q = \ang{q_n : n \in \{0,\ldots,15\}} \in [0,1]^{16}$ where $q_n$ is the proportion of node-$n$ encounters in the walk. 
This leads to the following system of 17 equations with $k \in \{0,1,2,3,4,5,6,7\}$:
\begin{eqnarray*}
	q_{2k} & = & q_{k}(1-p_{k}) + q_{k+8}(1-p_{k+8}); \\
	q_{2k+1} & = & q_{k}p_{k} + q_{k+8}p_{k+8};\\
	1 & = & \sum_{i = 0}^{15} q_i.
\end{eqnarray*}

Further, define $r_V$ as $R_V(W)$ above, substituting $q_n$ for $Q_n(M)$. 
\[\begin{array}{l l l l l }
	r_{000} = q_0 + q_8; && r_{010} = q_2 + q_{10}; && r_{0110} = q_6; \\
	r_{111} = q_7 + q_{15}; && r_{101} = q_5 + q_{13}; && r_{1001} = q_9.
\end{array}\]

Then the expected $Z_3$-density is asymptotically at least $d = \sum_{V\in \mathbf{V}} r_V^2$. 
By solving the above system of 17 equations for the $q_n$ in terms of the $p_n$, rewrite $d$ in terms of the probabilities. 
Minimizing $d$ over the 16-dimensional unit cube--each probability is in $[0,1]$--should give a lower bound for $\underline{\delta}(Z_3,2)$.
We need only to show that for every limit density of $\delta(Z_3,2)$, or at least for the liminf-density, there exists an associated set of probabilities for the de Bruijn graph.

Using the function sage.numerical.optimize.minimize\_constrained() in Sage \parencite{S-14}, one can obtain probabilities producing a lower bound for $Z_3$-density that is slightly larger than 1/28. 
From these approximate results, we have identified the following distinct probability edge-assignments which each give a density of exactly 1/28.
For two of these, we also have associated families of words which exhibit the given probabilities as $n$ grows.
(`--' denotes that a node does not appear a positive proportion of the time, so its probability is irrelevant).
\begin{eqnarray*}
	p^{(1)} & = & (-, 4/5, 0, 3/5, 2/5, -, 1/5, 0, 1, 4/5,-, 3/5, 2/5, 1, 1/5, -);\\
	p^{(2)} & = & (-, 1, 0, 3/4, 1,-, 1/2, 0, 1, 1/2, -, 0, 1/4, 1, 0, -),\\
	& & W^{(2)}_n = (0001110010011100011011000111)^n;\\
	p^{(3)} & = & (-, 1, -, 3/5, 2/5, -, 1/5, 0, 1, 1, 0, -, 2/5, 0, 1/5, -),\\
	& & W^{(3)}_n = ((11010001)^3(101001)^2(110001)^{12}(1001)^8)^n.
\end{eqnarray*}

\begin{conj}
	 $\underline{\delta}(Z_3,2) > \frac{1}{28}.$
\end{conj}

The conjecture is with a strict inequality, as we can presumably increase the lower bound by using a larger set of $Z_2$-instances.
For example, the set of all $Z_2$-bifix-free $Z_2$-instances of length at most $5$ is 
\[ \{000,010,101,111,0110,1001,01001,01101,10010,10110 \}.\]
We would then view words as walks on the 5-dimensional de Bruijn graph and minimize the associated expression in $2^5 = 32$ variables.

%% file: DensityDichotomy.tex
\chapter{Density Dichotomy in Random Words} \label{DICHOT}

Definition~\ref{defn:I} is contained within Definition~\ref{defn:density2} below for completeness within this chapter.

\begin{defn} \label{defn:density2}
	Fixed $n$ and select $W_n \in [q]^n$ uniformly at random.
	The \emph{expected density} of $V$ is 
	\[ \delta_n(V,q) = \EE(\delta(V,W_n)). \]
	The \emph{asympototic expected density} of $V$ is 
	\[ \delta(V,q) = \lim_{n \rightarrow \infty} \delta_n(V,q). \]

	The set of $V$-instances in $\Sigma^n$ is $\Inst_n(V,\Sigma)$.
	The probability that a random length-$n$ $q$-ary word is a $V$-instance is
	\[ \II_n(V,q) = \frac{\big|\Inst_n(V,[q])\big|}{q^n}. \]
	The \emph{asymptotic instance probability} of $V$ is
	\[ \II(V,q) = \lim_{n\rightarrow \infty} \II_n(V,q). \]
\end{defn}

Sometimes we will count homomorphisms to attain density upper bounds.
\begin{defn} \label{defn:hom}
	Fix alphabets $\Gamma$ and $\Sigma$ and assume $V \preceq W$. 
	An \emph{encounter of $V$}, or \emph{$V$-encounter}, in $W$ is an ordered triple $(a,b,\phi)$ where $W[a,b] = \phi(V)$ for nonerasing homomorphism $\phi:\Gamma^* \rightarrow \Sigma^*$. 
	When $\Gamma = \L(V)$ and $W \in  \Sigma^*$, denote with $\hom(V,W)$ the \emph{number of $V$-encounters} in $W$. 
(Note that the conditions on $\Gamma$ and $\Sigma$ are necessary for $\hom(V,W)$ to not be trivially 0 or $\infty$.) 
	For $W_n \in [q]^n$ chosen uniformly at random, the \emph{expected number of $V$-encounters} is
	\[ \hom_n(V,q) = \EE(\hom(V,W_n)). \]
\begin{ex}
$\hom(ab,cde) = 4$ since $cde[0,2]$ is an instance of $ab$ by one homomorphism $\{a,b\}^* \rightarrow \{c,d,e\}^*$, $cde[1,3]$ is an instances of $ab$ by one homomorphism, and $cde[0,3]$ is an instance of $ab$ by two homomorphisms.

In fact, for $q \in \Z^+$, $\hom_3(ab,q) = 4$, since $\hom(ab,W) = 4$ for all $W \in [q]^3$.
\end{ex}
\end{defn}

\section{The Dichotomy}

\begin{thm} [\cite{CR-15}, Theorem 2.1] \label{dichotomy}
	Let $V$ be a word on any alphabet.
	Fix integer $q \geq 2$.
	The following are equivalent:
	\begin{enumerate}
		\renewcommand{\theenumi}{(\roman{enumi})}
		\item $V$ is doubled (that is, every letter in $V$ appears at least twice);
		\item $\delta(V,q) = 0$.
	\end{enumerate}
\end{thm}

\begin{proof}
	First we prove $(i) \Longrightarrow (ii)$. 
	Let $W_n \in [q]^n$ be chosen uniformly at random. 
	Note that in $W_n$, there are in expectation the same number of encounters of $V$ as there are of any anagram of $V$. 
	Indeed, if $V'$ is an anagram of $V$ and $\phi$ is a nonerasing homomorphism, then $|\phi(V')| = |\phi(V)|$.
	\begin{fact} [\cite{CR-15}, Fact 2.2] \label{anagram}
		If $V'$ is an anagram of $V$, then 
		\[ \hom_n(V,q)=\hom_n(V',q) . \]
	\end{fact}
	
	Assume $V$ is doubled and let $\Gamma = \L(V)$ and $k = |\Gamma|$.
	Given Fact~\ref{anagram}, we consider an anagram $V' = XY$ of $V$, where $|X| = k$ and $\Gamma = \L(X) = \L(Y)$.
	That is, $X$ comprises one copy of each letters in $\Gamma$ and all the duplicate letters of $V$ are in $Y$.
	
	We obtain an upper bound for the average density of $V$ by estimating $\hom_n(V',q)$. 
	To do so, sum over starting position $i$ and length $j$ of encounters of $X$ in $W_n$ that might extend to an encounter of $V'$.
	There are $\binom{j+1}{k+1}$ homomorphisms $\phi$ that map $X$ to $W[i,i+j]$ and the probability that $W_n[i+j,i+j+|\phi(Y)|] = \phi(Y)$ is at most $q^{-j}$.
	Also, the series $\sum_{j = k}^{\infty} \binom{j+1}{k+1} q^{-j}$ converges (try the ratio test) to some $c$ not dependent on $n$.
	\begin{eqnarray*}
		\delta_n(V,q) &\leq& \frac{1}{\binom{n+1}{2}} \hom_n(V',q) \\
		&<& \frac{1}{\binom{n+1}{2}} \sum_{i = 0}^{n - |V|} \sum_{j = k}^{n - i} \binom{j+1}{k+1} q^{-j} \\
		&<& \frac{1}{\binom{n+1}{2}} \sum_{i = 0}^{n - |V|} c \\
		&=& \frac{c(n-|V|+1)}{\binom{n+1}{2}} \\
		&=& O(n^{-1}),
	\end{eqnarray*}
	
	We prove $(ii) \Longleftarrow (i)$ by contraposition. Assume there is a letter $x$ that occurs exactly once in $V$.
	Write $V = TxU$ where $\L(V) \setminus \L(TU) = \{x\}$.
	We obtain a lower bound for $\delta_n(V,q) = \EE(\delta(V,W_n))$ by only counting encounters with $|\phi(TU)| = |TU|$.
	Note that each such encounter is unique to its instance, preventing double-counting.
	For this undercount, we sum over encounters with $W_n[i,i+j]=\phi(x)$.
	\begin{eqnarray*}
		\delta_n(V,q) &=& \delta_n(TxU,q) \\
		&\geq&\frac{1}{\binom{n+1}{2}} \sum_{i = |T|}^{n-|U|-1} \sum_{j = 1}^{i-|T|} q^{-||TU||} \\
		&=&q^{-||TU||}\frac{1}{\binom{n+1}{2}} \sum_{i = |T|}^{n-|U|-1} (i - |T|) \\
		&=&q^{-||TU||}\frac{\binom{n-|UT|}{2}}{\binom{n+1}{2}}  \\
		&\sim&q^{-||TU||}\\
		&>&0.
	\end{eqnarray*}
\end{proof}

It behooves us now to develop more precise theory for these two classes of words: doubled and nondoubled.
Lemma~\ref{base} below both helps develop that theory and gives insight into the detrimental effect that letter repetition has on encounter frequency.

\begin{prn} [\cite{CR-15}, Proposition 2.3] \label{CRT}
	For $k \in \Z^+$, $\overline{r} = \{r_1, \ldots, r_k\} \in (\Z^+)^k$, and $d = \gcd_{i \in [k]}(r_i)$, there exists integer $N = N_{\overline{r}}$ such that for every $n>N$ there exist coefficients $a_1, \cdots, a_k \in \Z^+$ such that $dn = \sum_{i=1}^{k} a_ir_i$ and $a_i \leq N$ for $i \geq 2$.
\end{prn}

\begin{proof}
	For each $j \in [r_1/d]$, find integer coefficients $b_i^{(j)}$ so that $jd$ is a linear combination of the $r_i$: $jd = \sum_{i=1}^{k} b_i^{(j)}r_i$. 
	Let $m = 1+ \left|\min\left(b_i^{(j)}\right)\right|$, the minimum taken over all $i$ and $j$. 
	Define $a_i^{(j)} = b_i^{(j)} + m>0$ and $R = \sum_{i=1}^{k} r_i$.
	Now for each $j$, \[\sum_{i=1}^{k} a_i^{(j)}r_i = \sum_{i=1}^{k} b_i^{(j)}r_i + \sum_{i=1}^{k} mr_i = jd + mR.\]

	Set $N = r_1+mR$. 
	For $n>N$, identify $j_n \in [r_1/d]$ such that 
	\[ dn \equiv j_nd+mR \mod{r_1} . \]
	Then $a_i = a_i^{(j_n)}$ for $i>1$ and $a_1 = \frac{1}{r_1}\left(dn - \sum_{i = 2}^{k} a_ir_i \right)$.
\end{proof}

\begin{lem} [\cite{CR-15}, Lemma 2.4] \label{base}
	For any word $V$, Let $\Gamma = \L(V) = \{x_1, \ldots, x_k\}$ where $x_i$ has multiplicity $r_i$ for each $i \in [k]$. 
	Let $U$ be $V$ with all letters of multiplicity $r= \min_{i \in [k]}(r_i)$ removed. 
	Finally, let $\Sigma$ be any finite alphabet with $|\Sigma| = q\geq 2$ letters. 
	Then for a uniformly randomly chosen $V$-instance $W \in \Sigma^{dn}$, where $d = \gcd_{i \in [k]}(r_i)$, there is asymptotically almost surely a homomorphism $\phi: \Gamma^* \rightarrow \Sigma^*$ with $\phi(V) = W$ and $|\phi(U)| < \sqrt{dn}$.
\end{lem}

\begin{proof}
	Let $a_n$ be the number of $V$-instances in $\Sigma^n$ and $b_n$ be the number of homomorphisms $\phi : \Gamma^* \rightarrow \Sigma^*$ such that $|\phi(V)| =  n$.
	Let $b_n^1$ be the number of these $\phi$ such that $\phi(U) < \sqrt{n}$ and $b_n^2$ the number of all other $\phi$ so that $b_n = b_n^1 + b_n^2$.
	Similarly, let $a_n^1$ be the number of $V$-instances in $\Sigma^n$ for which there exists a $\phi$ counted by $b_n^1$ and $a_n^2$ the number of instances with no such $\phi$, so $a_n = a_n^1 + a_n^2$.
	Observe that $a_n^2 \leq b_n^2$.

	Without loss of generality, assume $r_1 = r$ (rearrange the $x_i$ if not). 
	We now utilize $N = N_{\overline{r}}$ from Proposition~\ref{CRT}. 
	For sufficiently large $n$, we can undercount $a_{dn}^1$ by counting homomorphisms $\phi$ with $|\phi(x_i)| = a_i$ for the $a_i$ attained from Proposition~\ref{CRT}. 
	Indeed, distinct homomorphisms with the same image-length for every letter in $V$ produce distinct $V$-instances.
	Hence 
	\begin{eqnarray*}
		a_{dn}^1 & \geq & q^{\sum_{i=1}^k a_i} \\
		& \geq & q^{\left(\frac{dn - (k-1)N}{r} + r(k-1)\right)} \\
		& = & cq^{\left(\frac{dn}{r}\right)},
	\end{eqnarray*}
	where $c = q^{(k-1)(r^2-N)/r}$ depends on $V$ but not on $n$.	
	To overcount $b_n^2$ (and $a_{dn}^2$ by extension), we consider all $\binom{n+1}{|V|+1}$ ways to partition an $n$-letter length and so determine the lengths of the images of the letters in $V$. 
	However, for letters with multiplicity strictly greater than $r$, the sum of the lengths of their images must be at least $\sqrt{n}$.
	\begin{eqnarray*}
		b_n^2 & \leq & \binom{n+1}{|V|+1}  \sum_{i = \ceil{\sqrt{n}}}^n q^{\left(\frac{n - i}{r} + \frac{i}{r+1}\right)}\\
		& =  & \binom{n+1}{|V|+1} \sum_{i = \ceil{\sqrt{n}}}^n q^{\left(\frac{n}{r} - \frac{i}{r(r+1)}\right)}\\
		& < & n^{|V|+2}  q^{\left(\frac{n}{r} - \frac{\sqrt{n}}{r(r+1)}\right)}\\
		& = & q^{\frac{n}{r}}o(1).\\
		a_{dn}^2 & \leq & b_{dn}^2\\
		& = & o(a_{dn}^1).
	\end{eqnarray*}
	
	That is, the proportion of $V$-instances of length $dn$ that cannot be expressed with $|\phi(U)| < \sqrt{dn}$ diminishes to 0 as $n$ grows.
\end{proof}
 
\section{Density of Nondoubled Words}

In Theorem~\ref{dichotomy}, we show that the density of nondoubled $V$ in long random words (over a fixed alphabet with at least two letters) does not approach 0. 
The natural follow-up question is: Does the density converge?
To answer this question, we first prove the following lemma. 
Fixing $V = TxU$ where $x$ is a nonrecurring letter in $V$, the lemma tells us that all but a diminishing proportion of $V$-instances can be obtained by some $\phi$ with $|\phi(TU)|$ negligible.

\begin{lem} [\cite{CR-15}, Lemma 3.1] \label{bulk}
	Let $V = U_0x_1U_1x_2\cdots x_rU_r$ with $r\geq 1$, where $U = U_0U_1\cdots U_r$ is doubled with $k$ distinct letters (though any particular $U_j$ may be the empty word), the $x_i$ are distinct, and no $x_i$ occurs in $U$. 
	Further, let $\Gamma$ be the $(k+r)$-letter alphabet of $V$ and let $\Sigma$ be any finite alphabet with $q\geq 2$ letters. 
	Then there exists a nondecreasing function $g(n) = o(n)$ such that, for a randomly chosen $V$-instance $W \in \Sigma^n$, there is asymptotically almost surely a homomorphism $\phi: \Gamma^* \rightarrow \Sigma^*$ with $\phi(V) = W$ and $|\phi(x_r)| > n - g(n)$.
\end{lem}

\begin{proof}
	Let $X_i = x_1x_2\cdots x_i$ for $0 \leq i \leq r$ (so $X_0 = \varepsilon$). 
	For any word $W$, let $\Phi_W$ be the set of homomorphisms $\{\phi:  \Gamma^* \rightarrow \Sigma^* \mid \phi(V)=W\}$ that map $V$ onto $W$. 
	Define $\textbf{P}_i$ to be the following proposition for $i \in [r]$:
	\begin{quotation}
		\noindent There exists a nondecreasing function $f_i(n) = o(n)$ such that, for a randomly chosen $V$-instance $W \in \Sigma^n$, there is asymptotically almost surely a homomorphism $\phi \in \Phi_W$ such that $|\phi(UX_{i-1})| \leq f_i(n)$.
	\end{quotation}
	
	The conclusion of this lemma is an immediate consequence of proposition $\textbf{P}_r$, with $g(n) = f_r(n)$, which we will prove by induction.
	Lemma~\ref{base} provides the base case, with $r = 1$ and $f_1(n) = \sqrt{n}$.
	
	Let us prove the inductive step: $\textbf{P}_i$ implies $\textbf{P}_{i+1}$ for $i \in [r-1]$.	
	Roughly speaking, this says: If most instances of $V$ can be made with a homomorphism $\phi$ where $|\phi(UX_{i-1})|$ is negligible, then most instances of $V$ can be made with a homomorphism $\phi$ where $|\phi(UX_{i})|$ is negligible.
	
	Assume $\textbf{P}_{i}$ for some $i \in [r-1]$, and set $f(n) = f_{i}(n)$. 
	Let $A_n$ be the set of $V$-instances in $\Sigma^n$ such that $|\phi(UX_{i-1})| \leq f(n)$ for some $\phi \in \Phi_W$. 
	Let $B_n$ be the set of all other $V$-instances in $\Sigma^n$. 
$\textbf{P}_{i}$ implies $|B_n| = o(|A_n|)$.
	
	Case 1: $U_{i} = \varepsilon$, so $x_{i}$ and $x_{i+1}$ are consecutive in $V$. 
	When $|\phi(UX_{i-1})| \leq f(n)$, we can define $\psi$ so that $\psi(x_{i}x_{i+1}) = \phi(x_{i}x_{i+1})$ and $|\psi(x_{i})|=1$; otherwise, let $\psi(y) = \phi(y)$ for $y \in\Gamma \setminus \{x_i,x_{i+1}\}$.
	Then $|\phi(UX_i)| \leq f(n)+1$ and $\textbf{P}_{i+1}$ with $f_{i+1}(n) = f_i(n) + 1$.

	Case 2: $U_i \neq \varepsilon$, so $|U_i| > 0$.
	Let $g(n)$ be some nondecreasing function such that $f(n) = o(g(n))$ and $g(n) = o(n)$. 
(This will be the $f_{i+1}$ for $\textbf{P}_{i+1}$.)
	Let $A_n^\alpha$ consist of $W \in A_n$ such that $|\phi(UX_{i})| \leq g(n)$ for some $\phi\in \Phi_W$. Let $A_n^\beta = A_n \setminus A_n^\alpha$. 
The objective henceforth is to show that $|A_n^\beta| = o(|A_n^\alpha|)$.
	
	For $Y \in A_n^\beta$, let $\Phi_Y^\beta$ be the set of homomorphisms $\{\phi \in \Phi_Y : |\phi(UX_{i-1})| \leq f(n)\}$ that disqualify $Y$ from being in $B_n$. 
	Hence $Y \in A_n$ implies $\Phi_Y^\beta \neq \emptyset$.
	Since $Y \not \in A_n^\alpha$, $\phi \in \Phi_Y^\beta$ implies $|\phi(UX_{i})| > g(n)$, so $|\phi(x_i)|> g(n)-f(n)$.
	Pick $\phi_Y \in \Phi_Y^\beta$ as follows:
	\begin{itemize}
		\item Primarily, minimize $|\phi(U_0x_1U_1x_2 \cdots U_{i-1}x_{i})|$;
		\item Secondarily, minimize $|\phi(U_i)|$;
		\item Tertiarily, minimize $|\phi(U_0x_1U_1x_2 \cdots U_{i-1})|$.
	\end{itemize}
	
	Roughly speaking, we have chosen $\phi_Y$ to move the image of $U_i$ as far left as possible in $Y$.
	But since $Y \not\in A_n^\alpha$, we want it further left!
	
	To suppress the details we no longer need, let $Y = Y_1\phi_Y(x_i) \phi_Y(U_i) \phi_Y(x_{i+1}) Y_2$, where $Y_1 = \phi_Y(U_0x_1U_1x_2 \cdots U_{i-1})$ and $Y_2 = \phi_Y(U_{i+1}x_{i+2} \cdots U_r)$. 
	
	Consider a word $Z \in \Gamma^n$ of the form $Y_1Z_1 \phi_Y(U_i) Z_2 \phi_Y(U_i) \phi_Y(x_{i+1}) Y_2$, where $Z_1$ is an initial string of $\phi_Y(x_{i})$ with $2f(n) \leq |Z_1| < g(n) - 2f(n)$ and $Z_2$ is a final string of $\phi_Y(x_{i})$. 
	(See Figure~\ref{fig:YtoZ}.)
	In a sense, the image of $x_i$ was too long, so we replace a leftward substring with a copy of the image of $U_i$.
	Let $C_Y$ be the set of all such $Z$ with $|Z_1|$ a multiple of $f(n)$.
	For every $Z \in C_Y$ we can see that $Z \in A_n^\alpha$, by defining $\psi \in \Phi_Z$ as follows:
	\[\psi(y) = \left\{ \begin{array}{l l}
		Z_1 & \mbox{ if } y = x_{i}; \\
		Z_2 \phi_Y(U_i)\phi_Y(x_{i+1})& \mbox{ if } y = x_{i+1}; \\
		\phi_Y(y) & \mbox{ otherwise.}
	\end{array} \right.\]
	
\begin{figure}[ht]
\centering
\begin{threeparttable}

	\begin{tabular}{c}
			\begin{tikzpicture}[scale=.5]
				\filldraw[fill=gray!20,draw=white](6,-1) rectangle (9,1);
				\draw(0,1)--(27,1)--(27,-1)--(0,-1)--(0,1);
				\draw(0,0)--(27,0);
				\draw(-1,.5) node {$Y=$};
				\draw(-1,-.5) node {$Z=$};
				\draw(4,1)--(4,-1.5);
				\draw(2,.5) node {$Y_1$};
				\draw(2,-.5) node {$Y_1$};
				\draw(14,1)--(14,-1);
				\draw(6,0)--(6,-1.5);
				\draw(9,0)--(9,-1.5);
				\draw(9,.5) node {$\phi_Y(x_i)$};
				\draw(5,-.5) node {$Z_1$};
				\draw(5,-1.5) node {$\psi(x_i)$};
				\draw(7.5,-.5) node {$\phi_Y(U_i)$};
				\draw(11.5,-.5) node {$Z_2$};
				\draw(17,1)--(17,-1);
				\draw(15.5,.5) node {$\phi_Y(U_i)$};
				\draw(15.5,-.5) node {$\phi_Y(U_i)$};
				\draw(23,1)--(23,-1.5);
				\draw(20,.5) node {$\phi_Y(x_{i+1})$};
				\draw(20,-.5) node {$\phi_Y(x_{i+1})$};
				\draw(14.5,-1.5) node {$\psi(x_{i+1})$};
				\draw(25,.5) node {$Y_2$};
				\draw(25,-.5) node {$Y_2$};
			\end{tikzpicture}		
	\end{tabular}

	\caption{Replacing a section of $\phi_Y(x_i)$ in $Y$ to create $Z$.}  \label{fig:YtoZ}

\end{threeparttable}
\end{figure}
	
	
	\medskip
	Claim 1: $\displaystyle \liminf_{|Y| = n\rightarrow \infty} |C_Y| = \infty$.
	
	\medskip
	Since we want $2f(n) \leq |Z_1| < g(n) - 2f(n)$, and $g(n) - 2f(n) < |\phi_Y(x_i)| - |\phi_Y(U_i)|$, there are $g(n) - 4f(n)$ places to put the copy of $\phi_Y(U_i)$.
	To avoid any double-counting that might occur when some $Z$ and $Z'$ have their new copies of $\phi_Y(U_i)$ in overlapping locations, we further required that $f(n)$ divide $|Z_1|$. This produces the following lower bound: \[|C_Y| \geq \floor{\frac{g(n) - 4f(n)}{f(n)}} \rightarrow \infty.\]
	
	\medskip
	Claim 2: For distinct $Y, Y' \in A_n^\beta$, $C_Y \cap C_{Y'} = \emptyset$.
	
	\medskip
	To prove Claim 2, take $Y,Y' \in A_n^\beta$ with $Z \in C_Y \cap C_{Y'}$. Define $Y_1$, $Y_2$, $Y'_1$, and $Y'_2$ as above:
	\[ \begin{array}{l l l}
		Y_1 = \phi_Y(U_0x_1U_1x_2 \cdots U_{i-1}), & & Y_2 = \phi_Y(U_{i+1}x_{i+2} \cdots U_r);\\
		Y'_1 = \phi_{Y'}(U_0x_1U_1x_2 \cdots U_{i-1}), & & Y'_2 = \phi_{Y'}(U_{i+1}x_{i+2} \cdots U_r). 
	\end{array} \] 
	Now for some $Z_1,Z'_1,Z_2,Z'_2$,
	\[Y_1Z_1 \phi_Y(U_i) Z_2 \phi_Y(U_i) \phi_Y(x_{i+1}) Y_2 = Z = Y'_1Z'_1 \phi_{Y'}(U_i) Z'_2 \phi_{Y'}(U_i) \phi_{Y'}(x_{i+1}) Y'_2,\]
	with the following constraints:
	\begin{enumerate}[(i)]
		\item $|Y_1\phi_Y(U_i)| \leq |\phi_Y(UX_{i})| \leq f(n)$;
		\item $|Y'_1\phi_{Y'}(U_i)| \leq |\phi_{Y'}(UX_{i})| \leq f(n)$;
		\item $2f(n) \leq |Z_1| < g(n) - 2f(n)$;
		\item $2f(n) \leq |Z'_1| < g(n) - 2f(n)$;
		\item $|Z_1 \phi_Y(U_i) Z_2| = |\phi_{Y}(x_{i})| > g(n) - f(n)$;
		\item $|Z'_1 \phi_{Y'}(U_i) Z'_2| = |\phi_{Y'}(x_{i})| > g(n) - f(n)$.
	\end{enumerate}
	As a consequence:
	\begin{itemize}
		\item $|Y_1Z_1 \phi_Y(U_i)| < g(n) - f(n) < |Z'_1 \phi_{Y'}(U_i) Z'_2|$, by (i), (iii), and (vi);
		\item $|Y_1Z_1| \geq |Z_1| > 2f(n) > |Y'_1|$, by (iii) and (ii).
	\end{itemize}
	
	Therefore, the copy of $\phi_Y(U_i)$ added to $Z$ is properly within the noted occurrence of $Z'_1 \phi_{Y'}(U_i) Z'_2$ in $Z'$, which is in the place of $\phi_{Y'}(x_{i})$ in $Y'$.
	In particular, the added copy of $\phi_Y(U_i)$ in $Z$ interferes with neither $Y_1'$ nor the original copy of $\phi_{Y'}(U_i)$.
	Thus $Y_1'$ is an initial substring of $Y$ and $\phi_{Y'}(U_i) \phi_{Y'}(x_{i+1}) Y_2'$ is a final substring of $Y$.
	Likewise, $Y_1$ is an initial substring of $Y'$ and $\phi_Y(U_i) \phi_Y(x_{i+1}) Y_2$ is a final substring of $Y'$.
	By the selection process of $\phi_Y$ and $\phi_{Y'}$, we know that $Y_1 = Y'_1$ and 
\[ \phi_Y(U_i) \phi_Y(x_{i+1}) Y_2 = \phi_{Y'}(U_i) \phi_{Y'}(x_{i+1}) Y_2' . \]
	Finally, since $f(n)$ divides $Z_1$ and $Z_1'$, we deduce that $Z_1 = Z_1'$. 
	Otherwise, the added copies of $\phi_Y(U_i)$ in $Z$ and of $\phi_{Y'}(U_i)$ in $Z'$ would not overlap, resulting in a contradiction to the selection of $\phi_Y$ and $\phi_{Y'}$.
	Therefore, $Y = Y'$, concluding the proof of Claim 2.
	
	
	Now $C_Y \subset A_n^\alpha$ for $Y \in A_n^\beta$. 
	Claims 1 and 2 together imply that $|A_n^\beta| = o(|A_n^\alpha|)$.

\end{proof}

Observe that the choice of $\sqrt{n}$ in Lemma~\ref{base} was arbitrary. 
The proof works for any function $f(n) = o(n)$ with $f(n) \rightarrow \infty$.
Therefore, where Lemma~\ref{bulk} claims the existence of some $g(n) \rightarrow \infty$, the statement is in fact true for all $g(n) \rightarrow \infty$.

Let $\II_n(V,q)$ be defined as \[\II_n(V,q) =\frac{|\{W \in [q]^n \mid \phi(V)=W \mbox{ for some homomorphism } \phi:\L(V)^* \rightarrow [q]^*\}|}{q^{n}}.\]
Note that $\II_n(V,q)$ is equivalently defined as the probability that a uniformly randomly selected length-$n$ word over a fixed $q$-letter alphabet is an instance of $V$.
Indeed, by the nature of the instance relation, only the cardinality of the alphabet matters.

\begin{defn} \label{defn:sur}
	$\delta_{sur}(V,W)$ (with \emph{sur} for surjection) is the number of factors of $W$ that are instances of $V$ via a function $\phi$ with $\phi(V)=W$, divided by the total possible such factors (1). 
	More directly, $\delta_{sur}(V,W)$ is the characteristic function for the event that $W$ is an instance of $V$. 
\end{defn}

\begin{fact} [\cite{CR-15}, Fact 3.2] \label{EI}
	For any $V$ and $q$ and for $W_n \in [q]^n$ chosen uniformly at random, 
	\begin{eqnarray*}
		\binom{n+1}{2}\EE(\delta(V,W_n)) & = & \sum_{m=1}^n (n+1-m)\EE(\delta_{sur}(V,W_m))\\
		& = & \sum_{m=1}^n (n+1-m)\II_m(V,q).
	\end{eqnarray*}
\end{fact}

Set $\II(V,q) = \lim_{n\rightarrow \infty}\II_n(V,q)$. When does this limit exist?

\begin{thm} [\cite{CR-15}, Theorem 3.3] \label{nondoubledProb}
	For nondoubled $V$ and integer $q\in \Z^+$, $\II(V,q)$ exists. Moreover, $\II(V,q) \geq q^{-||V||} > 0$.
\end{thm}

\begin{proof}
	If $q=1$, then $\II_n(V,q)=1$ for $n \geq |V|$.
	
	Assume $q \geq 2$.
	Let $V = TxU$ where $x$ is the right-most nonrecurring letter in $V$. Let $\Gamma = \L(V)$ be the alphabet of letters in $V$. 
	By Lemma~\ref{bulk}, there is a nondecreasing function $g(n) = o(n)$ such that, for a randomly chosen $V$-instance $W \in [q]^n$, there is asymptotically almost surely a homomorphism $\phi: \Gamma^* \rightarrow [q]^*$ with $\phi(V) = W$ and $|\phi(x_r)| > n - g(n)$.
	
	Let $a_n$ be the number of $W \in [q]^n$ such that there exists $\phi: \Gamma^* \rightarrow [q]^*$ with $\phi(V) = W$ and $|\phi(x_r)| > n - g(n)$.
	Lemma~\ref{bulk} tells us that $\frac{a_n}{q^n} \sim \II_n(V,q)$.
	Note that $\frac{a_n}{q^n}$ is bounded. 
	It suffices to show that $a_{n+1} \geq qa_n$ for sufficiently large $n$.
	Pick $n$ so that $g(n) < \frac{n}{3}$.

	For length-$n$ $V$-instance $W$ counted by $a_n$, let $\phi_W$ be a homomorphism that maximizing $|\phi_W(x_r)|$ and, of such, minimizes $|\phi_W(T)|$.
	For each $\phi_W$ and each $a \in [q]$, let $\phi_W^a$ be the function such that, if $\phi_W(x_r) = AB$ with $|A|  = \floor{|\phi_W(x_r)|/2}$, then $\phi_W^a(x) = AaB$; $\phi_W^a(y) = \phi_W(y)$ for each $y \in  \Gamma \setminus\{x\}$
	Roughly speaking, we are sticking $a$ into the middle of the image of $x$.
	
	Suppose we are double-counting, so $\phi_W^a(V) = \phi_Y^b(V)$.
	As \[|\phi_W(x_r)|/2 > (n - g(n))/2 > n/3 > g(n) \geq |\phi_Y(TU)|\] and vice-versa, the inserted $a$ (resp., $b$) of one map does not appear in the image of $TU$ under the other map. 
	So $\phi_W(T)$ is an initial string and $\phi_W(U)$ a final string of $\phi_Y(V)$, and vice-versa.
	By the selection criteria of $\phi_W$ and $\phi_Y$, $|\phi_W(T)| = |\phi_Y(T)|$ and $|\phi_W(U)| = |\phi_Y(U)|$. 
	Therefore the location of the added $a$ in $\phi_W^a(V)$ and the added $b$ in $\phi_W^b(V)$ are the same.
	Hence, $a = b$ and $W = Y$.

	Moreover $\II(V,q) \geq q^{-||V||} > 0$.
\end{proof}

Having established that $\II(V,q)$ exists for all $V$ and $q$, we explore the limit value in Chapter~\ref{ASYMP}.
    
\begin{cor} [\cite{CR-15}, Corollary 3.6] \label{cor:EdI}
	Let $V$ be a nondoubled word on any alphabet. 
	Fix an integer $q>0$, and let $W_n \in [q]^n$ be chosen uniformly at random. 
	Then \[\lim_{n\rightarrow \infty}\EE(\delta(V,W_n)) = \II(V,q).\]
\end{cor}

\begin{proof}
	Let $\II = \II(V,q)$ and $\epsilon > 0$. 
	Pick $N = N_\epsilon$ sufficiently large so $|\II - \II_n(V,q)| < \frac{\epsilon}{2}$ when $n > N$. 
	Applying Fact~\ref{EI} for $n > \max(N,4N/\epsilon)$,
	\begin{eqnarray*}
		|\II - \EE(\delta(V,W_n))|& = & \left|\II\frac{1}{\binom{n+1}{2}}  \sum_{m=1}^n (n+1-m) - \frac{1}{\binom{n+1}{2}}  \sum_{m=1}^n (n+1-m)\II_m(V,q)\right|\\
		& \leq & \frac{1}{\binom{n+1}{2}}  \sum_{m=1}^n (n+1-m)|\II - \II_m(V,q)|\\
		& = & \frac{1}{\binom{n+1}{2}}  \left[\sum_{m=1}^N + \sum_{m = N+1}^n\right] (n+1-m)|\II - \II_m(V,q)|\\
		& < & \frac{1}{\binom{n+1}{2}}  \left[\sum_{m=1}^{\floor{\epsilon n / 4}} (n+1-m)1 + \sum_{m = N+1}^n (n+1-m)\frac{\epsilon}{2} \right]\\
		& < & \frac{1}{\binom{n+1}{2}}  \left[\frac{\epsilon n}{4}n + \binom{n+1}{2}\frac{\epsilon}{2} \right]\\
		& < & \epsilon.
	\end{eqnarray*}
\end{proof}

If there are multiple nonrecurring letters in $V$, then most long $V$-instances are liable to have numerous homomorphisms.
However, if there is exactly one recurring letter in $V$, Theorem~\ref{IV} below provides an upper bound for $\II(V,q)$ that, as $q \rightarrow \infty$, approaches the lower bound from Theorem~\ref{nondoubledProb} above.

\begin{lem} \label{homV}
	Let $V$ be a word with $\L(V) = \{x_0,x_1,\cdots,x_n\}$, $|\L(V)| = n+1$, where $x_0$ occurs $r_0=1$ time in $V$ and $x_k$ occurs $r_k>1$ times in $V$ for each $k \in [n]$.
	For $q,M \in \Z^+$, and $W_M \in [q]^M$ chosen uniformly at random,
	\[ 
		\EE(\hom(V,W_M)) = \sum_{\substack{\ang{i_0,\ldots,i_n} \in [M]^{n+1}: \\ M \geq  \sum_{k=0}^{n} i_k r_k }} \left(M + 1 - \sum_{k=0}^n i_k r_k\right) q^{\left( - \sum_{ k = 1}^n i_k(r_k - 1)\right)}.
	\]
\end{lem}

\begin{proof}
	For a given $W$, every encounter of $V$ in $W$ can be defined by the starting location $j$ of the substring and the lengths $\ang{i_k = |\phi(x_k)|}_{k = 0}^{n}$ of the letter-images under the homomorphism $\phi$.

	To compute $\EE(\hom(V,W_M))$ over random selection of $W_M \in [q]^M$, we sum over all possible $j$ and $\ang{i_k}_{k = 0}^{n}$ the probability that, for every $k \leq n$, the $r_k$ substrings of length $i_k$ (which are to be the instances of $x_k$) are identical.

	Our outside $(n+1)$-fold summation is over the possible lengths $i_k$, which are positive integers with $|\phi(V)| = \sum_{k = 0}^n i_k r_k \leq M$.
	This leaves $M + 1 - |\phi(V)|$ possible values for $r$, the starting location of the instance.

	For each $k$, only one of the $r_k$ instances of $x_k$ can consists of arbitrary letters and then the rest, with their $i_k(r_k - 1)$ letters, are determined.
	Thus, the probability of an encounters for given $r$ and $\ang{i_k}_{k = 0}^{n}$ is 
	\[ q^{\left( - \sum_{ k=0}^n i_k(r_k - 1)\right)} = q^{\left(- \sum_{ k=1}^n i_k(r_k - 1)\right)}. \]
\end{proof}

\begin{thm} \label{IV}
	Let $V$ be a word with $\L(V) = \{x_0,x_1,\cdots,x_n\}$, $|\L(V)| = n+1$, where $x_0$ occurs once in $V$ and $x_k$ occurs $r_k>1$ times in $V$ for each $k \in [n]$.
	Then for $q \geq 2$,
	\[ \underline{\delta}(V,q) \leq \II(V,q) \leq \prod_{k = 1}^{n} \frac{1}{q^{(r_k-1)} - 1}. \]
\end{thm}

\begin{proof}	
	For $\ang{i_1,\ldots,i_n} \in (\Z^+)^n$, let $M_\ell = M - \sum_{k>\ell} i_k r_k$ for $-1 \leq \ell \leq n$, so $M_n = M$ and $M_{\ell - 1} = M_\ell - i_\ell r_\ell$.
	Then Lemma~\ref{homV} says
	\[ 
		\EE(\hom(V,W_M)) = \sum_{\substack{\ang{i_0,\ldots,i_n} \in [M]^{n+1}: \\ M \geq  \sum_{k=0}^{n} i_k r_k }} \left(M_{-1} + 1\right) q^{\left( - \sum_{ k = 1}^n i_k(r_k - 1)\right)}.
	\]
	Since $M_0(M_0+1)$ is always nonnegative,
	\begin{eqnarray*}
		\EE(\hom(V,W_M)) & = & \sum_{\substack{\ang{i_0,\ldots,i_n} \in [M]^{n+1}: \\ M \geq  \sum_{k=0}^{n} i_k r_k }} \left(M_{-1} + 1\right) q^{\left( - \sum_{ k = 1}^n i_k(r_k - 1)\right)} \\
		& = & \sum_{\substack{\ang{i_1,\ldots,i_n} \in [M]^{n}: \\ M >  \sum_{k=0}^{n} i_k r_k }} \sum_{i_0=1}^{M_0} \left(M_0 - i_0 + 1\right) q^{\left( - \sum_{ k = 1}^n i_k(r_k - 1)\right)} \\
		& = & \sum_{\substack{\ang{i_1,\ldots,i_n} \in [M]^n: \\ M > \sum_{k=1}^{n} i_k r_k }} \frac{1}{2}M_0(M_0 + 1) q^{\left( - \sum_{ k = 1}^n i_k(r_k - 1)\right)} \\
		& \leq & \sum_{\ang{i_1,\ldots,i_n} \in (\Z^+)^n} \frac{1}{2}M_0(M_0 + 1) q^{\left( - \sum_{ k = 1}^n i_k(r_k - 1)\right)}.
	\end{eqnarray*}

	Claim: For $0 \leq \ell \leq n$,
	\begin{eqnarray*}
		&& \sum_{\ang{i_1,\ldots,i_n} \in (\Z^+)^n} \frac{1}{2}M_0(M_0 + 1) q^{\left( - \sum_{ k = 1}^n i_k(r_k - 1)\right)}\\
		&=& \sum_{\ang{i_{\ell + 1}, \ldots, i_n} \in (\Z^+)^{n-\ell}} \frac{1}{2}R_\ell(q,M_\ell)q^{\left( - \sum_{ k = \ell+1}^n i_k(r_k - 1)\right)} ,
	\end{eqnarray*}
	where $R_\ell(q,x) \in \R[x]$ is a quadratic polynomial with coefficients depending on $q$ and 
	\[ [x^2]\left(R_\ell(q,x)\right) = \frac{1}{q^{(r_\ell-1)}-1} \cdot [x^2]\left(R_{\ell-1}(q,x)\right) = \prod_{k = 1}^{\ell} \frac{1}{q^{(r_k-1)}-1}. \]

	We already know the claim to be true for $\ell = 0$ with $R_0(q,x) = x^2 + x$.
	We proceed in proving the full claim by induction on $\ell$. 
	Assume the claim holds for $\ell-1$ with $R_{\ell-1}(q,x) = ax^2 + bx+c$.
	\begin{eqnarray*}
		&&\sum_{\ang{i_{\ell}, \ldots, i_n} \in (\Z^+)^{n-\ell+1}} \frac{1}{2}R_{\ell-1}(q,M_{\ell-1})q^{\left( - \sum_{ k = \ell}^n i_k(r_k - 1)\right)}\\
		&=& \sum_{\ang{i_{\ell + 1}, \ldots, i_n} \in (\Z^+)^{n-\ell}} \sum_{i_\ell = 1}^{\infty} \frac{1}{2}R_{\ell-1}(q,M_{\ell} - i_\ell r_\ell)q^{\left( - \sum_{ k = \ell}^n i_k(r_k - 1)\right)} \\
		&=& \sum_{\ang{i_{\ell + 1}, \ldots, i_n} \in (\Z^+)^{n-\ell}}  \sum_{i_\ell = 1}^{\infty} \frac{1}{2}\left[a(M_{\ell} - i_\ell r_\ell)^2 + b(M_{\ell} - i_\ell r_\ell) + c \right]q^{\left( - \sum_{ k = \ell}^n i_k(r_k - 1)\right)}\\
		&=& \sum_{\ang{i_{\ell + 1}, \ldots, i_n} \in (\Z^+)^{n-\ell}} \frac{1}{2}q^{\left( - \sum_{ k = \ell+1}^n i_k(r_k - 1)\right)}\sum_{i = 1}^{\infty} \left[a' + b' i + c'i^2\right] \left(q^{(1-r_\ell)}\right)^i,
	\end{eqnarray*}
	where $a' = aM_\ell^2 + bM_\ell + c$, $b' =-2aM_\ell r_\ell -br_\ell$, and $c' = ar_\ell^2$.
	Since $q^{(1-r_\ell)} \in (0,1)$, we have for some $d_1$ and $d_2$ dependent on $q$ and $r_\ell$:
	\begin{eqnarray*}
		\sum_{i = 1}^{\infty} \left(q^{(1-r_\ell)}\right)^i & = & \frac{1}{q^{(r_\ell - 1)} - 1} ; \\
		\sum_{i = 1}^{\infty} i \left(q^{(1-r_\ell)}\right)^i & = & d_1 ; \\
		\sum_{i = 1}^{\infty} i^2\left(q^{(1-r_\ell)}\right)^i & = & d_2 .
	\end{eqnarray*}
	We complete the proof of the claim with 
	\begin{eqnarray*}
		R_\ell(q,M_\ell) & = & a'\frac{1}{q^{(r_\ell - 1)} - 1} + b'd_1 + c'd_2 \\
		& = & (aM_\ell^2 + bM_\ell + c)\frac{1}{q^{(r_\ell - 1)} - 1} + (-2aM_\ell r_\ell -br_\ell)d_1 + (ar_\ell^2)d_2 \\
		& = & \left[a\frac{1}{q^{(r_\ell - 1)} - 1} \right]M_\ell^2 + \left[b\frac{1}{q^{(r_\ell - 1)} - 1}  - 2ar_\ell d_1 \right] M_\ell \\
		& &  + \left[c\frac{1}{q^{(r_\ell - 1)} - 1} -br_\ell d_1 + ar_\ell^2d_2 \right].
	\end{eqnarray*}

	To complete the proof of the theorem, apply the claim to $\ell = n$ and let $M \rightarrow \infty$.
	\begin{eqnarray*}
		 \EE(\hom(V,W_M)) & = & \sum_{\ang{i_1,\ldots,i_n} \in (\Z^+)^n} \frac{1}{2}M_0(M_0 + 1) q^{\left( - \sum_{ k = 1}^n i_k(r_k - 1)\right)}\\
		& \leq & \frac{1}{2}R_n(q,M_n)q^{\left( - \sum_{ k \in \emptyset} i_k(r_k - 1)\right)} \\
		& = & \frac{1}{2}R_n(q,M)q^{\left( - 0\right)} \\
		& \sim &  \frac{1}{2}M^2 \prod_{k = 1}^{n} \frac{1}{q^{(r_k - 1)}-1}.
	\end{eqnarray*}
	Therefore,
	\begin{eqnarray*}
		\II(V,q) & = & \lim_{M \rightarrow \infty} \EE(\delta(V,W_M))\\
		& \leq & \lim_{M \rightarrow \infty} \frac{1}{\binom{M+1}{2}} \EE(\hom(V,W_M)) \\
		& = & \prod_{k = 1}^{n} \frac{1}{q^{(r_k - 1)}-1}.
	\end{eqnarray*}
\end{proof}

\section{Density of Doubled Words}

Our main dichotomy says that the average density of a doubled word in large random words (over a fixed alphabet with at least two letters) goes to 0. 
Thus the expected number of instances in a random word of length $n$ is $o(n^2)$.
Perhaps we can find lower-order asymptotics for the expected number of instances of a doubled word. 
Hencefore, if $\binom{x}{y}$ is used with nonintegral $x$, we mean \[\binom{x}{y} = \frac{\prod_{i = 0}^{y-1} (x-i)}{y!}.\]

\begin{prn} [\cite{CR-15}, Proposition 4.1] \label{splice}
	For $k \in \Z^+$ and $\overline{r} = (r_1, \ldots, r_k) \in (\Z^+)^k$, let $a_n(\overline{r})$ be the number of $k$-tuples $\overline{a} = (a_1, \cdots, a_k) \in (\Z^+)^k$ so that $\sum_{i=1}^{k} a_ir_i = n$. 
	Then $a_{n}(\overline{r}) \leq \binom{n/d+k +1}{k +1}$, where $d = \gcd_{i \in [k]}(r_i)$.
\end{prn}

\begin{proof}
	If $d \! \not\!| \; n$, then $a_{n}(\overline{r}) = 0$. 
	Otherwise, for each $\overline{a}$ counted by $a_{n}(\overline{r})$, there is a unique corresponding $\overline{b}\in (\Z^+)^k$ such that $1 \leq b_1 < b_2 < \cdots < b_k = n/d$ and $b_j = \frac{1}{d} \sum_{i = 1}^{j} a_ir_i$. 
	The number of strictly increasing $k$-tuples of positive integers with largest value $n/d$ is $\binom{n/d+k +1}{k +1}$.
\end{proof}

Fix integer $q>0$. The number of instances of $V$ in $[q]^n$ is $q^n \II_n(V,q)$. 
Assume $V$ is doubled. 
Let $\Gamma = \L(V) = \{x_1, \ldots, x_k\}$ and $r_i$ be the multiplicity of $x_i$ in $V$ for each $i \in [k]$. 
Let $d = \gcd_{i \in [k]}(r_i)$ and $r = \min_{i \in [k]}(r_i)$.
Note that $\II_n(V,q) = 0$ when $d \!\not| \;n$.
But perhaps \[\lim_{\substack{n \rightarrow \infty \\ d \mid n}}\frac{q^{n}}{f(n)}\II_{n}(V,q)\]
exists for some function $f$ that only depends on $q$ and $V$.
For inspiration, note that $q^{n}\II_{n}(U^m,q) = q^{n/m} \II_{n/m}(U,\Sigma)$ when $m \mid n$.
Furthermore, using Proposition~\ref{splice},
\begin{equation} \label{instanceCount}
	q^{n}\II_{n}(V,q) \leq  \EE(\hom(V,W_{n})) < \binom{n/d+k +1}{k +1} q^{n/r}.
\end{equation}

Now select some letter $x$ of multiplicity $r$ and let $U$ be $V$ with all copies of $x$ removed. 
When $r | (n - |U|)$, we can get a lower bound on the number of instances by counting homomorphism $\phi$ with $|\phi(U)| = |U| = |V|-r$:
\begin{equation} \label{lower}
	q^{n}\II_{n}(V,q) \geq q^{(n - |U|)/r + (k-1)} = (q^{k-|V|/r})q^{n/r}.
\end{equation}

\begin{conj} [\cite{CR-15}, Conjecture 4.2] For $q \in \Z^+$, the following limit exists:
 	\[\lim_{\substack{n \rightarrow \infty \\ d \mid n}}q^{n(1 - 1/r)}\II_{n}(V,q).\]
\end{conj}

By \eqref{lower}, the limit (if it exists) cannot be 0. 
Theorem~\ref{nondoubledProb} is a special case of this conjecture, with $d = r = 1$. 
%

\section{Concentration} \label{Concentration}

For doubled $V$ and $q\geq 2$, we established that the expectation of the density $\delta(V,W_n)$ converges to zero.
What is the concentration of the distribution of this density?
By \eqref{instanceCount}, we can bound the probability that randomly chosen $W_n \in [q]^n$ is a $V$-instance: \[\PP(\delta_{sur}(V,W_n) = 1) = \II_{n}(V,q) \leq \binom{n/d+k +1}{k +1}q^{n(1-r)/r}.\] From this observation we get the following probabilistic result (which is only interesting for $q,r>1$).

\begin{lem} [\cite{CR-15}, Lemma 5.1] \label{ProbLem}
	Let $V$ be a word with $k$ distinct letters, each occurring at least $r \in \Z^+$ times. 
	Let $W_n \in [q]^n$ be chosen uniformly at random.
	Recall that $\binom{n+1}{2}\delta(V,W_n)$ is the number substrings of $W_n$ that are $V$-instances.
	Then for any nondecreasing function $f(n) > 0$, \[\PP\left(\binom{n+1}{2}\delta(V,W_n) > n\cdot f(n) \right) < n^{k+3}q^{f(n)(1-r)/r}.\]
\end{lem}

\begin{proof}
    Since $\delta_{sur}(V,W) \in \{0,1\}$,
	\begin{eqnarray*}
	    \sum_{m=1}^{\floor{f(n)}} \sum_{\ell=0}^{n-m}\delta_{sur}(V,W_n[\ell,\ell+m]) & < & n \cdot f(n).
	\end{eqnarray*}
	Therefore,
	\begin{eqnarray*}
		\PP\left(\binom{n+1}{2}\delta(V,W_n) > n\cdot f(n) \right) & = & \PP\left(\sum_{m=1}^n \sum_{\ell=0}^{n-m}\delta_{sur}(V,W_n[\ell,\ell+m]) > n\cdot f(n) \right) \\
		& < & \PP\left(\sum_{m=\ceil{f(n)}}^n \sum_{\ell=0}^{n-m}\delta_{sur}(V,W_n[\ell,\ell+m])> 0 \right) \\
		& < & \sum_{m=\ceil{f(n)}}^n \sum_{\ell=0}^{n-m}\PP\left(\delta_{sur}(V,W_n[\ell,\ell+m])> 0 \right) \\
		& = & \sum_{m=\ceil{f(n)}}^n (n-m+1) \PP\left(\delta_{sur}(V,W_m) = 1 \right) \\
		& \leq & \sum_{m=\ceil{f(n)}}^n (n-m+1)\binom{m/d+k +1}{k +1} q^{m(1-r)/r}\\
		& < & n^2 \binom{n/d+k +1}{k +1} q^{f(n)(1-r)/r}\\
		& < & n^{k+3}q^{f(n)(1-r)/r}.
	\end{eqnarray*}
\end{proof}

\begin{thm} [\cite{CR-15}, Theorem 5.2] \label{moments}
	Let $V$ be a doubled word, $q \geq 2$, and $W_n \in [q]^n$ chosen uniformly at random. 
	Then for $p \in \Z^+$, the $p$-th raw moment and the $p$-th central moment of $\delta(V,W_n)$ are both $O\left(\left(\log(n)/n\right)^p\right)$.
\end{thm}

\begin{proof}
	Let us use Lemma~\ref{ProbLem} to first bound the $p$-th raw moments for $\delta(V,W_n)$, assuming $r\geq 2$.
	To minimize our bound, we define the following function on $n$, which acts as a threshold for ``short'' substrings of a random length-$n$ word: \[s_p(n) = \frac{r}{1-r}\log_q(n^{-(k+5+p)}) = t_p \log_q n,\] where $t_p = \frac{r(k+5+p)}{r-1} > 0$.
	\begin{eqnarray*}
		\EE(\delta(V,W_n)^p) & = & \sum_{i = 0}^{\binom{n+1}{2}} \PP\left(\delta(V,W_n) = \frac{i}{\binom{n+1}{2}}\right)\left(\frac{i}{\binom{n+1}{2}}\right)^p\\
		& < & \sum_{i = 0}^{\floor{n\cdot s_p(n)}} \PP\left(\delta(V,W_n) = \frac{i}{\binom{n+1}{2}}\right)\left(\frac{i}{\binom{n+1}{2}}\right)^p \\&&+ \sum_{i = \ceil{n\cdot s_p(n)}}^{\binom{n+1}{2}} n^{k+3}q^{s_p(n)(1-r)/r}\left(\frac{i}{\binom{n+1}{2}}\right)^p\\
		& < & \left(\frac{n\cdot s_p(n)}{\binom{n+1}{2}}\right)^p +  n^{k+5 }q^{s_p(n)(1-r)/r}(1)^p\\
		& = & \left(\frac{nt_p \log_q n}{\binom{n+1}{2}}\right)^p +  n^{k+5 }q^{\log_q\left(n^{-(k+5+p)}\right)}\\
		&= &O_p\left(\left(\frac{\log n}{n}\right)^p\right).
	\end{eqnarray*}
	
	Setting $p=1$, $\EE_n= \EE(\delta(V,W_n)) < (c\log n)/n$ for some large $c$. We use this upper bound on the expectation (1st raw moment) to bound the central moments. 
	\begin{eqnarray*}
		\EE(\left|\delta(V,W_n) - \EE_n\right|^p) & = &  \sum_{i = 0}^{\binom{n+1}{2}} \PP\left(\delta(V,W_n) = \frac{i}{\binom{n+1}{2}}\right)\left|\frac{i}{\binom{n+1}{2}} - \EE_n \right|^p\\
		& < &  \sum_{i = 0}^{\floor{n\cdot s_p(n)}}  \PP\left(\delta(V,W_n) = \frac{i}{\binom{n+1}{2}}\right)\left(\frac{c\log n}{n}\right)^p \\
			& & + \sum_{i = \ceil{ns_p(n)}}^{\binom{n+1}{2}} \PP\left(\delta(V,W_n) = \frac{i}{\binom{n+1}{2}}\right)\left(1\right)^p\\
		& < &\left(\frac{c\log n}{n}\right)^p + n^{k+5}q^{s_p(n)(1-r)/r}\\
		&= &O_p\left(\left(\frac{\log n}{n}\right)^p\right).
	\end{eqnarray*}
\end{proof}

\begin{cor} [\cite{CR-15}, Corollary 5.3] \label{expectation}
	Let $V$ be a doubled word, $q \geq 2$, and $W_n \in [q]^n$ chosen uniformly at random. 
	Then \[\frac{1}{n} \ll \EE(\delta(V,W_n)) \ll \frac{\log n}{n}.\]
\end{cor}

\begin{proof}
	The upper bound was stated explicitly in the proof of Theorem~\ref{moments}. 
	The lower bound follows from an observation in Section~\ref{dense}: ``the event that $W_n[b|V|,(b+1)|V|]$ is an instance of $V$ has nonzero probability and is independent for distinct $b \in \N$.''
	Hence \[ \EE(\delta(V,W_n)) \geq \frac{1}{\binom{n+1}{2}}\floor{\frac{n}{|V|}}\II_{|V|}(V,q) = \Omega(n^{-1}) . \]
\end{proof}

The bound that Theorem~\ref{moments} gives on the variance (2nd central moment) is not very interesting. However, we obtain nontrivial concentration using covariance and the fact that most ``short'' substrings in a word do not overlap. 

\begin{thm} [\cite{CR-15}, Theorem 5.4] \label{variance}
	Let $V$ be a doubled word, $q \geq 2$, and $W_n \in [q]^n$ chosen uniformly at random. 
	\[ \Var(\delta(V,W_n)) =O\left(\EE(\delta(V,W_n))^2 \frac{(\log n)^3}{n}\right). \]
\end{thm}

\begin{proof}
	Let $X_n = \binom{n+1}{2}\delta(V,W_n)$ be the random variable counting the number of substrings of $W_n$ that are $V$-instances. 
	For fixed $n$, let $X_{a,b}$ be the indicator variable for the event that $W_n[a,b]$ is a $V$-instance, so $X_n = \sum_{a=0}^{n-1} \sum_{b=a+1}^{n} X_{a,b}$. 
	We use $(a,b) \sim (c,d)$ to denote that $[a,b]$ and $[c,d]$ overlap.
	Note that 
	\begin{eqnarray*}
		\Cov(X_{a,b},X_{c,d})& \leq & \EE(X_{a,b}X_{c,d})\\
		& \leq & \min(\EE(X_{a,b}),\EE(X_{c,d})) \\
		& = & \min(\II_{(b-a)}(V,q),\II_{(d-c)}(V,q)) , 
	\end{eqnarray*}
	and for $i \in \{b-a,d-c\}$,
	\[ \min(\II_{(b-a)}(V,q),\II_{(d-c)}(V,q)) \leq  \binom{i/d+k+1}{k+1}q^{i(1-r)/r} . \]
	For $i < n/3$, the number of intervals in $W_n$ of length at most $i$ that overlap a fixed interval of length $i$ is less than $\binom{3i}{2}$.
	Let $s(n)=s_0(n) = t_0\log_q n$ as defined in Theorem~\ref{moments}.
	For sufficiently large $n$,
    \begin{eqnarray*}
    	\Var(X_n) & = & \sum_{\substack{0 \leq a < b \leq n \\ 0 \leq c < d \leq n}} \Cov(X_{a,b},X_{c,d})\\
    	& \leq & \sum_{(a,b) \sim (c,d)} \min(\II_{(b-a)}(V,q),\II_{(b-a)}(V,q))  \\
    	& = & \left[ \sum_{\substack{(a,b) \sim (c,d) \\ b-a,d-c \leq s(n)}} + \sum_{\substack{(a,b) \sim (c,d) \\ else}} \right] \min(\II_{(b-a)}(V,q),\II_{(b-a)}(V,q)) \\
    	& < & 2\sum_{i = 1}^{\floor{s(n)}}(n+1-i)\binom{3i}{2}\cdot 1 \\
    	& &  + \sum_{i = \ceil{s(n)}}^{n} (n+1-i) \binom{n+1}{2}\cdot \binom{i/d+k+1}{k+1}q^{i(1-r)/r} \\
    	& < &2s(n)n(3s(n))^2 + nnn^2n^{k+1}q^{s(n)(1-r)/r}\\
    	& = &18(t_0 \log_q n)^3n + n^{5+k} q^{\log_q\left(n^{-(k+5)}\right)}\\\
    	& = & O(n(\log n)^3).
    \end{eqnarray*}

    Since $\EE(\delta(V,W_n)) = \Omega(n^{-1})$ by Corollary~\ref{expectation},
    \begin{eqnarray*}
    	\Var(\delta(V,W_n)) & = & \Var\left(\frac{X_n}{\binom{n+1}{2}}\right)\\
    	& = & \frac{\Var(X_n)}{\binom{n+1}{2}^2}\\
    	& = & O\left(\frac{(\log n)^3}{n^3}\right)\\
    	& = &O\left(\EE(\delta(V,W_n))^2 \frac{(\log n)^3}{n}\right).
    \end{eqnarray*}

\end{proof}

\begin{qun} [\cite{CR-15}, Question 5.5]
For nondoubled word $V$, what is the concentration of the density distribution of $V$ in random words?
\end{qun}

%% file: Asymp.tex
%
%
%
%
%
%
%
%
%
%
%

\chapter{Asymptotic Probability of Being Zimin} \label{ASYMP}

In Chapter~\ref{AVOID}, we investigated bounds on the length of words that avoid Zimin words.
In subsequent chapters, we proceeded to develop the theory of word densities, some of which applies to Zimin words.

We proved in Chapter~\ref{DICHOT} that the asymptotic instance probability of $V$ in $q$-ary words, $\II(V,q) = \lim_{n \rightarrow \infty} \II_n(V,q)$, exists for any word $V$, and is equal to the asymptotic expected density of $V$ in random words.
We also proved the following dichotomy for $q \geq 2$ (Theorem~\ref{dichotomy}): $\II(V,q) = 0$ if and only if $V$ is doubled (that is, every letter in $V$ occurs at least twice).
Trivially, if $V$ is composed of $k$ distinct, nonrecurring letters, then $\II_n(V,[q])=1$ for $n\geq k$, so $\II(V,q) = 1$.
But if $V$ contains at least one recurring letter, it becomes a nontrivial task to compute $\II(V,q)$.

\begin{cor}
	For $n,q \in \Z^+$, 
	\[ q^{-2^n + n + 1} \leq \II(Z_n,q) \leq \prod_{j = 1}^{n-1} \frac{1}{q^{(2^j-1)} - 1}. \]
\end{cor}

\begin{proof}
For the lower bound, note that $||Z_n|| = |Z_n| - |\L(Z_n)| = (2^n-1) - (n)$. 
Theorem~\ref{nondoubledProb} tells us that for all $q \in \Z^+$ and nondoubled $V$, $\II(V,q) \geq q^{-||V||}$.

For the upper bound, observe that the $n$ letters occurring in $Z_n$ have multiplicities $\ang{r_j = 2^j : 0 \leq j < n}$. 
Since there is exactly one nonrecurring letter in $Z_n$, $r_0 = 2^0 = 1$, Theorem~\ref{IV} provides an upper bound of $\prod_{j = 1}^{n-1} \frac{1}{q^{(r_j-1)} - 1}$.
\end{proof}

A nice property of these bounds is that they are asymptotically equivalent as $q \rightarrow \infty$.
For some specific $V$, we can do better.
In this chapter, we provide infinite series for computing the asymptotic instance probability $\II(V,q)$ for two Zimin words, $V = Z_2 = aba$ (Section~\ref{IZ2}) and $V = Z_3 = abacaba$ (Section~\ref{IZ3}).
Table~\ref{table:Z2Z3} below gives numerical approximations for $2 \leq q \leq 6$. 
Our method also provides bounds on  $\II(Z_n,q)$ for general $n$ (Section~\ref{IZn}).

\begin{table}[ht]
\centering
\begin{threeparttable}

    	\caption{Approximate values of $\II(Z_2,q)$ and $\II(Z_3,q)$ for $2 \leq q \leq 6$.} \label{table:Z2Z3}

    	\begin{tabular}{c}
	$\begin{array}{c | c | c | c | c | c | c }
    		q & 2 & 3 & 4 & 5 & 6 & \cdots \\ \hline
        		\II(Z_2,q) & 0.7322132 & 0.4430202 & 0.3122520 & 0.2399355 & 0.1944229 & \cdots\\ \hline
    		\II(Z_3,q) & 0.1194437 & 0.0183514 & 0.0051925 & 0.0019974 & 0.0009253 & \cdots\\
	\end{array}$
    	\end{tabular}
%

\end{threeparttable}
\end{table}


\section{Calculating \texorpdfstring{$\II(Z_2,q)$}{the asymptotic instance probability of Z2}} \label{IZ2}


Let $a_\ell = a_\ell^{(q)}$ be the number of bifix-free $q$-ary strings of length $\ell$. For $q=2$, this is sequence oeis.org/A003000; for $q=3$, oeis.org/A019308 \parencite{OEIS}.


\begin{lem}
	If word $W$ has a bifix, then it has a bifix of length at most $\lfloor |W|/2 \rfloor$.
\end{lem}

\begin{proof}
	Let $W$ be a word with minimal-length bifix of length $k$, $\lfloor |W|/2 \rfloor < k < |W|$. Then we can write $W = W_1W_2W_3$ where $W_1W_2 = W_2W_3$ and $|W_1W_2| = k = |W_2W_3|$. But then $W$ has bifix $W_2$ with $|W_2| < k$, which contradicts our selection of the shortest bifix of $W$.
\end{proof}

\begin{lem} $a_\ell = a_\ell^{(q)}$ has the following recursive definition:
	\begin{eqnarray*}
		a_0 & = & 0;\\
		a_1 & = & q;\\
		a_{2k} & = & qa_{2k-1} - a_k;\\
		a_{2k+1} & = & qa_{2k}.
	\end{eqnarray*}
\end{lem}

\begin{proof}
	Fix a $q$-letter alphabet. 
	Let $W = UV$ be a bifix-free word with $|U| = \ceil{\frac{|W|}{2}}$ and $|V| = \floor{\frac{|W|}{2}}$. 
	Suppose $UaV$ has a bifix for some letter $a$. 
	Then by the lemma, $UaV$ has a bifix is of length at most $|UaV|/2$. 
	But $W$ is bifix free, so the only possibility is $U = aV$.

	Therefore, for every bifix-free word of length $2k$ there are $q$ bifix-free words of length $2k+1$. 
	For every bifix-free word of length $2k-1$, there are $q$ bifix-free words of length $2k$, with exception of the the length-$2k$ words that are the square of a bifix-free word of length $k$.
\end{proof}

\begin{thm} \label{thm:IZ2}
	For $q \geq 2$,
	\begin{eqnarray*}
		\II(Z_2,q) &=& \sum_{j=0}^{\infty} \frac{(-1)^jq^{\left(1-2^{j+1}\right)}}{\prod_{k=0}^{j} \left(1 - q^{\left(1-2^{k+1}\right)}\right)}.
	\end{eqnarray*}
\end{thm}

\begin{proof}
	Since $a_\ell = a_\ell^{(q)}$ counts bifix-free words, the number of $q$-ary words of length $M$ that are $Z_2$-instances is (without double-count)
	\[\sum_{\ell=0}^{\lceil M/2 \rceil -1} a_\ell q^{M - 2\ell},\]
	so the proportion of $q$-ary words of length $M$ that are $Z_2$-instances is
	\[\frac{1}{q^M} \sum_{\ell=0}^{\lceil M/2 \rceil -1} a_\ell q^{M - 2\ell} = \sum_{\ell=0}^{\lceil M/2 \rceil -1} \frac{a_\ell}{q^{2\ell}}.\]
	Therefore $\II(Z_2,q)  = f(1/q^2)$, where $f(x) = f^{(q)}(x)$ is the generating function for $\{a_\ell\}_{\ell = 0}^\infty$:	\[f(x) = \sum_{\ell=0}^{\infty} a_\ell x^\ell.\]
	From the recursive definition of $a_\ell$, we obtain the functional equation
	\begin{eqnarray} \label{Z2func}
		f(x) = qx + qxf(x) - f(x^2).
	\end{eqnarray}
	Solving \eqref{Z2func} for $f(x)$ gives \[f(x) = \frac{qx - f(x^2)}{1-qx} = \cdots = \sum_{j=0}^{\infty} \frac{(-1)^jqx^{2^j}}{\prod_{k=0}^{j} (1 - qx^{2^k})}.\]
\end{proof}

\begin{cor} \label{cor:IZ2}
	For $q \geq 2$:
	\[\frac{1}{q} < \II(Z_2,q) < \frac{1}{q-1}.\]
	Moreover, as $q \rightarrow \infty$,
	\[\II(Z_2,q) = \frac{1}{q-1} - \frac{1+o(1)}{q^3}.\]
\end{cor}

\begin{proof}
	The lower bound follows from the fact that a word of length $M>2$ is a $Z_2$-instance when the first and last character are the same. 
	This occurrence has probability $1/q$. 
	Note that $f^{(q)}(q^{-2})$ is an alternating series.
	Moreover, the terms in absolute value are monotonically approaching 0; the routine proof of monotonicity can be found in the appendices  (Lemma~\ref{lemF}).
	Hence, the partial sums provide successively better upper and lower bounds: 
	\begin{eqnarray*}
		 f^{(q)}\left(\frac{1}{q^2}\right) & = & \sum_{j=0}^{\infty} \frac{(-1)^j\left(q^{1-2^{j+1}}\right)}{\prod_{k=0}^{j} \left(1 - \left(q^{1-2^{k+1}}\right)\right)};\\
		\\
		f^{(q)}\left(\frac{1}{q^2}\right) &>& \sum_{j=0}^{1} \frac{(-1)^j\left(q^{1-2^{j+1}}\right)}{\prod_{k=0}^{j} \left(1 - \left(q^{1-2^{k+1}}\right)\right)}\\
		& = &  \frac{1/q}{1-1/q} - \frac{1/q^3}{(1 - 1/q)(1 - 1/q^3)}\\
		& = & \frac{1}{q-1} - \frac{1+o(1)}{q^3};\\
		\\
		f^{(q)}\left(\frac{1}{q^2}\right) & < & \sum_{j=0}^{2} \frac{(-1)^jq\left(\frac{1}{q^2}\right)^{2^j}}{\prod_{k=0}^{j} \left(1 - q\left(\frac{1}{q^2}\right)^{2^k}\right)}\\
		& = &\frac{1}{q-1} -\frac{1+o(1)}{q^3} + \frac{1/q^5}{(1 - 1/q)(1 - 1/q^3)(1 - 1/q^5)}\\
		& = &\frac{1}{q-1} -\frac{1+o(1)}{q^3} +\frac{O(1)}{q^5}.
	\end{eqnarray*}
\end{proof}

\begin{table}[ht]
\centering
\begin{threeparttable}

	\caption{Approximate values of $\II(Z_2,q)$ for $2 \leq q \leq 8$.}

	\def\arraystretch{1.3}
	\begin{tabular}{c | c c c c c c c c c}

			$q$ & 2 & 3 & 4 & 5 & 6 & 7 & 8 \\ \hline 
			$q^{-1}$ & 0.50000 & .33333 & .25000 & .20000 & .16667 & .14286 & .12500\\ \hline
			$\II(Z_2,q)$ & 0.73221 & .44302 & .31225 & .23994 & .19442 & .16326 & .14062\\ \hline
			$(q-1)^{-1} - q^{-3}$ & 0.87500 & .46296 & .31771 & .24200 & .19537 &.16375 &.14090\\ \hline
			$(q-1)^{-1}$ & 1.00000 & .50000 & .33333 & .25000 & .20000 & .16667 & .14286\\
	\end{tabular}

\end{threeparttable}
\end{table}


%
%
%

\section{Calculating \texorpdfstring{$\II(Z_3,q)$}{the asymptotic instance probability of Z3}} \label{IZ3}

Will use similar methods to compute $\II(Z_3,q)$. 
To avoid unnecessary subscripts and superscripts, assume throughout this section that we are using a fixed alphabet with $q>1$ letters, unless explicitly stated otherwise.
Since $Z_2$ has more interesting structure than $Z_1$, there are more cases to consider in developing the necessary recursion. 

\begin{lem}
\label{P3cases}
	Fix bifix-free word $L$. 
	Let $W = LAL$ be a $Z_2$-instance with a $Z_2$-bifix. Then $LAL$ can be written in exactly one of the following ways:
	\begin{enumerate}[$\<$i$\>$]
		\item $LAL = LBLCLBL$ with $LBL$ the shortest $Z_2$-bifix of $W$ and $|C|>0$;
		\item $LAL = LBLLBL$ with $LBL$ the shortest $Z_2$-bifix of $W$;
		\item $LAL = LBLBL$ with $LBL$ the shortest $Z_2$-bifix of $W$;
		\item $LAL = LLFLLFLL$ with $LLFLL$ the shortest $Z_2$-bifix of $W$;
		\item $LAL = LLLL$.
	\end{enumerate}
\end{lem}

\begin{proof}
	With some thought, the reader should recognize that the five listed cases are in fact mutually exclusive. 
The proof that these are the only possibilities follows.

	Given that $W$ has a $Z_2$-bifix and $L$ is bifix-free, it follows that $W$ has a $Z_2$-bifix $LBL$ for some nonempty $B$. 
Let $LBL$ be chosen of minimal length. We break this proof into nine cases depending on the lengths of $L$ and $LBL$ (Figure~\ref{overlap}). 
Set $m = |W|$, $\ell = |L|$, and $k = |LBL|$.

	\begin{figure}[ht] 
	\centering
	\begin{threeparttable}

		\begin{tabular}{| c | c | c |} \hline

		\begin{tikzpicture}[scale=.85]
			\draw (0,0) rectangle node{$W$} (4,.6);
			\draw (2,1.7) node[above]{Case (1) $\rightarrow \ang{i}$};
			\draw (2,1.2) node[above]{$2k < m$};
		
			\draw (0,1.2) rectangle node{$B$} (1.5,.6);
			\draw (0,.6) rectangle node{$L$} (.5,1.2);
			\draw (1,1.2) rectangle node{$L$} (1.5,.6);
		
			\draw (4,0) rectangle node{$B$} (2.5,-.6);
			\draw (3,-.6) rectangle node{$L$} (2.5,0);
			\draw (4,0) rectangle node{$L$} (3.5,-.6);
		
		\end{tikzpicture}
& 
		\begin{tikzpicture}[scale=.85]
				\draw (0,0) rectangle node{$W$} (4,.6);
				\draw (2,1.7) node[above]{Case (2) $\rightarrow \ang{ii}$};
				\draw (2,1.2) node[above]{$2k = m$};
			
				\draw (0,1.2) rectangle node{$B$} (2,.6);
				\draw (0,.6) rectangle node{$L$} (.5,1.2);
				\draw (2,1.2) rectangle node{$L$} (1.5,.6);
			
				\draw (4,0) rectangle node{$B$} (2,-.6);
				\draw (2,-.6) rectangle node{$L$} (2.5,0);
				\draw (4,0) rectangle node{$L$} (3.5,-.6);
		\end{tikzpicture}
&		
		\begin{tikzpicture}[scale=.85]
				\draw (0,0) rectangle node{$W$} (4,.6);
				\draw (2,1.7) node[above]{Case (3) $\rightarrow \; \Rightarrow\!\Leftarrow$};
				\draw (2,1.2) node[above]{$m < 2k < m + \ell$};
			
				\draw (0,1.2) rectangle node{$B$} (2.1,.6);
				\draw (0,.6) rectangle node{$L$} (.5,1.2);
				\draw (2.1,1.2) rectangle node{$L$} (1.6,.6);
			
				\draw (4,0) rectangle node{$B$} (1.9,-.6);
				\draw (1.9,-.6) rectangle node{$L$} (2.4,0);
				\draw (4,0) rectangle node{$L$} (3.5,-.6);
		\end{tikzpicture}
\\ \hline
		\begin{tikzpicture}[scale=.85]
				\draw (0,0) rectangle node{$W$} (4,.6);
				\draw (2,1.7) node[above]{Case (4) $\rightarrow \ang{iii}$};
				\draw (2,1.2) node[above]{$2k = m + \ell$};
			
				\draw (0,1.2) rectangle node{$B$} (2.25,.6);
				\draw (0,.6) rectangle node{$L$} (.5,1.2);
				\draw (2.25,1.2) rectangle node{$L$} (1.75,.6);
			
				\draw (4,0) rectangle node{$B$} (1.75,-.6);
				\draw (1.75,-.6) rectangle node{$L$} (2.25,0);
				\draw (4,0) rectangle node{$L$} (3.5,-.6);
		\end{tikzpicture}
&		
		\begin{tikzpicture}[scale=.85]
				\draw (0,0) rectangle node{$W$} (4,.6);
				\draw (2,1.7) node[above]{Case (5) $\rightarrow \; \Rightarrow\!\Leftarrow$};
				\draw (2,1.2) node[above]{$m + \ell < 2k < m + 2\ell$};
			
				\draw (0,1.2) rectangle node{$B$} (2.35,.6);
				\draw (0,.6) rectangle node{$L$} (.5,1.2);
				\draw (2.35,1.2) rectangle node{$L$} (1.85,.6);
			
				\draw (4,0) rectangle node{$B$} (1.65,-.6);
				\draw (1.65,-.6) rectangle node{$L$} (2.15,0);
				\draw (4,0) rectangle node{$L$} (3.5,-.6);
		\end{tikzpicture}
&		
		\begin{tikzpicture}[scale=.85]
				\draw (0,0) rectangle node{$W$} (4,.6);
				\draw (2,1.7) node[above]{Case (6) $\rightarrow \ang{iv}/\Rightarrow\!\Leftarrow$};
				\draw (2,1.2) node[above]{$m + 2\ell = 2k < 2(m - \ell)$};
			
				\draw (0,1.2) rectangle node{$B$} (2.5,.6);
				\draw (0,.6) rectangle node{$L$} (.5,1.2);
				\draw (2.5,1.2) rectangle node{$L$} (2,.6);
			
				\draw (4,0) rectangle node{$B$} (1.5,-.6);
				\draw (1.5,-.6) rectangle node{$L$} (2,0);
				\draw (4,0) rectangle node{$L$} (3.5,-.6);
		\end{tikzpicture}
\\ \hline
		\begin{tikzpicture}[scale=.85]
				\draw (0,0) rectangle node{$W$} (4,.6);
				\draw (2,1.7) node[above]{Case (7) $\rightarrow \; \Rightarrow\!\Leftarrow$};
				\draw (2,1.2) node[above]{$m + 2\ell < 2k < 2(m - \ell)$};
			
				\draw (0,1.2) rectangle node{$B$} (2.9,.6);
				\draw (0,.6) rectangle node{$L$} (.5,1.2);
				\draw (2.9,1.2) rectangle node{$L$} (2.4,.6);
			
				\draw (4,0) rectangle node{$B$} (1.1,-.6);
				\draw (1.1,-.6) rectangle node{$L$} (1.6,0);
				\draw (4,0) rectangle node{$L$} (3.5,-.6);
		\end{tikzpicture}
&		
		\begin{tikzpicture}[scale=.85]
				\draw (0,0) rectangle node{$W$} (4,.6);
				\draw (2,1.7) node[above]{Case (8) $\rightarrow \ang{v} / \Rightarrow\!\Leftarrow$};
				\draw (2,1.2) node[above]{$k = m - \ell$};
			
				\draw (0,1.2) rectangle node{$B$} (3.5,.6);
				\draw (0,.6) rectangle node{$L$} (.5,1.2);
				\draw (3,1.2) rectangle node{$L$} (3.5,.6);
			
				\draw (4,0) rectangle node{$B$} (.5,-.6);
				\draw (.5,-.6) rectangle node{$L$} (1,0);
				\draw (4,0) rectangle node{$L$} (3.5,-.6);
		\end{tikzpicture}
&		
		\begin{tikzpicture}[scale=.85]
				\draw (0,0) rectangle node{$W$} (4,.6);
				\draw (2,1.7) node[above]{Case (9) $\rightarrow \; \Rightarrow\!\Leftarrow$};
				\draw (2,1.2) node[above]{$m - \ell < k < m$};
			
				\draw (0,1.2) rectangle node{$B$} (3.7,.6);
				\draw (0,.6) rectangle node{$L$} (.5,1.2);
				\draw (3.7,1.2) rectangle node{$L$} (3.2,.6);
			
				\draw (4,0) rectangle node{$B$} (.3,-.6);
				\draw (.3,-.6) rectangle node{$L$} (.8,0);
				\draw (4,0) rectangle node{$L$} (3.5,-.6);
		\end{tikzpicture}
\\ \hline
		\end{tabular}

		\caption[All possible ways the minimal $Z_2$-bifix of a word can overlap.]{All possible ways the minimal $Z_2$-bifix of $W$ can overlap, with $m = |W|$, $\ell = |L|$, and $k = |LBL|$} \label{overlap}

	\end{threeparttable}
	\end{figure}
	\begin{enumerate}[\text{Case} (1):]
		\item $2k < m$. This is $\< i\>$.
		\item $2k = m$. This is $\<ii\>$.
		\item $m < 2k < m + \ell$. In $LAL$, the first and last occurrences of $LBL$ overlap by a length strictly between $0$ and $\ell$. This is impossible, since $L$ is bifix-free.
		\item $2k = m + \ell$. This is $\<iii\>$
		\item $m + \ell < 2k < m + 2\ell$. The first and last occurrences of $LBL$ overlap by a length strictly between $\ell$ and $2\ell$. This is impossible, since $L$ is bifix-free.
		\item $m + 2\ell = 2k < 2(m - \ell)$. $LAL = L(DL)(LE)L$ where $DL = B = LE$. Thus $L$ is a bifix of $B$, so $LAL = LLFLLFLL$ where $B = LFL$. If $|F|>0$, this is $\<iv\>$. If $|F|=0$, then $LAL = LLLLLL$. But this contradicts the minimality of $LBL$, since $LLLLLL$ has $Z_2$-bifix $LLL$, which is shorter than $LBL = LLLL$.
		\item $m + 2\ell < 2k < 2(m - \ell)$. $LAL = LDLELD'L$ where $DLE = B = ELD'$. 
		Since $EL$ is a prefix of $B$, $LEL$ is a prefix of $LAL$. 
		Likewise, since $LE$ is a suffix of $B$, $LEL$ is a suffix of $LAL$.
		Therefore, $LEL$ is a bifix of $LAL$ and $|LEL| < |LDLEL| = |LBL|$, contradicting the minimality of $LBL$.
		\item $k = m - \ell$. $LAL = LLCLL$ where $LC = B = CL$. If $|C|=0$, this is $\<v\>$. Otherwise, $LCL$ is a bifix of $LAL$, contradicting the minimality of $LBL$.
		\item $m - \ell < k < m$. The first and last occurrences of $LBL$ overlap by a length strictly between $k-\ell$ and $k$. This is impossible, since $L$ is bifix-free.
	\end{enumerate}
\end{proof}

For fixed bifix-free word $L$ of length $\ell$, define $b_m^\ell$ to count the number of $Z_2$ words with bifix $L$ that are $Z_2$-bifix-free $q$-ary words of length $m$. 
Then 
\begin{equation} \label{eqn:IZ3}
\II(Z_3,q) = \sum_{\ell = 1}^{\infty} \left( a_\ell \sum_{m = 1}^\infty b_m^\ell q^{-2m} \right).
\end{equation}

In order to form a recursive definition of $b_n$ as we did for $a_n$, we now describe two new terms. Let $AB$ be a word of length $W$ with $|A| = \ceil{W/2}$ and $|B| = \floor{W/2}$. Then $AB$ has $q$ length-$(n+1)$ \textit{children} of the form $AxB$, each having $AB$ as its \textit{parent}. In this way every nonempty word has exactly $q$ children and exactly 1 parent, which establishes the 1:$q$ ratio of words of length $n$ to words of length $n+1$. The set of a word's children together with successive generations of progeny we refer to as that word's \textit{descendants}.

\begin{thm} \label{P3}
	$b_n^\ell = c_n^\ell + d_n^\ell$ where $c_n=c_n^\ell$ and $d_n=d_n^\ell$ are defined recursively as follows:
	\begin{eqnarray*}
		\text{For even }\ell:\\
		c_1 = \cdots = c_{2\ell} & = & 0,\\
		c_{2\ell+1} & = & q,\\
		c_{4\ell} & = & qc_{4\ell-1} - (c_{5\ell/2} + 1),\\
		c_{5\ell} & = & qc_{5\ell - 1} - (c_{5\ell/2} + c_{3\ell} - 1),\\
		c_{5\ell+1} & = & q(c_{5\ell} + c_{3\ell} - 1),\\
		c_{6\ell} & = & qc_{6\ell-1} - (c_{3\ell} - 1 + c_{5\ell/2});\\
		c_{2k} & = & qc_{2k-1} - (c_k + c_{k + \ell/2}) \text{ for } k>\ell,k \not\in\{2\ell,5\ell/2,3\ell\},\\
		c_{2k+1} & = & q(c_{2k} + c_{k + \ell/2}) \text{ for } k>\ell, k \neq 5\ell/2,\\
		d_1 = \cdots = d_{4\ell} & = & 0,\\
		d_{4\ell+1} & = & q,\\
		d_{5\ell} & = & qd_{5\ell-1} - 1,\\
		d_{5\ell+1} & = & q(d_{5\ell} + 1),\\
		d_{6\ell} & = & qd_{6\ell-1} - 1,\\
		d_{2k} & = & qd_{2k-1} - (d_k + d_{k+\ell} + d_{k + \ell/2}) \text{ for } k>2\ell,k \not\in \{5\ell/2,3\ell\},\\
		d_{2k+1} & = & q(d_{2k} + d_{k+\ell} + d_{k + \ell/2}) \text{ for } k\geq 2\ell, k \neq 5\ell/2.\\
		\text{For odd }\ell>1:\\
		c_1 = \cdots = c_{2\ell} & = & 0,\\
		c_{2\ell+1} & = & q,\\
		c_{4\ell} & = & q\left(c_{4\ell-1} + c_{\floor{\frac{5\ell}{2}}}\right) - (c_{2\ell} +1),\\
		c_{5\ell} & = & qc_{5\ell - 1} - (c_{3\ell} - 1),\\
		c_{5\ell+1} & = & q(c_{5\ell} + c_{3\ell} - 1) - c_{\ceil{\frac{5\ell}{2}}},\\
		c_{6\ell} & = & q\left(c_{6\ell-1} + c_{\floor{\frac{7\ell}{2}}}\right) - (c_{3\ell} -1),\\
		c_{2k} & = & q\left(c_{2k-1} + c_{k+\floor{\frac{\ell}{2}}}\right) - c_k; k>\ell,k \not\in \left\{2\ell,\ceil{\frac{\ell}{2}},3\ell\right\},\\
		c_{2k+1} & = & qc_{2k} - c_{k+\ceil{\frac{\ell}{2}}}; k>\ell,k \neq \floor{\frac{5\ell}{2}};\\
		d_1 = \cdots = d_{4\ell} & = & 0,\\
		d_{4\ell+1} & = & q,\\
		d_{5\ell} & = & qd_{5\ell-1} - 1,\\
		d_{5\ell+1} & = & q(d_{5\ell} + 1),\\
		d_{6\ell} & = & qd_{6\ell-1} - 1,\\
		d_{2k} & = & q\left(d_{2k-1}+ d_{k + \floor{\frac{\ell}{2}}}\right) - (d_k + d_{k+\ell}); k>2\ell,k\not\in \left\{\ceil{\frac{5\ell}{2}},3\ell\right\},\\
		d_{2k+1} & = & q\left(d_{2k} + d_{k+\ell}\right) -  d_{k + \ceil{\frac{\ell}{2}}}; k> 2\ell, k  \neq \floor{\frac{5\ell}{2}}.\\
		\text{For }\ell=1:\\
		c_1 = c_1 = c_2 & = & 0,\\
		c_3 & = & q,\\
		c_4 & = & qc_3 - 1,\\
		c_5 & = & qc_4 - (c_3 - 1),\\
		c_{6} & = & q(c_5 + c_3 - 1) - (c_3 - 1),\\
		c_{2k} & = & q(c_{2k-1} + c_k) - c_k; k>3,\\
		c_{2k+1} & = & qc_{2k} - c_{k+1}; k>2;\\
		d_1 = d_2 = d_3 = d_4 & = & 0,\\
		d_5 & = & q - 1,\\
		d_{6} & = & q(d_5+1) - 1,\\
		d_{2k} & = & q(d_{2k-1}+ d_k) - (d_k + d_{k+1}); k>3,\\
		d_{2k+1} & = & q(d_{2k} + d_{k+1}) -  d_{k + 1}; k> 2.
	\end{eqnarray*}

\end{thm}

\begin{proof}
	Fix a bifix-free word $L$ of length $\ell$.
	The full recursion is too messy to prove all at once, so we build up to it in stages.
	Within each stage, $\approx$ indicates an incomplete definition.
	Example word trees with small $q$ and short $L$ are found in Appendix~\ref{TREES}.

	\textbf{Stage I}\\
	Since $L$ is bifix free, any $Z_2$-instance with $L$ as a bifix has to be of greater length than $2\ell.$ Thus, $b_1 = \cdots = b_{2\ell} = 0$. 
	The only such words of length $2\ell+1$ are of the form $LxL$ for some letter $x$, therefore, $b_{2\ell+1} = q$.

	Every word of length $n>2\ell+1$ has $L$ as a bifix if and only if its parent has $L$ as a bifix. 
	This is why, for $k>\ell$, the definition of $b_{2k}$ includes the term $qb_{2k-1}$, and the definition of $b_{2k+1}$ includes the term $qb_{2k}$. If $b_n$ were counting $Z_2$-instances with bifix $L$, we would be done. 
	However, we do not want $b_n$ to count words that have a $Z_2$-bifix. 
	Thus, we must deal with each of the 5 cases listed in Lemma~\ref{P3cases}.

	First, let us deal with case $\ang{ii}$: $LAL = LBLLBL$ with $LBL$ the shortest $Z_2$-bifix of $LAL$. 
	The number of these of length $2k$ ($k > \ell$) is $b_k$. 
	Therefore, in the definition of $b_{2k}$, we subtract $b_k$. 
	Conveniently, the descendants of case-$\ang{ii}$ words are precisely words of case $\ang{i}$. 
	Therefore, we have accounted for two cases at once.

	Next, let us look at case $\ang{iii}$: $LAL = LBLBL$ with $LBL$ the shortest $Z_2$-bifix of $LAL$. 
	For the moment, assume $|L| = \ell$ is even. Then $|LBLBL|$ is even. 
	The number of such words of length $2k$ ($k > \ell$) is $b_{k+\ell/2}$. 
	We want to exclude words of this form, but we do not necessarily want to exclude their children. 
	Therefore, in the definition of $b_{2k}$ we subtract $b_{k+\ell/2}$, but then we add $qb_{k+\ell/2}$ in the definition of $b_{2k+1}$. 

	Now we look at when $|L|$ is odd, so $|LBLBL|$ is odd. 
	The number of such words of length $2k+1$ ($k > \ell$) is $b_{k+\ceil{\ell/2}}$. 
	Therefore, in the definition of $b_{2k+1}$ we subtract $b_{k+\ceil{\ell/2}}$, but then we add $qb_{(k - 1) +\ceil{\ell/2}} = qb_{k + \floor{\ell/2}}$ in the definition of $b_{(2(k-1)+1) + 1} = b_{2k}$.

	Our work so far renders the following tentative definition of $b_n$.
	\begin{eqnarray*}
		\text{For even }\ell:\\
		b_1 = \cdots = b_{2\ell} & = & 0,\\
		b_{2\ell+1} & = & q,\\
		b_{2k} & \approx & qb_{2k-1} - (b_k + b_{k + \ell/2}) \text{ for } k>\ell,\\
		b_{2k+1} & \approx & q(b_{2k} + b_{k + \ell/2}) \text{ for } k>\ell.\\
		\text{For odd }\ell:\\
		b_1 = \cdots = b_{2\ell} & = & 0,\\
		b_{2\ell+1} & = & q,\\
		b_{2k} & \approx & q(b_{2k-1} + b_{k+\floor{\ell/2}}) - b_k \text{ for } k>\ell,\\
		b_{2k+1} & \approx & qb_{2k} - b_{k+\lceil\ell/2\rceil} \text{ for } k>\ell.
	\end{eqnarray*}

	We continue with case $\ang{iv}$: $LAL = LLFLLFLL$ with $LLFLL$ the shortest $Z_2$-bifix of $LAL$. 
	Note that $|LLFLLFLL|$ is even. 
	It would apear that the number of such words of length $2k$ would be $b_{k - \ell}$ (counting words of the form $LFL$), which we could deal with in the same fashion as we did for case $\ang{iii}$. 
	However, when counting words of the form $LFL$, we do not want words of the form $LLGLL$, because $LLFLLFLL = LLLGLLLLGLLL$ is already accounted for in case $\ang{i}$.

	\textbf{Stage II}\\
	To address this issue, we will define two different recursions. 
	Let $d_n$ count the $Z_2$-instances of the form $LLALL$ that are $Z_2$-bifix free. 
	Let $c_n$ count all other $Z_2$-instances of the form $LAL$ that are $Z_2$-bifix free. Therefore, $b_n = c_n + d_n$ by definition. 

	As with $b_n$, we quickly see that $c_n = 0$ for $n\leq 2\ell$ and $c_{2\ell+1} = q$. 
	Now the shortest words counted by $d_n$ are of the form $LLxLL$ for some letter $x$, so $d_n = 0$ for $n\leq 4\ell$ and $d_{4\ell+1} = q$. 

	To deal with cases $\ang{i}$ and $\ang{ii}$, we can do the same things as before, but recognizing that $LL$ is a bifix of $LBLLBL$ if and only if $LL$ is a bifix of $LBL$. 
	Therefore, subtract $c_k$ in the definition of $c_{2k}$ and subtract $d_k$ in the definition of $d_{2k}$ (both for $k>\ell$).

	We also deal with case $\ang{iii}$ as before, recognizing that $LL$ is a bifix of $LBLBL$ if and only if $LL$ is a bifix of $LBL$. 
	For even $\ell$: subtract $c_{k+\ell/2}$ in the definition of $c_{2k}$ and add $qc_{k+\ell/2}$ in the definition of $c_{2k+1}$; subtract $d_{k+\ell/2}$ in the definition of $d_{2k}$ and add $qd_{k+\ell/2}$ in the definition of $d_{2k+1}$. 
	For odd $\ell$: subtract $c_{k+\ceil{\ell/2}}$ in the definition of $c_{2k+1}$ and add $qc_{k+\floor{\ell/2}}$ in the definition of $c_{2k}$; subtract $d_{k+\ceil{\ell/2}}$ in the definition of $d_{2k+1}$ and add $qd_{k+\floor{\ell/2}}$ in the definition of $d_{2k}$.

	Having split $b_n$ into $c_n$ and $d_n$, we can address case $\ang{iv}$: $LAL = LLFLLFLL$ with $LLFLL$ the shortest $Z_2$-bifix of $LAL$. 
	These words are counted by $d_n$, not by $c_n$, and there are $d_{k+\ell}$ such words of length $2k$. 
	Therefore, we subtract $d_{k+\ell}$ in the definition of $d_{2k}$ and add $qd_{k+\ell}$ in the definition of $d_{2k+1}$.

	This brings us to the following tentative definitions of $c_n$ and $d_n$.
	\begin{eqnarray*}
		\text{For even }\ell:\\
		c_1 = \cdots = c_{2\ell} & = & 0,\\
		c_{2\ell+1} & = & q,\\
		c_{2k} & \approx & qc_{2k-1} - (c_k + c_{k + \ell/2}),\\
		c_{2k+1} & \approx & q(c_{2k} + c_{k + \ell/2});\\
		d_1 = \cdots = d_{4\ell} & = & 0,\\
		d_{4\ell+1} & = & q,\\
		d_{2k} & \approx & qd_{2k-1} - (d_k + d_{k+\ell} + d_{k + \ell/2}),\\
		d_{2k+1} & \approx & q(d_{2k} + d_{k+\ell} + d_{k + \ell/2}).\\
		\text{For odd }\ell:\\
		c_1 = \cdots = c_{2\ell} & = & 0,\\
		c_{2\ell+1} & = & q,\\
		c_{2k} & \approx & q(c_{2k-1} + c_{k+\floor{\ell/2}}) - c_k,\\
		c_{2k+1} & \approx & qc_{2k} - c_{k+\lceil\ell/2\rceil};\\
		d_1 = \cdots = d_{4\ell} & = & 0,\\
		d_{4\ell+1} & \approx & q,\\
		d_{2k} & \approx & q(d_{2k-1}+ d_{k + \lfloor\ell/2\rfloor}) - (d_k + d_{k+\ell}),\\
		d_{2k+1} & \approx & q(d_{2k} + d_{k+\ell}) -  d_{k + \lceil\ell/2\rceil}.
	\end{eqnarray*}

	\textbf{Stage III}\\
	Next, let us deal with case $\ang{v}$: $LLLL$. 
	We merely need to subtract 1 in the definition of $c_{4\ell}$. 
	Since all of the words counted by $d_n$ are descendants of $LLLL$, this is what prevents overlap of the words counted by $c_n$ and $d_n$.

	There was a small omission in the previous stage. 
	When dealing with cases $\ang{i}$ and $\ang{ii}$, we pointed out that $LL$ is a bifix of $LBLLBL$ if and only if $LL$ is a bifix of $LBL$, this was a true and important observation. 
	The one problem is that $LLL$ has $LL$ as a bifix but is not of the form $LLALL$. 
	Therefore, $LLLLLL$ was ``removed'' in the definition of $c_{6\ell}$ when it should have been ``removed'' from $d_{6\ell}$. 
	We must account for this by adding 1 in the definition of $c_{6\ell}$ and subtracting 1 in the definition of $d_{6\ell}$.

	Similarly, in dealing with case $\ang{iii}$, we ``removed'' $LLLLL$ in the definition of $c_{5\ell}$ and ``replaced'' its children in the definition of $c_{5\ell+1}$. These should have happened to $d_n$. 
	Therefore, we add 1 and subtract $q$ in the definitions of $c_{5\ell}$ and $c_{5\ell+1}$, respectively, then subtract 1 and add $q$ in the definitions of $d_{5\ell}$ and $d_{5\ell+1}$, respectively.

	Since $LLL$ does not cause any trouble with case $\ang{iv}$, we are done building the recursive definition for even $\ell$ as found in the theorem statement.

	\textbf{Stage IV}\\
	The recursion for odd $\ell$ has the additional caveat that $\ell \neq 1$. When $\ell=1$, there exist conflicts in the recursive definitions: $4\ell + 1 = 5\ell$ and $5\ell+1 = 6\ell$. After consolidating the``adjustments'' for these cases, we get the definition for $\ell=1$ as appears in the theorem statement.
\end{proof}

With our recursively defined sequences $a_n$ and $b_n$, the latter in terms of $c_n$ and $d_n$, we are now able to formulate Theorem~\ref{thm:IZ2} for $Z_3$.

\begin{thm} \label{thm:IZ3}
For integers $q \geq 2$,
\begin{eqnarray*}
	 \II(Z_3,q) & = & \sum_{\ell = 1}^{\infty} a_\ell \left(\sum_{i = 0}^{\infty}(G(i) + H(i))\right).
\end{eqnarray*} 

where
\begin{eqnarray*}
	G(i) = G_\ell^{(q)}(i) &=& \frac{ (-1)^i r\!\left(q^{-2^{i+1}}\right) \prod_{j = 0}^{i-1} s\!\left(q^{-2^{j+1}}\right)}{ \prod_{k = 0}^i \left(1 - q^{1-2^{k+1}}\right)};\\
	r(x) = r_\ell^{(q)}(x) & = & qx^{2\ell+1} - x^{4\ell} + x^{5\ell} - qx^{5\ell+1} +x^{6\ell};\\
	s(x) = s_\ell^{(q)}(x) & = & 1 - qx^{1-\ell} + x^{-\ell};\\
	H(i) = H_\ell^{(q)}(i) &=& \frac{ (-1)^i u\!\left(q^{-2^{i+1}}\right) \prod_{j = 0}^{i-1} v\!\left(q^{-2^{j+1}}\right) }{ \prod_{k = 0}^i \left(1 - q^{1-2^{k+1}}\right)};\\
	u(x) = u_\ell^{(q)}(x) & = & qx^{4\ell+1} - x^{5\ell} + qx^{5\ell+1} - x^{6\ell};\\
	v(x) = v_\ell^{(q)}(x) & = & 1 - qx^{1-\ell} + x^{-\ell} - qx^{1-2\ell} + x^{-2\ell}.
\end{eqnarray*}
\end{thm}

\begin{proof}
	Recalling Equation~\eqref{eqn:IZ3},
	\begin{eqnarray*}
		\II(Z_3,q) & = & \sum_{\ell = 1}^{\infty} \left( a_\ell \sum_{m = 1}^\infty b_m^\ell q^{-2m} \right) \\
		& = & \sum_{\ell = 1}^{\infty} \left( a_\ell \sum_{m = 1}^\infty \left(c_m^\ell + d_m^\ell\right) q^{-2m} \right).
	\end{eqnarray*}

Similar to our proof for $\II(Z_2,q)$, let us define generating functions for the sequences $c_n = c_n^\ell$ and $d_n = d_n^\ell$:
\[ g(x) = g_\ell^{(q)}(x) = \sum_{i = 1}^\infty c_nx^n \text{ and } h(x) = h_\ell^{(q)}(x) = \sum_{i = 1}^\infty d_nx^n. \]
Despite having to write the recursive relations three different ways, depending on $\ell$, the underlying recursion is fundamentally the same and results in the following functional equations:
\begin{eqnarray}
	g(x) & = & q\left(xg(x) + x^{1-\ell}g(x^2) + x^{2\ell+1} - x^{5\ell+1}\right) \label{eqn:g}\\
	\nonumber & & - \left(g(x^2) + x^{-\ell}g(x^2) + x^{4\ell} - x^{5\ell} - x^{6\ell}\right);\\
	h(x) & = & q\left(xh(x) + x^{1-2\ell}h(x^2) + x^{1-\ell}h(x^2) + x^{4\ell+1} + x^{5\ell+1}\right) \label{eqn:h} \\
	\nonumber & & - \left(h(x^2) + x^{-2\ell}h(x^2) + x^{-\ell}h(x^2) + x^{5\ell} + x^{6\ell}\right).
\end{eqnarray}

Solving \eqref{eqn:g} for $g(x)$, we get
\begin{eqnarray}
	g(x) = \frac{r(x) - s(x)g(x^2)}{1-qx}, \label{eqn:g2}
\end{eqnarray}
with $r(x)$ and $s(x)$ as defined in the theorem statement. 
Expanding \eqref{eqn:g2} gives
\begin{eqnarray}
	\nonumber g(x) & = & \frac{r(x) - s(x)g(x^2)}{1-qx} \\
	\nonumber & = & \frac{r(x)}{1-qx}\left(1 -  \frac{s(x)}{r(x)}g(x^2)\right)\\
	\nonumber & = & \frac{r(x)}{1-qx}\left(1 -  \frac{s(x)}{r(x)}\frac{r(x^2) - s(x^2)g(x^4)}{1-qx^2}\right)\\
	\nonumber & = & \frac{r(x)}{1-qx}\left(1 -  \frac{s(x)}{r(x)}\frac{r(x^2)}{1-qx^2}\left(1 - \frac{s(x^2)}{r(x^2)}g(x^4) \right)\right)\\
	\nonumber & \vdots & \\
	& = & \sum_{i = 0}^\infty \frac{ (-1)^i r\!\left(x^{2^i}\right) \prod_{j = 0}^{i-1} s\!\left(x^{2^j}\right) }{ \prod_{k = 0}^i \left(1 - qx^{2^k}\right)}.
\end{eqnarray}
Likewise, solving \eqref{eqn:h} for $h(x)$, we get
\begin{eqnarray}
	h(x) &=& \frac{u(x) - v(x)h(x^2)}{1-qx}\\
	& = & \sum_{i = 0}^\infty \frac{(-1)^i u\!\left(x^{2^i}\right) \prod_{j = 0}^{i-1} v\!\left(x^{2^j}\right) }{ \prod_{k = 0}^i \left(1 - qx^{2^k}\right)},
\end{eqnarray}
with $u(x)$ and $v(x)$ as defined in the theorem statement.
\end{proof}

\begin{cor} \label{P3bounds}
For integers $N \geq 0$ and $M \geq 0$,
\begin{eqnarray*}
	\sum_{\ell = 1}^{N} a_\ell \left(\sum_{i = 0}^{2M+1} (G(i) + H(i)) \right) & \leq & \II(Z_3,q);\\
	 \II(Z_3,q) & \leq & q^{-N} + \sum_{\ell = 1}^{N} a_\ell \left(\sum_{i = 0}^{2M}(G(i) + H(i))\right),
\end{eqnarray*} 
with $G(i) = G_\ell^{(q)}(i)$ and $H(i) = H_\ell^{(q)}(i)$ as defined in Theorem~\ref{thm:IZ3}.
\end{cor}

\begin{proof} 
	For fixed integers $q\geq 2$ and $\ell \geq 1$, $\sum_{i=0}^\infty (G(i) + H(i))$ is an alternating series. 
	We need to show that the sequence $|G(i)+H(i)|$ is decreasing. Since $(-1)^iG(i)>0$ and $(-1)^iH(i) > 0$ for each $i$, $|G(i) +H(i)| = |G(i)| + |H(i)|$. 
	Thus it suffices to show that $\left\{|G(i)|\right\}_{i=1}^{\infty}$ and $\left\{|H(i)|\right\}_{i = 1}^\infty$ are both decreasing sequences, the routine proof of which can be found in the appendices (Lemma~\ref{lemGH}).
	
	Now for any integer $M\geq 0$:
	\[\sum_{i=0}^{2M+1} G_\ell(i) + H_\ell(i) < \sum_{m = 0}^{\infty} b_m^\ell q^{-2m} < \sum_{i=0}^{2M} G_\ell(i) + H_\ell(i).\]
	Moreover, since the $a_\ell$ are nonnegative, the lower bound for the theorem is evident. 
	For a bifix-free word $L$ of length $\ell$, $\sum_{m = 0}^{\infty} b_m^\ell q^{-2m}$ is the limit, as $M \rightarrow \infty$, of the probability that a word of length $M$ is a $Z_3$-instance of the form $LALBLAL$. 
	A necessary condition for such a word is that it starts and ends with $L$, which (for $M\geq 2\ell$) has probability $q^{-2\ell}$. 
	Also $a_\ell$ counts the number of bifix-free words of length $\ell$, so $a_\ell \leq q^\ell$. 
	Hence for any integer $N \geq 0$:
	\begin{eqnarray*}
		\II(Z_3,q) & < & \sum_{\ell=1}^{N}a_\ell \sum_{m = 0}^{\infty} b_m^\ell q^{-2m} + \sum_{\ell = N+1}^{\infty}q^\ell \left(q^{-2\ell}\right)\\
		& = &\sum_{\ell=1}^{N}a_\ell \sum_{m = 0}^{\infty} b_m^\ell q^{-2m} + \sum_{\ell = N+1}^{\infty} q^{-\ell}\\
		& \leq &\sum_{\ell=1}^{N}a_\ell \sum_{m = 0}^{\infty} b_m^\ell q^{-2m} + q^{-N}.
	\end{eqnarray*}
\end{proof}


\begin{table}[ht]
\centering
\begin{threeparttable}

	\caption{Approximate values of $\II(Z_3,q)$ for $2 \leq q \leq 6$.} \label{table:IZ3}

	\begin{tabular}{c | c | c | c | c | c}
		$q$ & 2 & 3 & 4 & 5 & 6 \\ \hline
		$\II(Z_3,q)$ & 0.11944370 & 0.01835140 & 0.00519251 & 0.00199739 & 0.00092532
	\end{tabular}

\end{threeparttable}
\end{table}

The values in Table~\ref{table:IZ3} were generated by the Sage code found in Appendix~\ref{P3code}, which was derived directly from Corollary~\ref{P3bounds} and can be used to compute $\II(Z_3,q)$ to arbitrary precision for any $q\geq 2$.

\section{Bounding \texorpdfstring{$\II(Z_n,q)$ for Arbitrary $n$}{the asymptotic instance probability of Zn}} \label{IZn}

This programme is not practical for $n$ in general. 
The number of cases for a generalization of Lemma 3.1 is likely to grow with $n$. 
Even if that stabilizes somehow, the expression for calculating $\II(Z_n,q)$ requires $n$ nested infinite series.
Nevertheless, ignoring some of the more subtle details, we proceed with this method to obtain a not-overly-messy way to calculate bounds for $\II(Z_n,q)$ in general.



Fix a $Z_{n-1}$-instance $L$ of length $\ell \geq 1$, let $\hat{b}_m^\ell$ be the number of words of length $m$ of the form $LAL$ for $A \neq \varepsilon$ but not of the form $LBLBL$. That is, $\hat{b}$ is an overcount for the number of $Z_n$-instances of the form $LAL$. Then $\hat{b}_m = \hat{b}_m^\ell$ is recursively defined as follows:
\begin{eqnarray*}
	\text{For even } \ell: \\
	\hat{b}_0 = \cdots = \hat{b}_{2\ell} & = & 0,\\
	\hat{b}_{2k} & = & q\hat{b}_{2k-1} - (\hat{b}_k + \hat{b}_{k + \ell/2}) \text{ for } k>\ell, \\
	\hat{b}_{2k+1} & = & q(\hat{b}_{2k} + \hat{b}_{k+\ell/2}) \text{ for } k>\ell. \\
	\text{For odd } \ell: \\
	\hat{b}_0 = \cdots = \hat{b}_{2\ell} & = & 0, \\
	\hat{b}_{2k} & = & q(\hat{b}_{2k-1} + \hat{b}_{k + \floor{\ell/2}}) - \hat{b}_k \text{ for } k>\ell, \\
	\hat{b}_{2k+1} & = & q\hat{b}_{2k} - \hat{b}_{k+\ceil{\ell/2}} \text{ for } k>\ell. 
\end{eqnarray*}

The the associated generating function $\hat{f}(x) := \hat{f}_\ell^q(x) = \sum_{m=1}^{\infty} \hat{b}_m^\ell x^m$ satisfies
\[\hat{f}(x) = q(x^{2\ell+1} + x\hat{f}(x) + x^{1-\ell}\hat{f}(x^2)) - (\hat{f}(x^2) + x^{-\ell}\hat{f}(x^2)).\]
Therefore, setting $t(x) = t_\ell^{(q)}(x) = 1 - qx^{1-\ell} + x^{-\ell}$,
\begin{eqnarray*}
	\hat{f}(x) & = & \frac{qx^{2\ell+1} - t(x)\hat{f}(x^2)}{1-qx} \\
	& = &  q\cdot \sum_{i = 0}^\infty \frac{ (-1)^i x^{(2^i)(2\ell+1)} \prod_{j = 0}^{i-1} t\!\left(x^{2^j}\right) }{ \prod_{k = 0}^i \left(1 - qx^{2^k}\right)}.
\end{eqnarray*}
Now $\hat{f}(q^{-2})$ gives an upper bound for the limit (as word-length approaches infinity) of the probability that a word is a $Z_n$-instance of the form $LAL$. 
We can write the following expressions as upper bounds for $\II(Z_n,q)$:
\begin{eqnarray*}
	\II(Z_n,q) & \leq &\sum_{\ell_0=1}^\infty \cdots \sum_{\ell_n=1}^\infty\sum_{m=1}^\infty a_{\ell_1}\hat{b}_{\ell_2}^{\ell_1} \cdots \hat{b}_{\ell_{n}}^{\ell_{n-1}}\hat{b}_{m}^{\ell_n} q^{-2\ell_n} ; \\
	\II(Z_n,q) & \leq &\sum_{\ell_0=1}^{N_1} \cdots \sum_{\ell_n=1}^{N_n} \sum_{m=1}^\infty a_{\ell_1}\hat{b}_{\ell_2}^{\ell_1} \cdots \hat{b}_{\ell_{n}}^{\ell_{n-1}}\hat{b}_{m}^{\ell_n} q^{-2\ell_n}\\
	& & + \; n\sum_{\ell = N_1+1}^\infty q^{-\ell} .
\end{eqnarray*}

A more precise recursion can be attained by extensive case-work, but the improvement gained is likely not worth the effort.

%% file: Further.tex
\chapter{Future Directions} \label{FUTURE}

\section{Word Densities}


\subsection{Limit Factor Densities}

We saw in our density comparison of Section~\ref{Compareakal} that the limsup factor density of $a^k$ is 1 for any $q,k \in \Z^+$. 
However, this is not the case for words with at least two distinct letters.
Generating functions or the de Bruijn graph may provide great tools for answering the following question.

\begin{qun}
	For $q \geq 2$ and $V$ with at least two distinct letters, what is 
	\[ \limsup_{\substack{W \in [q]^n \\ |W| \rightarrow \infty}} \d(V,W) ?\]
\end{qun}


\subsection{Density Comparisons}

The plots of possible $Z_2$- and $Z_3$-densities in short binary words (Figure~\ref{figure:deltaZ2Z3}) suggests a nonlinear asymptotic lower bound for $\delta(Z_3,W)$ in terms of $\delta(Z_2,W)$.
Moreover, it is surprising to observe that the minimum $Z_3$-density does not coincide with the minimum $Z_2$-density.
Considering the words $(a^ib^j)^n$ with $n \rightarrow \infty$, we see that the absolute upper bound of $y=x$ is asymptotically tight, at least for $x = \frac{i^2 + j^2}{(i + j)^2}$.

\begin{qun}
	Over a fixed alphabet, what is the asymptotic lower bound for $\delta(Z_3,W)$ in terms of $\delta(Z_2,W)$?
\end{qun}

\subsection{Encounter Enumeration}

Given a $V$-instance $W$, there might be multiple homomorphisms on $\L(V)^*$ that produce $W$.
For this reason, the number of encounters, $\hom(V,W)$, was only used to find an upper bound for $\delta(V,W)$.
However, the quantity $\frac{\hom(V,W)}{\binom{|W| + 1}{2}}$ is not generally expected to be less than 1.
The worst-case scenario is with factors of the form $a^k$, for which every one of the $\binom{k+1}{|V|+1}$ partitions into $|V|$ nonempty substrings gives a unique encounter.
However, when $V$ has exactly 1 nonrecurring letter, the lower and upper bounds on $\II(V,q)$ (Theorems~\ref{nondoubledProb},~\ref{IV}) are asympototic in $q$.
So for such $V$ and large random $W$, $\EE(\hom(V,W))$ is a good estimate for $\EE\left(\binom{|W|+1}{2}\delta(V,W)\right)$.
Yet we see from the proof of Lemma~\ref{bulk}, that if $V$ has multiple nonrecurring letters, we can expect numerous homomorphisms for a given instance.

\begin{qun}
	Fixed $q \geq 2$.
	Assuming a uniformly random selection of $W_n \in [q]^n$, let $\hom_{n,sur}(V,q)$ be the expected number of nonerasing homomorphisms $\phi:\L(V)^* \rightarrow [q]^*$ such that that $\phi(V) = W_n$.
	If $V$ has exactly $k$ nonrecurring letters, what is the asymptotic growth of 
	\[ \frac{\hom_{n,sur}(V,q)}{\II_n(V,q)} \]
in terms of $n$, $k$, and $q$?
\end{qun}

\subsection{Abelian Encounters}

In Problem (II.2) of a list of unsolved problems, \textcite{E-61} suggested that `perhaps an infinite sequence of four symbols can be formed without consecutive ``identical'' [factors]' where two word are ``identical'' provided `each symbol occurs the same number of times in both of them (i.e., we disregard order).'
For a summary of the history of this problem by Erd\H{o}s, through its positive answer by \textcite{D-79}, see Section 5.3 of \textcite{BLRS-08}. This appears to be the first consideration of what are now called Abelian encounters.
\begin{defn}
	Word $W$ is an \emph{Abelian $V$-instance} for word $V = a_1a_2\cdots a_n$ provided $W = A_1A_2\cdots A_n$ for nonempty words $A_i$ such that $A_i$ and $A_j$ are anagrams whenever $a_i = a_j$.
	$W$ \emph{encounters $V$ in an Abelian sense} provided some factor of $W$ is an Abelian $V$-instance.
\end{defn}

\textcite{C-05} restates and introduces a number of open problems regarding avoidability in the Abelian sense.
It was in response to Currie's paper that \textcite{T-14} proved the Abelian variant of Theorem~\ref{TaoLower}, with which he established a lower bound for Zimin-avoidance.
It is perhaps worth reproducing the present density results in the Abelian sense.

%

%

%
%
%
%

\section{Word Limits} \label{future:WordLimits}

\subsection{Convergence}

A driving force of the Graph Limits programme (see \cite{L-12}) is found in the various forms of convergence, especially for dense graphs. 
For example, a sequence of graphs $\{G_n\}_{n=1}^{\infty}$ with $|V(G)| \rightarrow \infty$ is \emph{left-convergent} provided the graph densities $t(F,G_n)$ converge for every finite graph $F$. 
There is also a concept of \emph{right-convergence}, convergences via a cut metric $\delta_\square$, convergence of ground state energy (from statistical physics), and more. 
The remarkable fact is that many of these forms of convergence are equivalent. 

Now there are multiple ways to define convergence of a sequence of words $\{W_n\}_{n=1}^{\infty}$ with length $|W_n| \rightarrow \infty$. 
One might define convergence in terms of factors:
\begin{itemize}
	\item $W_n$ is an initial factor of $W_{n+1}$ for all $n$;
	\item $W_n \leq W_{n+1}$ for all $n$;
	\item $d(V,W_n)$ converges for every finite words $V$;
	\item $\PP(V \text{ is followed by } x \text{ in } W_n)$ converges for every word-letter pair $(V,x)$.
\end{itemize}
Alternatively, convergence could be defined in terms of instances:
\begin{itemize}
	\item $W_{n+1}$ is an instance of $W_n$ for all $n$;
	\item $W_n \preceq W_{n+1}$ for all $n$;
	\item $\delta(V,W_n)$ converges for every finite words $V$.
\end{itemize}

These are clearly not all equivalent, but which ones are? More importantly, which ones are productive for a combinatorial limit theory.

\subsection{Lexons}

The rigorous theory of convergent graph sequences is crowned by the concept of a \emph{graphon}, the limit object for dense graphs. 
A graphon is a symmetric function $w: [0,1]^2 \rightarrow [0,1]$, and is determined (up to a measure 0 set and application of a measure preserving function on $[0,1]$) by the set of homomorphism densities of graphs into it. 
For example, the triangle-density of $w$ is 
\[t(K_3, w) = \int_{[0,1]^3} w(x,y)w(y,z)w(z,x) \; dx \; dy \; dz.\]
Since graphons lie in a compact space, various analytic tools can be used to develop continuous theory that then applies to associated large graphs.

\begin{qun}
	Do there exists limit objects for free words that lie in some compact space.
	Further, can we define metrics on words that extends productively to the limit object?
\end{qun}

For example, if we define convergence to be that ``$W_n$ is an initial factor of $W_{n+1}$ for all $n$,'' then the obvious limit object is a right-infinite word. 
For convergence defined as ``$W_n \leq W_{n+1}$ for all $n$,'' the limit object should be a bi-infinite word.
However, these particular forms of convergence do not appear sufficiently strong to guarantee any form of homomorphism density in the limit object. 

\subsection{Randomness}

A foundational result in graph theory is the Szemeredi Regularity Lemma, which roughly states that the vertex set of every sufficiently large graph can be partitioned so that the edges between parts are ``random-like.''
Generally quasirandomness is used to characterize a sequence of ``random-like'' graphs. 
Several of the many equivalent definitions of quasirandomness are in terms of the homomorphism densities of graphs.

\begin{qun}
	Does there exists a productive definition of quasirandomness for free words?
\end{qun}

Perhaps this would be in terms of factor or instance densities, or perhaps in terms of transition probabilities as used in the de Bruijn graph (Section~\ref{DeB}).

%% file: AppendixZ.tex
\chapter{Computations for Zimin Word Avoidance} \label{AppendZ}


\section{All Binary Words that Avoid \texorpdfstring{$Z_2$}{Z2}}

The following 13 words are the only words over the alphabet $\{0,1\}$ that avoid  the second Zimin word, $Z_2 = aba$. 

\begin{table}[ht]
\centering
\begin{threeparttable}

	\caption{Binary words that avoid $Z_2$.}
	
	\begin{tabular}{c}
	$\begin{matrix}
		\varepsilon, & 0, & 00, & 001, & 0011, \\
		& & 01, & 011, & \\
		& 1, & 10, & 100, & \\
		& & 11, & 110, & 1100. 
	\end{matrix}$
	\end{tabular}

\end{threeparttable}
\end{table}

\newpage

\section{Maximum-Length Binary Words that Avoid \texorpdfstring{$Z_3$}{Z3}}

The 48 words in Table~\ref{tableZ3avoid} are all the words of length $f(3,2)-1 = 28$ over the alphabet $\{0,1\}$ that avoid $Z_3= abacaba$. All binary words of length at least $f(3,2) = 29$ encounter $Z_3$. This result is easily computationally verified by constructing the binary tree of words on $\{0,1\}$, eliminating branches as you find words that encounter $Z_3$.

\begin{table}[ht]
\centering
\begin{threeparttable}
	\caption{Maximum-length binary words that avoid $Z_3$.} \label{tableZ3avoid}
			\begin{tabular}{c | c}
				0010010011011011111100000011,&1100000010010011011011111100,\\
				0010010011111100000011011011,&1100000010010011111101101100,\\
				0010010011111101101100000011,&1100000010101100110011111100,\\
				0010101100110011111100000011,&1100000010101111110011001100,\\
				0010101111110000001100110011,&1100000011001100101011111100,\\
				0010101111110011001100000011,&1100000011001100111111010100,\\
				0011001100101011111100000011,&1100000011011010010011111100,\\
				0011001100111111000000101011,&1100000011011011111100100100,\\
				0011001100111111010100000011,&1100000011111100100101101100,\\
				0011011010010011111100000011,&1100000011111100110011010100,\\
				0011011011111100000010010011,&1100000011111101010011001100,\\
				0011011011111100100100000011,&1100000011111101101100100100,\\
				0011111100000010010011011011,&1100100100000011011011111100,\\
				0011111100000010101100110011,&1100100100000011111101101100,\\
				0011111100000011001100101011,&1100100101101100000011111100,\\
				0011111100000011011010010011,&1100110011000000101011111100,\\
				0011111100100100000011011011,&1100110011000000111111010100,\\
				0011111100100101101100000011,&1100110011010100000011111100,\\
				0011111100110011000000101011,&1101010000001100110011111100,\\
				0011111100110011010100000011,&1101010000001111110011001100,\\
				0011111101010000001100110011,&1101010011001100000011111100,\\
				0011111101010011001100000011,&1101101100000010010011111100,\\
				0011111101101100000010010011,&1101101100000011111100100100,\\
				0011111101101100100100000011,&1101101100100100000011111100.
			\end{tabular}
\end{threeparttable}
\end{table}

\newpage

\section{A Long Binary Word that Avoids \texorpdfstring{$Z_4$}{Z4}}

Figure~\ref{figureZ4avoid} shows a binary word of length 10482 that avoids $Z_4 = abacabadabacaba$. This implies that $f(4,2)\geq 10483$. The word is presented here as an image with each row, consisting of 90 squares, read left to right. Each square, black or white, represents a bit. For example, the longest string of black in the first row is 14 bits long. We cannot have the same bit repeated $15 = |Z_4|$ times consecutively, as that would be a $Z_4$-instance. A string of 14 white bits is found in the $46^{\mathrm{th}}$ row.

\begin{figure}[ht]
\centering
\begin{threeparttable}

	\begin{tabular}{c}
		\includegraphics[width=270px]{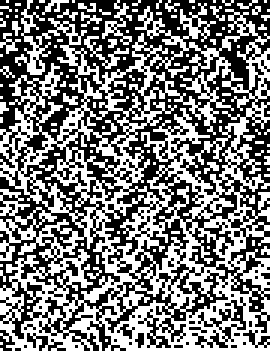}
	\end{tabular}

	\caption{A binary word of length 10482 that avoids $Z_4 = abacabadabacaba$.} \label{figureZ4avoid}

\end{threeparttable}
\end{figure}

\newpage

\section{Verifying \texorpdfstring{$Z_n$}{Zn}-Avoidance}

The code to generate a $Z_4$-avoiding word of length 10482 is messy. 
The following, easy-to-validate, inefficient, brute-force, Sage \parencite{S-14} code was used for verification of the word above. 
It took roughly 12 hours of computation on an Intel\textregistered\,Core\texttrademark\,i5-2450M CPU $@$ 2.50GHz $\times$ 4.


\begin{lstlisting}[frame=single]
# Recursive function to test if V is an instance of Z_n.
def inst(V, n):
	if len(V)==0:
		return False
	if n==1:
		return True
	for i in range(2^(n - 1) - 1, ceil(len(V)/2)):
		if V[:i]==V[-i:]:
			if inst(V[:i], n - 1):
				return True
	return False
W = # Paste word here as a string.
(L, n) = (len(W), 4)
# Check every subword V of length at least 2^n - 1.
for b in range(L + 1):
	for a in range(b - (2^n - 1)):
		if inst(W[a:b], n):
			print a, b, W[a:b]
\end{lstlisting}

%% file: AppendixZ2Z3.tex
\chapter{Computational Comparison: \texorpdfstring{$\delta(Z_2,W)$ vs. $\delta(Z_3,W)$}{Densities of Zimin Words}} \label{Z2Z3}


Figure~\ref{figure:Z2Z3} below shows plots of all $(x,y)$-pairs with $x = \delta(Z_2,W)$ and $y = \delta(Z_3,W)$ for binary words $W \in [2]^k$, where $k \in \{13,16,19,22,25,28\}$. 
More discussion of these plots is found in Section~\ref{Comparez2z3}. 
The following Sage \parencite{S-14} code was used to compute all $(x,y)$-pairs in the plots. 

\begin{lstlisting}[frame=single]
def is_Zn(W, n): # Checks if nonempty W is a Zn-instance.
	if n==1:
		return True
	for i in range(1, ceil(len(W)/2)):
		if W[:i]==W[-i:] and is_Zn(W[:i], n - 1):
			return True
	return False
def z2z3(W): # Counts Z2- and Z3-instance substrings.
	(M, z2, z3) = (len(W), 0, 0)
	for i in range(M - 2):
		for j in range(i + 3, M + 1):
			V = W[i:j]
			if is_Zn(V, 2):
				z2 += 1
				if is_Zn(V, 3):
					z3 += 1
	return [z2, z3]
L = 10 # Change to desired word-length.
(D2, D3) = ([1], []) # Create lists to store density values.
for n in xrange(2^L): # Check every binary word of length L.
	word = str(bin(n))[2:]
	word = '0'*(L - len(word)) + word
	p = z2z3(word)
	d2 = p[0]/binomial(L + 1, 2)
	d3 = p[1]/binomial(L + 1, 2)
	i = 0
	while d2>D2[i]:
		i += 1
	if d2<D2[i]:
		D2.insert(i, d2)
		D3.insert(i, set([]))
	D3[i].add(d3)
D2.pop(-1) # Remove the unnecessary 1.
\end{lstlisting}

\begin{figure}[ht]
\centering
\begin{threeparttable}
	\begin{tabular}{c c}
		\includegraphics[width=203px]{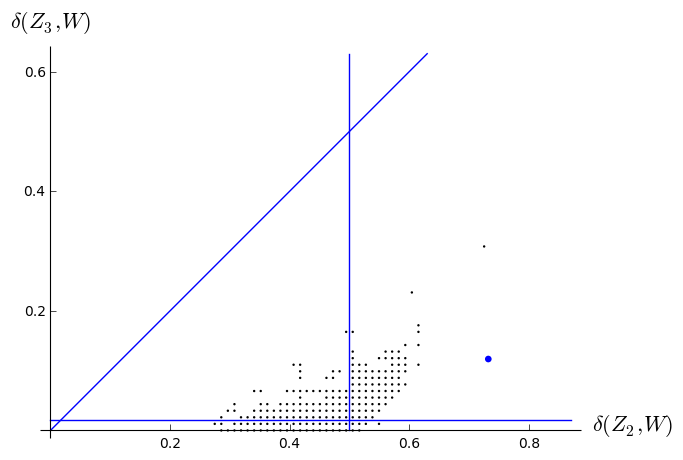} & \includegraphics[width=203px]{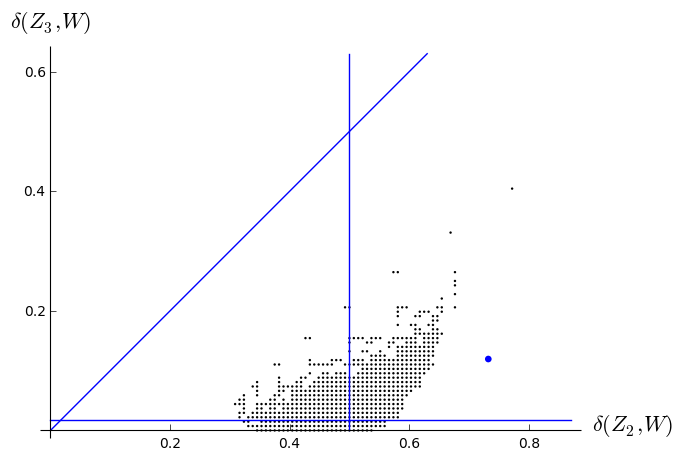} \\
		\includegraphics[width=203px]{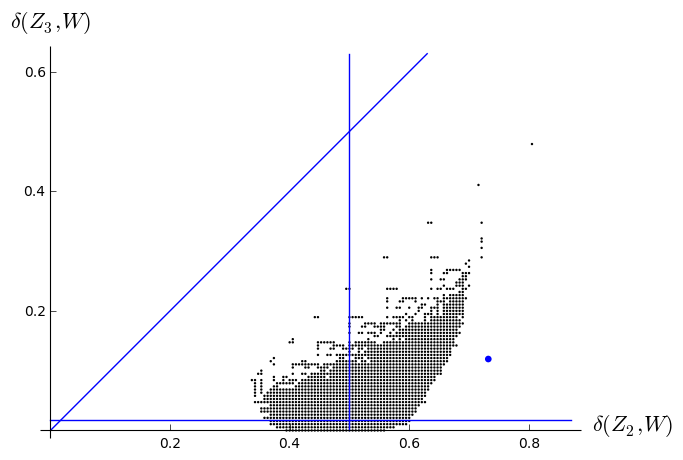} & \includegraphics[width=203px]{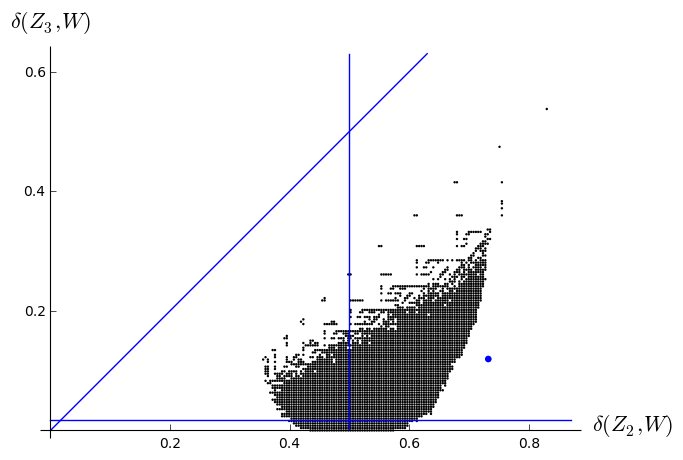} \\
		\includegraphics[width=203px]{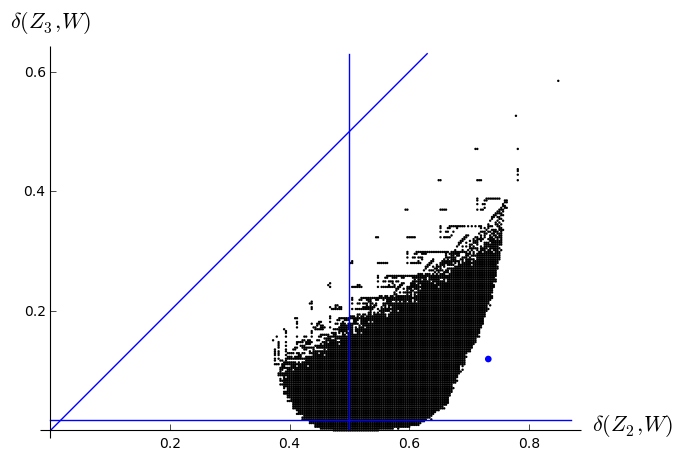} & \includegraphics[width=203px]{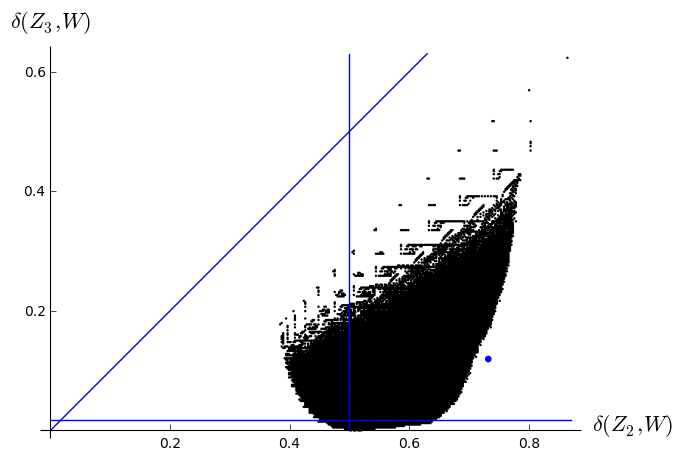} \\
	\end{tabular}

	\caption{$(\delta(Z_2,W),\delta(Z_3,W))$ for binary $W$ of length \{13,16,19,22,25,28\}.} \label{figure:Z2Z3}
	
\end{threeparttable}
\end{figure}

%% file: AppendixF.tex
\chapter{Proofs and Computations for Chapter~\ref{ASYMP}}

\section{Proofs of Monotonicity}

\begin{lem} \label{lemF}
	For fixed $q \geq 2$, $\left\{|F(i)|\right\}_{i=0}^{\infty}$ is a decreasing sequence, where \[F(i) = F^q(i) = \frac{(-1)^jq^{1-2^i}}{\prod_{k=0}^{i} (1 - q^{1-2^k})}.\]
\end{lem}

\begin{proof}
	For $i>0$:
	\begin{eqnarray*}
		\frac{|F(i)|}{|F(i-1)|}  & = & \frac{q^{1 - 2^i}}{q^{1 - 2^{(i-1)}}\left(1 - q^{1 - 2^i}\right)}\\
		 & = & \frac{q^{-2^{(i-1)}}}{1 - q^{1 - 2^i}}\cdot \frac{1 + q^{1 - 2^i}}{1 + q^{1 - 2^i}}\\
		 & = & \frac{q^{-2^{(i-1)}}\left(1 + q^{1 - 2^i}\right)}{1 + q^{2 - 2^{i+1}}}\\
		 & < & \frac{(2)^{-2^{((1)-1)}}\left(1 + (2)^{1 - 2^{(1)}}\right)}{1 +(0)}\\
		 & = & 2^{-1}\left(1 + 2^{1 - 2}\right)\\
		 & < & 1.
	\end{eqnarray*}
\end{proof}

\begin{lem} \label{lemGH}
For fixed $\ell \geq 1$ and $q \geq 2$, $\left\{|G(i)|\right\}_{i=1}^{\infty}$ and $\left\{|H(i)|\right\}_{i = 1}^\infty$ are both decreasing sequences, where
\begin{eqnarray*}
	G(i)  = G_\ell^q(i) &=& \frac{ (-1)^i r\!\left(q^{-2^{i+1}}\right) \prod_{j = 0}^{i-1} s\!\left(q^{-2^{j+1}}\right)}{ \prod_{k = 0}^i \left(1 - q^{1-2^{k+1}}\right)};\\
	r(x) = r_\ell^q(x) & = & qx^{2\ell+1} - x^{4\ell} + x^{5\ell} - qx^{5\ell+1} +x^{6\ell};\\
	s(x) = s_\ell^q(x) & = & 1 - qx^{1-\ell} + x^{-\ell};\\
	H(i) = H_\ell^q(i) &=& \frac{ (-1)^i u\!\left(q^{-2^{i+1}}\right) \prod_{j = 0}^{i-1} v\!\left(q^{-2^{j+1}}\right) }{ \prod_{k = 0}^i \left(1 - q^{1-2^{k+1}}\right)};\\
	u(x) = u_\ell^q(x) & = & qx^{4\ell+1} - x^{5\ell} + qx^{5\ell+1} - x^{6\ell};\\
	v(x) = v_\ell^q(x) & = & 1 - qx^{1-\ell} + x^{-\ell} - qx^{1-2\ell} + x^{-2\ell}.
\end{eqnarray*}

\end{lem}

\begin{proof}
For $i>0$:
\begin{eqnarray*}
	\frac{|G(i)|}{|G(i-1)|}  & = & \frac{r\!\left(q^{-2^{i+1}}\right)}{r\!\left(q^{-2^i}\right)}\cdot \frac{s\!\left(q^{-2^i}\right)}{1 - q^{1-2^{i+1}}}\\
	& = & \frac{ q^{1-2^i(4\ell+2)} - q^{-2^i(8\ell)} + q^{-2^i(10\ell)} - q^{1-2^i(10\ell+2)} + q^{-2^i(12\ell)} }{ q^{1-2^i(2\ell+1)} - q^{-2^i(4\ell)} + q^{-2^i(5\ell)} - q^{1-2^i(5\ell+1)} + q^{-2^i(6\ell)} }\\
	& &\cdot \frac{ 1 - q^{1+2^i(\ell-1)} + q^{2^i\ell} }{ 1 - q^{1-2^i(2)} }\\
	& < & \frac{ q^{1-2^i(4\ell+2)} }{ q^{1-2^i(2\ell+1)} - q^{-2^i(4\ell)} }\cdot \frac{q^{2^i\ell} }{ 1 - q^{1-2^i(2)} }\\
	& = & \frac{ q^{1-2^i(3\ell+2)} }{ q^{1-2^i(2\ell+1)} - q^{-2^i(4\ell)} - q^{2-2^i(2\ell+3)} + q^{1-2^i(4\ell+2)}}\cdot \frac{ q^{-1+2^i(2\ell+1)} }{ q^{-1+2^i(2\ell+1)} }\\
	& = & \frac{ q^{-2^i(\ell+1)} }{ 1 - q^{-1-2^i(2\ell-1)} - q^{1-2^i(2)} + q^{2^i(2\ell+1)} }\\
	& < & \frac{ (2)^{-2^1((1)+1)} }{ 1 - (2)^{-1-2^1(2(1)-1)} - (2)^{1-2^1(2)} + 0 }\\
	& = & \frac{ 2^{-4} }{ 1 - 2^{-3} - 2^{-3} }\\
	& < & 1;\\
\\
	\frac{|H(i)|}{|H(i-1)|} & = & \frac{u\!\left(q^{-2^{i+1}}\right)}{u\!\left(q^{-2^i}\right)}\cdot \frac{v\!\left(q^{-2^i}\right)}{1 - q^{1-2^{i+1}}}\\
	& = & \frac{ q^{1-2^i(8\ell+2)} - q^{-2^i(10\ell)} + q^{1-2^i(10\ell+2)} - q^{-2^i(12\ell)} }{ q^{1-2^i(4\ell+1)} - q^{-2^i(5\ell)} + q^{1-2^i(5\ell+1)} - q^{-2^i(6\ell)} }\\
	& &\cdot \frac{ 1 - q^{1+2^i(\ell-1)} + q^{2^i\ell} - q^{1+2^i(2\ell - 1)} + q^{2^i(2\ell)} }{ 1 - q^{1-2^i(2)} }\\
	& < & \frac{ q^{1-2^i(8\ell+2)} }{ q^{1-2^i(4\ell+1)} - q^{-2^i(5\ell)} }\cdot \frac{q^{2^i(2\ell)} }{ 1 - q^{1-2^i(2)} }\\
	& = & \frac{ q^{1-2^i(6\ell+2)} }{ q^{1-2^i(4\ell+1)} - q^{-2^i(5\ell)} - q^{2-2^i(4\ell+3)} + q^{1-2^i(5\ell+2)}}\cdot \frac{ q^{-1+2^i(4\ell+1)} }{ q^{-1+2^i(4\ell+1)} }\\
	& = & \frac{ q^{-2^i(2\ell+1)} }{ 1 - q^{-1-2^i(\ell-1)} - q^{1-2^i(2)} + q^{2^i(\ell+1)} }\\
	& < & \frac{ (2)^{-2^1(2(1)+1)} }{ 1 - (2)^{-1-2^1((1)-1)} - (2)^{1-2^1(2)} + 0 }\\
	& = & \frac{ 2^{-6} }{ 1 - 2^{-1} - 2^{-3} }\\
	& < & 1.
\end{eqnarray*}

\end{proof}

\section{Sage Code for Table~\ref{table:IZ3} of \texorpdfstring{$\II(Z_3,q)$-Values}{Values of the Asymptotic Density of Z3}} \label{P3code}

The following code to generate Table~\ref{table:IZ3} was run with Sage 6.1.1 \parencite{S-14}.


\begin{lstlisting}[frame=single]
# Calculate G(i), term i of expanded g(q^(-2)).
def r(L, q, x):
    X = x^L
    return q*x*X^2 - X^4 + X^5 - q*x*X^5 + X^6
def s(L, q, x):
    return 1 - q*x^(1-L) + x^(-L)
def G(L, q, i):
    num = prod([s(L, q, q^(-2^(j+1))) for j in range(i)])
    den = prod([1 - q^(1-2^(k+1)) for k in range(i+1)])
    return (-1)^i * r(L, q, q^(-2^(i+1))) * num / den
# Calculate H(i), term i of expanded h(q^(-2)).
def u(L, q, x):
    return q*x^(4*L+1) - x^(5*L) + q*x^(5*L+1) - x^(6*L)
def v(L, q, x):
    return 1 - q*x^(1-L) + x^(-L) - q*x^(1-2*L) + x^(-2*L)
def H(L, q, i):
    num = prod([v(L, q, q^(-2^(j+1))) for j in range(i)])
    den = prod([1 - q^(1-2^(k+1)) for k in range(i+1)])
    return (-1)^i * u(L, q, q^(-2^(i+1))) * num / den
# Generate the first N terms of {a_n}.
def a(q,N):
    A = [0,q]
    for n in range(2, N+1):
        A.append(q*A[-1] - ((n+1)%2)*A[floor(n/2)])
    return A
# Calculate the partial sum of I(Z_3, q).
def I(q, N, M):
    A = a(q, N)
    partial = 0
    for L in range(1, N+1):
        terms = [G(L, q, n) + H(L, q, n) for n in range(M+1)]
        partial += A[L]*sum(terms)
    return partial
# Output bounds on I(Z_3, q) for small values of q.
prec = 15 # Level of precision.
N = 2*prec
for q in range(2, 7):
	print 'q = %d:' %q
	print 'Lower bound with N = %d and M = 4:' %N, 
	print round(I(q, N, 4), prec)
	print 'Upper bound with N = %d and M = 5:' %N, 
	print round(I(q, N, 5) + 2^(-N), prec)
\end{lstlisting}

%% file: AppendixT.tex
\chapter{Word Trees Illustrating Theorem \ref{P3}} \label{TREES}

From Section \ref{IZ3}: ``For fixed bifix-free word $L$ length $\ell$, define $b_m^\ell$ to count the number of $Z_2$ words with bifix $L$ that are $Z_2$-bifix-free $q$-ary words of length $m$.''

In each of the following images, word is struck through if it in not counted by $b_m$ but its descendants are. It is hashed through if its descendants are also eliminated.

%% file: AppendixT21.tex
\begin{figure}[ht]
\centering
\begin{threeparttable}

\begin{tabular}{l}

\begin{tikzpicture}[grow'=right,level distance = 53pt,sibling distance=-5pt,every node/.style={scale=.75}]
	\Tree [.{$b$_3^1 = 2} [.{$b$_4^1 = 3} [.{$b$_5^1 = 6} [.{$b$_6^1=14} [.{$b$_7^1=25} [.{$b$_8^1 = 52} {$b$_9^1 = 100} ] ] ] ] ] ];
\end{tikzpicture}
\\
\begin{tikzpicture}[grow'=right,level distance=55pt, sibling distance=-4pt,every node/.style={scale=.75}]
	\Tree [.000
			[.\sout{0000} 
				[.\sout{00000} 
					[.\xout{000000} ]
					[.000100 
						[.0000100 
							[.00000100 
								[.000000100 ]
								[.000010100 ]
							]
							[.00001100 
								[.000001100 ]
								[.000011100 ]
							]
						]
						[.0001100 
							[.00010100 
								[.000100100 ]
								[.000110100 ]
							]
							[.00011100 
								[.000101100 ]
								[.000111100 ]
							]
						]
					]
				]
				[.00100 
					[.001000 
						[.0010000 
							[.00100000 
								[.001000000 ]
								[.001010000 ]
							]
							[.00101000 
								[.001001000 ]
								[.001011000 ]
							]
						]
						[.0011000 
							[.00110000 
								[.001100000 ]
								[.001110000 ]
							]
							[.00111000 
								[.001101000 ]
								[.001111000 ]
							]
						]
					]
					[.001100 
						[.0010100 
							[.\sout{00100100} 
								[.\sout{001000100} ]
								[.001010100 ]
							]
							[.00101100 
								[.001001100 ]
								[.001011100 ]
							]
						]
						[.0011100 
							[.00110100 
								[.001100100 ]
								[.001110100 ]
							]
							[.00111100 
								[.001101100 ]
								[.001111100 ]
							]
						]
					]
				]
			] 
			[.0010
				[.00010 
					[.000010 
						[.0000010 
							[.00000010 
								[.000000010 ]
								[.000010010 ]
							]
							[.00001010 
								[.000001010 ]
								[.000011010 ]
							]
						]
						[.0001010 
							[.00010010 
								[.\sout{000100010} ]
								[.000110010 ]
							]
							[.00011010 
								[.000101010 ]
								[.000111010 ]
							]
						]
					]
					[.000110 
						[.0000110 
							[.00000110 
								[.000000110 ]
								[.000010110 ]
							]
							[.00001110 
								[.000001110 ]
								[.000011110 ]
							]
						]
						[.0001110 
							[.00010110 
								[.000100110 ]
								[.000110110 ]
							]
							[.00011110 
								[.000101110 ]
								[.000111110 ]
							]
						]
					]
				]
				[.00110 
					[.001010 
						[.\sout{0010010} 
							[.\xout{00100010} ]
							[.00101010 
								[.001001010 ]
								[.001011010 ]
							]
						]
						[.0011010 
							[.00110010 
								[.001100010 ]
								[.001110010 ]
							]
							[.00111010 
								[.001101010 ]
								[.001111010 ]
							]
						]
					]
					[.001110 
						[.0010110 
							[.00100110 
								[.001000110 ]
								[.001010110 ]
							]
							[.00101110 
								[.001001110 ]
								[.001011110 ]
							]
						]
						[.0011110 
							[.00110110 
								[.\sout{001100110} ]
								[.001110110 ]
							]
							[.00111110 
								[.001101110 ]
								[.001111110 ]
							]
						]
					]
				]
			 ]
		];
	\draw (1.58,0.19) rectangle (13,6.9);
	\draw (13,3.55)node[left]{$d_n^1$};
\end{tikzpicture}

\end{tabular}

\caption[Example word tree for Theorem \ref{P3} with $q=2, \ell = 1$.]{The `000' half of an example word tree for Theorem \ref{P3} with $q=2$, $L=\text{`0'}$, $\ell = |L| = 1$. The tree from LLLL counted by $d_n$ is boxed.} \label{figure:T21}

\end{threeparttable}
\end{figure}
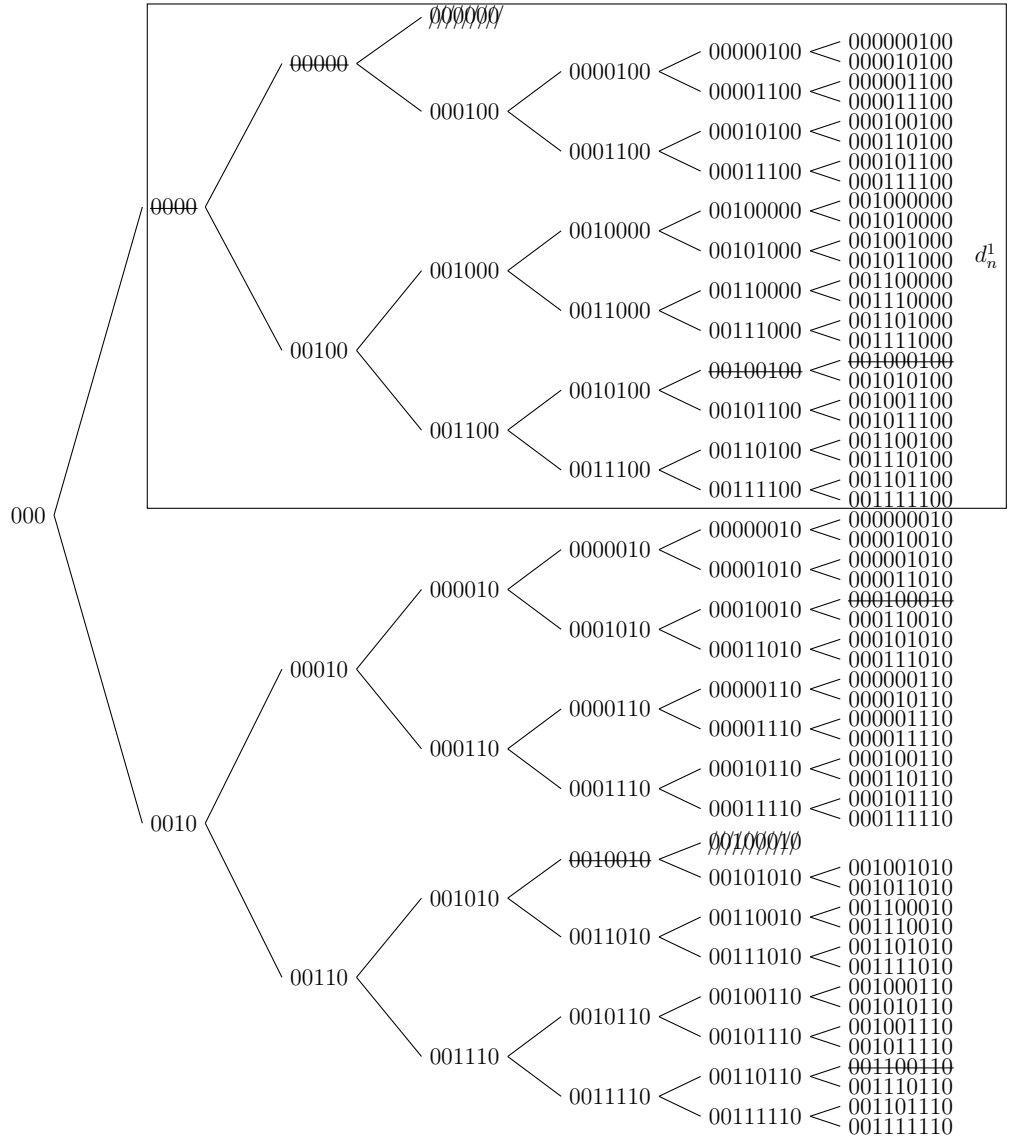

%% file: AppendixT22.tex
\begin{figure}[ht]
\centering
\begin{threeparttable}

\begin{tabular}{l}

\begin{tikzpicture}[grow'=right,level distance = 60pt,sibling distance=-5pt,every node/.style={scale=.75}]
	\Tree [.{$b$_5^2 = 2} [.{$b$_5^2 = 4} [.{$b$_7^2 = 8} [.{$b$_8^2=13} [.{$b$_9^2=32} {$b$_{10}^2=58} ] ] ] ] ];
\end{tikzpicture}
\\
\begin{tikzpicture}[grow'=right,level distance=62pt, sibling distance=-5pt, every node/.style={scale=.75}]
	\Tree [.01001
			[.010001 
				[.0100001 
					[.01000001 
						[.010000001 
							[.0100000001 ]
							[.\sout{0100010001} ]
						]
						[.010010001 
							[.0100100001 ]
							[.0100110001 ]
						]
					]
					[.\sout{01001001} 
						[.010001001 
							[.0100001001 ]
							[.0100011001 ]
						]
						[.010011001 
							[.\xout{0100101001} ]
							[.0100111001 ]
						]
					]
				]
				[.0101001 
					[.01010001 
						[.010100001 
							[.0101000001 ]
							[.0101010001 ]
						]
						[.010110001 
							[.0101100001 ]
							[.0101110001 ]
						]
					]
					[.01011001 
						[.010101001 
							[.0101001001 ]
							[.0101011001 ]
						]
						[.010111001 
							[.0101101001 ]
							[.0101111001 ]
						]
					]
				]
			]
			[.010101 
				[.0100101 
					[.01000101 
						[.010000101 
							[.0100000101 ]
							[.0100010101 ]
						]
						[.010010101 
							[.0100100101 ]
							[.0100110101 ]
						]
					]
					[.01001101 
						[.010001101 
							[.0100001101 ]
							[.0100011101 ]
						]
						[.010011101 
							[.0100101101 ]
							[.0100111101 ]
						]
					]
				]
				[.0101101 
					[.\sout{01010101} 
						[.010100101 
							[.0101000101 ]
							[.\sout{0101010101} ]
						]
						[.010110101 
							[.0101100101 ]
							[.0101110101 ]
						]
					]
					[.01011101 
						[.010101101 
							[.0101001101 ]
							[.0101011101 ]
						]
						[.010111101 
							[.0101101101 ]
							[.0101111101 ]
						]
					]
				]
			]
		] ;

	\draw (5.85,-1.75) rectangle (12.3,-2.68);
	\draw (12.3,-2.2) node[left]{$d_n^2$};
\end{tikzpicture}
\\
\begin{tikzpicture}[grow'=right,level distance=62pt, sibling distance=-5pt, every node/.style={scale=.75}]
	\Tree [.01101
			[.011001 
				[.0110001 
					[.01100001 
						[.011000001 
							[.0110000001 ]
							[.0110010001 ]
						]
						[.011010001 
							[.0110100001 ]
							[.0110110001 ]
						]
					]
					[.01101001 
						[.011001001 
							[.0110001001 ]
							[.\sout{0110011001} ]
						]
						[.011011001 
							[.0110101001 ]
							[.0110111001 ]
						]
					]
				]
				[.0111001 
					[.01110001 
						[.011100001 
							[.0111000001 ]
							[.0111010001 ]
						]
						[.011110001 
							[.0111100001 ]
							[.0111110001 ]
						]
					]
					[.01111001 
						[.011101001 
							[.0111001001 ]
							[.0111011001 ]
						]
						[.011111001 
							[.0111101001 ]
							[.0111111001 ]
						]
					]
				]
			]
			[.011101 
				[.0110101 
					[.01100101 
						[.011000101 
							[.0110000101 ]
							[.0110010101 ]
						]
						[.011010101 
							[.0110100101 ]
							[.0110110101 ]
						]
					]
					[.\sout{01101101} 
						[.011001101 
							[.0110001101 ]
							[.0110011101 ]
						]
						[.011011101 
							[.\xout{0110101101} ]
							[.0110111101 ]
						]
					]
				]
				[.0111101 
					[.01110101 
						[.011100101 
							[.0111000101 ]
							[.0111010101 ]
						]
						[.011110101 
							[.0111100101 ]
							[.0111110101 ]
						]
					]
					[.01111101 
						[.011101101 
							[.0111001101 ]
							[.\sout{0111011101} ]
						]
						[.011111101 
							[.0111101101 ]
							[.0111111101 ]
						]
					]
				]
			]
		]
\end{tikzpicture}

\end{tabular}

\caption[Example word tree for Theorem \ref{P3} with $q=2, \ell = 2$.]{Example word tree for Theorem \ref{P3} with $q=2$, $L=\text{`01'}$, $\ell = |L| = 2$. The tree from LLLL counted by $d_n$ is boxed.} \label{figure:T22}

\end{threeparttable}
\end{figure}

%% file: AppendixT23.tex
\begin{figure}[ht]
\centering
\begin{threeparttable}

\begin{tabular}{l}

\begin{tikzpicture}[grow'=right,level distance = 59pt,sibling distance=-5pt,every node/.style={scale=.75}]
	\Tree [.{$b$_7^3 = 2} [.{$b$_8^3 = 4} [.{$b$_9^3 = 8} [.{$b$_{10}^3 = 16} [.{$b$_{11}^3 = 30} {$b$_{12}^3 = 63} ]]]]];
\end{tikzpicture}
\\
\begin{tikzpicture}[grow'=right,level distance=60pt,sibling distance=-3pt,every node/.style={scale=.75}]
	\Tree [.1000100 
			[.10000100 
				[.100000100 
					[.1000000100 
						[.10000000100 
							[.100000000100 ]
							[.100000100100 ]
						]
						[.10000100100 
							[.100001000100 ]
							[.100001100100 ]
						]
					]
					[.1000010100 
						[.10000010100 
							[.100000010100 ]
							[.100000110100 ]
						]
						[.10000110100 
							[.100001010100 ]
							[.100001110100 ]
						]
					]
				]
				[.100010100 
					[.1000100100 
						[.\sout{10001000100} 
							[.100010000100 ]
							[.100010100100 ]
						]
						[.10001100100 
							[.100011000100 ]
							[.100011100100 ]
						]
					]
					[.1000110100 
						[.10001010100 
							[.100010010100 ]
							[.100010110100 ]
						]
						[.10001110100 
							[.100011010100 ]
							[.100011110100 ]
						]
					]
				]
			]
			[.10001100 
				[.100001100 
					[.1000001100 
						[.10000001100 
							[.100000001100 ]
							[.100000101100 ]
						]
						[.10000101100 
							[.100001001100 ]
							[.100001101100 ]
						]
					]
					[.1000011100 
						[.10000011100 
							[.100000011100 ]
							[.100000111100 ]
						]
						[.10000111100 
							[.100001011100 ]
							[.100001111100 ]
						]
					]
				]
				[.100011100 
					[.1000101100 
						[.10001001100 
							[.100010001100 ]
							[.100010101100 ]
						]
						[.10001101100 
							[.100011001100 ]
							[.100011101100 ]
						]
					]
					[.1000111100 
						[.10001011100 
							[.100010011100 ]
							[.100010111100 ]
						]
						[.10001111100 
							[.100011011100 ]
							[.100011111100 ]
						]
					]
				]
			]
		]
\end{tikzpicture}
\\
\begin{tikzpicture}[grow'=right,level distance=60pt,sibling distance=-3pt,every node/.style={scale=.75}]
	\Tree [.1001100 
			[.10010100 
				[.100100100 
					[.1001000100 
						[.10010000100 
							[.100100000100 ]
							[.\sout{100100100100} ]
						]
						[.10010100100 
							[.100101000100 ]
							[.100101100100 ]
						]
					]
					[.1001010100 
						[.10010010100 
							[.100100010100 ]
							[.100100110100 ]
						]
						[.10010110100 
							[.100101010100 ]
							[.100101110100 ]
						]
					]
				]
				[.100110100 
					[.1001100100 
						[.10001000100
							[.100110000100 ]
							[.100110100100 ]
						]
						[.10011100100 
							[.100111000100 ]
							[.100111100100 ]
						]
					]
					[.1001110100 
						[.10011010100 
							[.100110010100 ]
							[.100110110100 ]
						]
						[.10011110100 
							[.100111010100 ]
							[.100111110100 ]
						]
					]
				]
			]
			[.10011100 
				[.100101100 
					[.1001001100 
						[.10010001100 
							[.100100001100 ]
							[.100100101100 ]
						]
						[.10010101100 
							[.100101001100 ]
							[.100101101100 ]
						]
					]
					[.1001011100 
						[.10010011100 
							[.100100011100 ]
							[.100100111100 ]
						]
						[.10010111100 
							[.100101011100 ]
							[.100101111100 ]
						]
					]
				]
				[.100111100 
					[.1001101100 
						[.\sout{10011001100} 
							[.100110001100 ]
							[.100110101100 ]
						]
						[.10011101100 
							[.100111001100 ]
							[.100111101100 ]
						]
					]
					[.1001111100 
						[.10011011100 
							[.100110011100 ]
							[.100110111100 ]
						]
						[.10011111100 
							[.100111011100 ]
							[.100111111100 ]
						]
					]
				]
			]
		];

	\draw (9.55,4.29) rectangle (12,4.64);
	\draw (12,4.47) node[left] {$d_n^3$};
\end{tikzpicture}

\end{tabular}

\caption[Example word tree for Theorem \ref{P3} with $q=2, \ell = 3$.]{Example word tree for Theorem \ref{P3} with $q=2$, $L=\text{`100'}$, $\ell = |L| = 3$. The tree from LLLL counted by $d_n$ is boxed.} \label{figure:T23}

\end{threeparttable}
\end{figure}

%% file: AppendixT31.tex
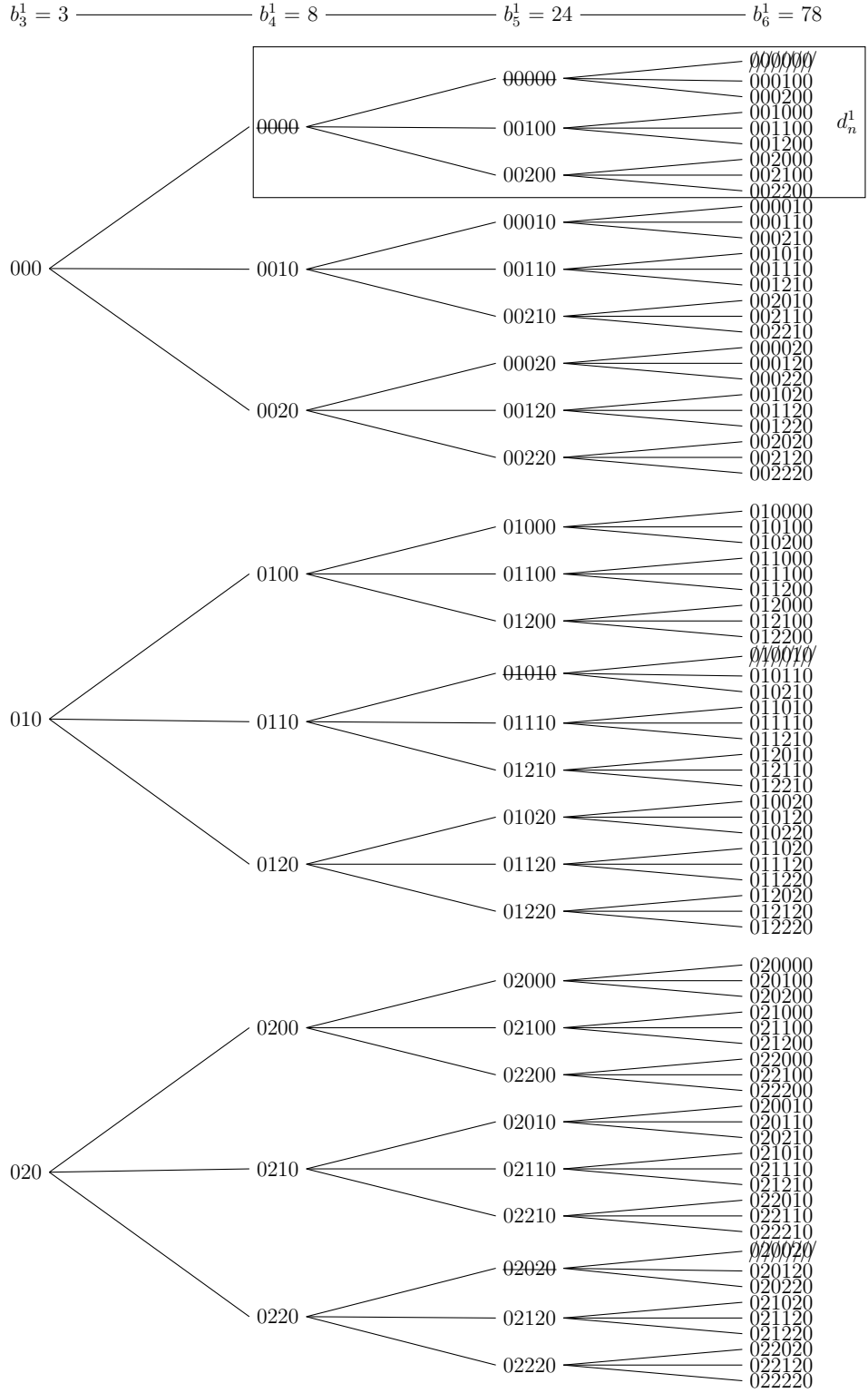
\begin{figure}[ht]
\centering
\begin{threeparttable}

\begin{tabular}{l}

\begin{tikzpicture}[grow'=right,level distance = 104pt,sibling distance=-5pt,every node/.style={scale=.75}]
	\Tree [.{$b$_3^{1} = 3} [.{$b$_4^{1} = 8} [.{$b$_5^{1} = 24} {$b$_6^1 = 78} ] ] ];
\end{tikzpicture}
\\
\begin{tikzpicture}[grow'=right,level distance = 105pt,sibling distance=-5pt,every node/.style={scale=.75}]
	\Tree [.000 
			[.\sout{0000} 
				[.\sout{00000} 
					[.\xout{000000} ]
					[.000100 ]
					[.000200 ]
				]
				[.00100 
					[.001000 ]
					[.001100 ]
					[.001200 ]
				]
				[.00200 
					[.002000 ]
					[.002100 ]
					[.002200 ]
				]
			]
			[.0010 
				[.00010 
					[.000010 ]
					[.000110 ]
					[.000210 ]
				]
				[.00110 
					[.001010 ]
					[.001110 ]
					[.001210 ]
				]
				[.00210 
					[.002010 ]
					[.002110 ]
					[.002210 ]
				]
			]
			[.0020 
				[.00020 
					[.000020 ]
					[.000120 ]
					[.000220 ]
				]
				[.00120 
					[.001020 ]
					[.001120 ]
					[.001220 ]
				]
				[.00220 
					[.002020 ]
					[.002120 ]
					[.002220 ]
				]
			]
		];

	\draw (3.33,1.14) rectangle (12.3,3.35);
	\draw (12.3,2.25) node[left]{$d_n^1$};
\end{tikzpicture}
\\
\begin{tikzpicture}[grow'=right,level distance = 105pt,sibling distance=-5pt,every node/.style={scale=.75}]
	\Tree [.010 
			[.0100 
				[.01000 
					[.010000 ]
					[.010100 ]
					[.010200 ]
				]
				[.01100 
					[.011000 ]
					[.011100 ]
					[.011200 ]
				]
				[.01200 
					[.012000 ]
					[.012100 ]
					[.012200 ]
				]
			]
			[.0110 
				[.\sout{01010} 
					[.\xout{010010} ]
					[.010110 ]
					[.010210 ]
				]
				[.01110 
					[.011010 ]
					[.011110 ]
					[.011210 ]
				]
				[.01210 
					[.012010 ]
					[.012110 ]
					[.012210 ]
				]
			]
			[.0120 
				[.01020 
					[.010020 ]
					[.010120 ]
					[.010220 ]
				]
				[.01120 
					[.011020 ]
					[.011120 ]
					[.011220 ]
				]
				[.01220 
					[.012020 ]
					[.012120 ]
					[.012220 ]
				]
			]
		]
\end{tikzpicture}
\\
\begin{tikzpicture}[grow'=right,level distance = 105pt,sibling distance=-5pt,every node/.style={scale=.75}]
	\Tree [.020 
			[.0200 
				[.02000 
					[.020000 ]
					[.020100 ]
					[.020200 ]
				]
				[.02100 
					[.021000 ]
					[.021100 ]
					[.021200 ]
				]
				[.02200 
					[.022000 ]
					[.022100 ]
					[.022200 ]
				]
			]
			[.0210 
				[.02010 
					[.020010 ]
					[.020110 ]
					[.020210 ]
				]
				[.02110 
					[.021010 ]
					[.021110 ]
					[.021210 ]
				]
				[.02210 
					[.022010 ]
					[.022110 ]
					[.022210 ]
				]
			]
			[.0220 
				[.\sout{02020} 
					[.\xout{020020} ]
					[.020120 ]
					[.020220 ]
				]
				[.02120 
					[.021020 ]
					[.021120 ]
					[.021220 ]
				]
				[.02220 
					[.022020 ]
					[.022120 ]
					[.022220 ]
				]
			]
		]
\end{tikzpicture}

\end{tabular}

\caption[Example word tree for Theorem \ref{P3} with $q=3, \ell = 1$.]{Example word tree for Theorem \ref{P3} with $q=3$, $L=\text{`0'}$, $\ell = |L| = 1$. The tree from LLLL counted by $d_n$ is boxed.} \label{figure:T31}

\end{threeparttable}
\end{figure}

%

%% file: AppendixN.tex
\chapter{Notation Index} \label{NOTATION}

Generally, majuscule Greek letters are used for alphabets (especially $\Gamma,\Sigma$). Minuscule Greek $\varepsilon$ (``var-epsilon'') represents the empty word, whereas $\epsilon$ is used in proofs for arbitrarily-small positive real values; other minuscule Greek letters are used for monoid homomorphisms (especially $\phi,\psi$).

Frequently, minuscule Roman letters are used for letters in words (especially $a$, $b$, $c$, $d$, $t$, $u$, $v$, $w$, $x$, $y$, and $z$), variables (especially $a$, $b$, $c$, $d$, $i$, $j$, $k$, $\ell$, $m$, $n$, $p$, $q$, $r$, $t$, $u$, and $v$), or functions (especially $f$ and $g$); majuscule Roman letters are used for words (especially $S$, $T$, $U$, $V$, $W$, $X$, $Y$, and $Z$), variables (especially $M$ and $N$), or functions (especially $F$, $G$, and $H$). 
Natural numbers are also used for letters.

For notation established within a numbered definition in the text, the definition number is given in Table~\ref{table:notation} below.

\begin{table}[ht]
	\centering
	\begin{threeparttable}
	
		\caption{Notation used.} \label{table:notation}

\begin{tabular}{l | l | l}
	Notation & Meaning & Defined \\ \hline
	$\Z$ & The set of integers: $\{\ldots, -2, -1, 0, 1, \ldots \}$. \\
	$\Z^+$ & The set of positive integers: $\{1, 2, 3, \ldots \}$. \\
	$\N$ & The set of natural numbers: $\{0, 1, 2, 3, \ldots \}$. \\
	$f(n) \sim g(n)$ & $\lim_{n \rightarrow \infty} \frac{f(n)}{g(n)} = 1$. \\
	$f(n) = O(g(n))$ & There exists $c > 0$ so that $f(n) \leq cg(n)$. \\
	$f(n) \ll g(n)$ & $f(n) = O(g(n))$. \\
	$f(n) = o(g(n))$ & $\lim_{n \rightarrow \infty} \frac{f(n)}{g(n)} = 0$. \\
	$\Sigma^*$ & The set of finite $\Sigma$-words. & \ref{defn:word} \\
	$\Sigma^n$ & The set of length-$n$ $\Sigma$-words. & \ref{defn:word} \\
	$\varepsilon$ & The empty word. & \ref{defn:word}\\
	$[n]$ & The set $\{1, 2, \ldots, n\}$. & \\
	$w \in W$ & Letter $w$ occurs in word $W$. & \ref{defn:letter} \\
	$w^n$ & The word formed from $n$ copies of the letter $w$. & \ref{defn:letter}\\
	$|W|$ & The length of word $W$. & \ref{defn:letter} \\
	$\L(W)$ & The set of letters that occur in word $W$. & \ref{defn:letter} \\
	$||W||$ & The number of letter recurrences in word $W$. &  \ref{defn:letter} \\
	$W[i:j]$ & The factor of $W$ stretching from letter $i+1$ to $j$. & \ref{defn:factor} \\
	$V \leq W$ & Word $V$ is a factor of word $W$. & \ref{defn:factor} \\
	$V \preceq W$ & $W$ encounters $V$. & \ref{defn:instance}\\
	$Z_n$ & The $n$-th Zimin word. & \ref{defn:Zimin}\\
	$\f(n,q)$ & Least $M$ such that every word in $[q]^M$ encounters $Z_n$. & \ref{defn:f} \\
	${}^ba$ & Towering exponential $a^{\cdot^{\cdot^a}}$ with $b$ occurrences of $a$. & \\
	$\Inst_n(V,\Sigma)$ & The set of $W$-instances in $\Sigma^n$. & \ref{defn:I} \\
	$\II_n(V,q)$ & The proportion of words in $\Sigma^n$ that are $V$-instances & \ref{defn:I} \\
	$\EE( \cdot )$ & The expected value of a given random variable. & \\
	$\PP( \cdot )$ & The probability of a given event. & \\
	$\m(n,q)$ & The number of minimal $Z_n$-instances in $[q]^*$. & \ref{defn:min} \\
	$\d(V,W)$ & The factor density of word $V$ in word $W$. & \ref{defn:density} \\
	$\delta(V,W)$ & The (instance) density of word $V$ in word $W$. & \ref{defn:density}\\
	$\underline{\delta}(V,q)$ & The liminf density of word $V$ over $[q]$. & \ref{defn:density}\\
	$\delta_n(V,q)$ & The expected density of word $V$ in $W \in [q]^n$. & \ref{defn:density2}\\
	$\delta(V,q)$ & $\lim_{n \rightarrow \infty} \delta_n(V,q)$. & \ref{defn:density2}\\
	$\II(V,q)$ & $\lim_{n \rightarrow \infty} \II_n(V,q)$. & \ref{defn:density2} \\
	$\hom(V,W)$ & The number of $V$-encounters in $W$. & \ref{defn:hom} \\
	$\hom_n(V,q)$ & The expected number of $V$-encounters in $W \in [q]^n$. & \ref{defn:hom}\\
	$\delta_{sur}(V,W)$ & 1 if $W$ is a $V$-instance; 0 otherwise. & \ref{defn:sur} \\
\end{tabular}

	\end{threeparttable}
\end{table}